\documentclass[a4paper, DIV=12, 11pt]{scrartcl}
\usepackage[latin1]{inputenc}
\usepackage{amsmath}
\usepackage{amsfonts}
\usepackage{amssymb}
\usepackage{amsthm}
\usepackage{graphicx}
\usepackage{thmtools}
\usepackage{bbm}
\usepackage{hyperref}
\usepackage[bottom, marginal]{footmisc}
\usepackage{todonotes}

\usepackage{tikz}
\usetikzlibrary{matrix}

\declaretheoremstyle[headfont=\normalfont]{normalhead}
\newtheorem{lemma}{Lemma}[section]
\newtheorem{theorem}[lemma]{Theorem}
\newtheorem{proposition}[lemma]{Proposition}
\newtheorem{corollary}[lemma]{Corollary}
\newtheorem{definition}[lemma]{Definition}
\newtheorem{remark}[lemma]{Remark}

\newcounter{mt}

\newtheorem{maintheorem}[mt]{Theorem}

\newtheorem{maincorollary}[mt]{Corollary}

\newcommand{\R}{\mathbb{R}}
\newcommand{\C}{\mathbb{C}}

\newcommand{\U}{\mathrm{U}}

\DeclareMathOperator{\sign}{sign}
\DeclareMathOperator{\Val}{Val}
\DeclareMathOperator{\VConv}{VConv}
\DeclareMathOperator{\Conv}{Conv}
\DeclareMathOperator{\vol}{vol}

\DeclareMathOperator{\supp}{supp}

\DeclareMathOperator{\GL}{GL}

\DeclareMathOperator{\Gr}{\mathrm{Gr}}

\DeclareMathOperator{\nc}{\mathrm{nc}}

\DeclareMathOperator{\SO}{\mathrm{SO}}

\setkomafont{title}{\normalfont\Large}

\author{Jonas Knoerr}
\title{A geometric decomposition for unitarily invariant valuations on convex functions}
\date{}

\newcommand{\Addresses}{{
		\bigskip
		\footnotesize
		
		Jonas Knoerr, \textsc{Institute of Discrete Mathematics and Geometry, TU Wien, Wiedner Hauptstrasse 8-10, 1040 Wien, Austria}\par\nopagebreak
		\textit{E-mail address}: \texttt{jonas.knoerr@tuwien.ac.at}
		
		\medskip
	}}
	
\makeatletter
\def\blfootnote{\xdef\@thefnmark{}\@footnotetext}
\makeatother

\makeindex
\begin{document}
\maketitle
\begin{abstract}
	Valuations on the space of finite-valued convex functions on $\C^n$ that are continuous, dually epi-translation invariant, as well as $\U(n)$-invariant are completely classified. It is shown that the space of these valuations decomposes into a direct sum of subspaces defined in terms of vanishing properties with respect to restrictions to a finite family of special subspaces of $\C^n$, mirroring the behavior of the Hermitian intrinsic volumes introduced by Bernig and Fu. Unique representations of these valuations in terms of principal value integrals involving two families of Monge-Amp\`ere-type operators are established. 
\end{abstract}

\section{Introduction}

Given a family $\mathcal{S}$ of sets, a map $\mu:S\rightarrow (\mathcal{A},+)$ with values in an abelian semi-group is called a valuation if
\begin{align*}
	\mu(K)+\mu(L)=\mu(K\cup L)+\mu(K\cap L)
\end{align*}
for all $K,L\in\mathcal{S}$ such that $K\cup L, K\cap L\in \mathcal{S}$. This notion goes back to Dehn's solution of Hilbert's third problem and has since become a core concept in integral and convex geometry with many applications in other areas of mathematics. For applications in integral geometry, the focus is usually on valuations on the space $\mathcal{K}(\R^n)$ of convex bodies, that is, the set of all non-empty, convex and compact subsets of $\R^n$ equipped with the Hausdorff metric. Many classical formulas in integral geometry follow easily from the follow classification result.
\begin{theorem}[Hadwiger \cite{HadwigerVorlesungenuberInhalt1957}]
	Let $\mu:\mathcal{K}(\R^n)\rightarrow\R$ be a rigid motion invariant and continuous valuation. Then $\mu$ is a linear combination of the intrinsic volumes.
\end{theorem}
We refer to Section \ref{section:valuationsBodies} and \cite{SchneiderConvexbodiesBrunn2014} for the definition of the intrinsic volumes and their role in integral geometry. Let us, however, remark that Hadwiger's result implies in particular that the space of these valuations is finite dimensional.\\
Since the late 1990s, the theory of valuations on convex bodies has developed rapidly, with a variety of breakthrough results and new applications \cite{AleskerDescriptiontranslationinvariant2001,AleskerHardLefschetztheorem2003,Aleskermultiplicativestructurecontinuous2004,AleskerFaifmanConvexvaluationsinvariant2014,BernigHadwigertypetheorem2009,BernigBroeckerValuationsmanifoldsRumin2007,BernigEtAlCurvaturemeasurespseudo2022,BernigEtAlHardLefschetztheorem2024,FaifmanHofstaetterConvexvaluationsWhitney2023,FreyerEtAlUnimodularValuationsEhrhart2024,FuStructureunitaryvaluation2006,KotrbatyWannererharmonicanalysistranslation2023,LudwigReitznerclassification$SLn$invariant2010,Wannerermoduleunitarilyinvariant2014}. Many of these results rely on the seminal work by Alesker, who in particular showed that Hadwiger-type finiteness results hold for a variety of other subgroups. Let $\Val(\C^n)^{\U(n)}$ denote the space of all continuous, translation invariant and $\U(n)$-invariant valuations. This space is finite dimensional, as shown by Alesker \cite{AleskerDescriptiontranslationinvariant2001}, who also introduced the first bases for this space \cite{AleskerHardLefschetztheorem2003}. For our purposes, the following basis is more suitable.
\begin{theorem}[Bernig--Fu \cite{BernigFuHermitianintegralgeometry2011} Theorem 3.2]
	\label{theorem:BernigFuHermitianIntrinsicVolumes}
	For $0\le k\le 2n$, $\max(0,k-n)\le q\le \lfloor\frac{k}{2}\rfloor$ there exists a unique $k$-homogeneous valuation $\mu_{k,q}\in \Val(\C^n)^{\U(n)}$ such that
	\begin{align*}
		\mu_{k,q}|_{E_{k,p}}=\delta_{pq}\vol_k,
	\end{align*}
	where $\vol_k$ denotes the $k$-dimensional Lebesgue measure on $E_{k,p}$. Moreover, these valuations form a basis for $\Val(\C^n)^{\U(n)}$.
\end{theorem}
Due to their resemblance to the intrinsic volumes, the functionals $\mu_{k,q}$ are also called hermitian intrinsic volumes.\\

In recent years, there has been a considerable effort to extend classical results and notions from convex geometry to a functional setting \cite{ArtsteinAvidanEtAlFunctionalaffineisoperimetry2012,ArtsteinAvidanMilmanconceptdualityconvex2009,BaryshnikovEtAlHadwigersTheoremdefinable2013,BobkovEtAlQuermassintegralsquasiconcave2014,ColesantiEtAlTranslationinvariantvaluations2018,HofstaetterSchusterBlaschkeSantaloinequalities2023,KolesnikovWernerBlaschkeSantaloinequality2022,KoneValuationsOrliczspaces2014,LiEtAlAffineinvariantmaps2022,MilmanRotemMixedintegralsrelated2013}. One of the most fruitful approaches builds on the valuation theoretic classification of geometric functionals on convex bodies and tries to obtain similar characterization results for valuations on spaces of functions. Here a functional $\mu$ defined on a space $\mathcal{F}$ of (extended) real-valued functions is called a valuation if
\begin{align*}
	\mu(f)+\mu(h)=\mu(f\vee h)+\mu(f\wedge h)
\end{align*}
for all $f,h\in \mathcal{F}$ such that the pointwise maximum $f\vee h$ and minimum $f\wedge h$ also belong to $\mathcal{F}$. Depending on the chosen function space, the topology, and the invariance assumptions, various well-known functionals can be characterized as the unique valuations with the given properties, see for example \cite{ColesantiEtAlMinkowskivaluationsconvex2017,HofstaetterKnoerrEquivariantendomorphismsconvex2023,HofstaetterKnoerrEquivariantValuationsConvex2024,KnoerrMongeAmpereoperators2024,LiMaLaplacetransformsvaluations2017,LudwigFisherinformationmatrix2011,LudwigValuationsSobolevspaces2012,MussnigVolumepolarvolume2019,Mussnig$SLn$InvariantValuations2021,TradaceteVillanuevaRadialcontinuousvaluations2017,TradaceteVillanuevaContinuityrepresentationvaluations2018,TradaceteVillanuevaValuationsBanachlattices2020,VillanuevaRadialcontinuousrotation2016}.\\

One of the most active areas of research in this part of valuation theory are valuations on spaces of convex functions \cite{AleskerValuationsconvexfunctions2019,ColesantiEtAlHadwigertheoremconvex2020,ColesantiEtAlHadwigertheoremconvex,ColesantiEtAlHadwigertheoremconvex2022,ColesantiEtAlHadwigertheoremconvex2023,HugEtAlAdditivekinematicformulas2024,KnoerrSingularvaluationsHadwiger2022,KnoerrUlivellivaluationsconvexbodies2024}. This applies in particular to the space
\begin{align*}
	\Conv(V,\R):=\{f:V\rightarrow\R:f ~\text{is convex}\}
\end{align*}
of all finite convex functions on a finite dimensional real vector space $V$. This space is naturally  equipped with the topology induced by epi-convergence (which coincides with the topology induced by pointwise convergence, compare Section \ref{section:PreliminariesConvex}). We denote by $\VConv(V)$ the space of all continuous valuations $\mu:\Conv(V,\R)\rightarrow\R$ that are \emph{dually epi-translation invariant}, that is, that satisfy 
\begin{align*}
	\mu(f+\lambda)=\mu(f)\quad\text{for all}~\lambda:V\rightarrow \R~\text{affine},~f\in\Conv(V,\R).
\end{align*}
This notion is intimately tied to translation invariance: Functionals with this property are invariant under translations of the epi-graph of the convex conjugate of the given function $f\in\Conv(V,\R)$ considered as a subset of $V^*\times \R$, compare \cite{Knoerrsupportduallyepi2021,KnoerrSmoothvaluationsconvex2024,KnoerrUlivellivaluationsconvexbodies2024}. Consequently, many results for translation invariant valuations on convex bodies admit a functional version for elements of $\VConv(V)$. For example, Colesanti, Ludwig, and Mussnig \cite{ColesantiEtAlhomogeneousdecompositiontheorem2020} showed that there exists a homogeneous decomposition mirroring McMullen's decomposition \cite{McMullenValuationsEulertype1977},
\begin{align*}
	\VConv(V)=\bigoplus_{k=0}^{\dim V}\VConv_k(V),
\end{align*}
where $\mu\in\VConv_k(V)$ if and only if $\mu(tf)=t^k\mu(f)$ for $t\ge 0$ and $f\in\Conv(V,\R)$. In \cite{ColesantiEtAlHadwigertheoremconvex2020} they also obtained a Hadwiger-type characterization of the space of all $\SO(n)$-invariant valuations in $\VConv(\R^n)$, and the corresponding functionals are called \emph{functional intrinsic volumes}. Their construction involves integrals of certain rotation invariant densities with respect to the so-called \emph{Hessian measures} $\Phi_k$, $0\le k\le n$, which are closely related to the real Monge-Amp\`ere operator. These may be characterized as the unique continuous valuations on $\Conv(\R^n,\R)$ with values in the space $\mathcal{M}(\R^n)$ of signed Radon measures on $\R^n$, considered as the continuous dual of $C_c(\R^n)$, satisfying
\begin{align*}
	d\Phi_k(f)[x]=[D^2f(x)]_kdx \quad\text{for all}~f\in\Conv(\R^n,\R)\cap C^2(\R^n).
\end{align*}
Here, $[A]_k$ denotes the $k$th elementary symmetric function of the eigenvalues of an $(n\times n)$-matrix $A$. We refer to \cite{KnoerrMongeAmpereoperators2024} for a valuation theoretic characterization of the Hessian measures, to \cite{TrudingerWangHessianmeasures.I1997,TrudingerWangHessianmeasures.II1999} for their role in the theory of Hessian equations, and to \cite{AleskerValuationsconvexsets2005,AleskerValuationsconvexfunctions2019} for the role of Monge-Amp\`ere-type operators in the construction of valuations on convex bodies and functions.\\

Let $C_b((0,\infty))$ denote the space of all continuous functions $\zeta:(0,\infty)\rightarrow \R$ whose support is bounded from above and consider for $a>0$ the subspace
\begin{align*}
	D^a:=\left\{\zeta\in C_b((0,\infty)):\lim_{t\rightarrow0}t^a\zeta(t)=0,~\lim\limits_{t\rightarrow0}\int_0^\infty \zeta(r)r^{a-1}dr~\text{exists and is finite}\right\}.
\end{align*}
Let $\VConv(\R^n)^{\SO(n)}$ denote the subspace of all $\mu\in\VConv(\R^n)$ that are $\SO(n)$-invariant, that is, that satisfy
\begin{align*}
	\mu(f\circ g)=\mu(f)\quad\text{for all}~g\in \SO(n),~f\in\Conv(\R^n,\R).
\end{align*}
\begin{theorem}[Colesanti--Ludwig--Mussnig \cite{ColesantiEtAlHadwigertheoremconvex2020}]
	\label{theorem:HadwigerVConv}
	Let $1\le k\le n-1$. For every $\mu\in\VConv_k(\R^n)^ {\SO(n)}$ there exists a unique $\zeta\in D^{n-k}$ such that
	\begin{align*}
		\mu(f)=\int_{\R^n}\zeta(|x|)d\Phi_k(f)\quad \text{for all}~f\in\Conv(\R^n,\R)\cap C^2(\R^n).
	\end{align*}
	 Conversely, the right hand side of this equation extends uniquely by continuity to a continuous valuation in $\VConv_k(\R^n)^{\SO(n)}$ for  $\zeta\in D^{n-k}$. 
\end{theorem}
Let us remark that we have omitted the cases $k=0$ (which corresponds to constant valuations) and $k=n$. In the latter case, any valuation in $\VConv_n(\R^n)$ admits a unique integral representation with respect to the real Monge-Amp\`ere operator $\Phi_n$, compare \cite{ColesantiEtAlhomogeneousdecompositiontheorem2020,KnoerrUlivellivaluationsconvexbodies2024}. In contrast, the functions $x\mapsto \zeta(|x|)$ are in general not integrable with respect to $\Phi_k(f)$, and it is a highly non-trivial result that the right hand side of the equation above extends by continuity to arbitrary convex functions. Two alternative approaches to these functionals involving mixed Monge-Amp\`ere operators and Cauchy-Kubota-type formulas are  discussed in  \cite{ColesantiEtAlHadwigertheoremconvex,ColesantiEtAlHadwigertheoremconvex2022}, and we refer to \cite{KnoerrSingularvaluationsHadwiger2022} for an interpretation of the construction in terms of certain principal value integrals. 

\subsection{Main results}

This is the second of two articles on the characterization of the space $\VConv(\C^n)^{\U(n)}$ of all $\U(n)$-invariant valuations in $\VConv(\C^n)$. In the first part \cite{KnoerrUnitarilyinvariantvaluations2021}, general properties of these functionals were investigated, including a description of the dense subspace of smooth valuations (see Section \ref{section:VConv} for the definition), which will be the foundation for this article. Here, the space $\VConv(\C^n)$ and all of its subspaces are equipped with the topology of uniform convergence on compact subsets in $\Conv(\C^n,\R)$, compare Section \ref{section:VConv}.

Our main results establish unique representation formulas for elements of $\VConv(\C^n)^{\U(n)}$ in terms of two families of Monge-Amp\`ere-type operators, where the number of families needed in the construction grows with the square of the dimension. This implies in particular that in contrast to the case of $\SO(n)$-invariant valuations treated in Theorem \ref{theorem:HadwigerVConv}, we are not only dealing with one family of valuations for each degree of homogeneity. This  presents a challenge to the known approaches to Theorem \ref{theorem:HadwigerVConv}: They all rely on a version of the template method to reconstruct a suitable candidate for the function $\zeta\in D^{n-k}$ from a given valuation $\mu\in\VConv_k(\R^n)^{\SO(n)}$ - that is, evaluating the valuation in a suitable $1$-parameter family of convex functions and inverting some associated integral transforms.\\

At the present time, we are lacking suitable tools to circumvent the template method completely. In order to reduce the computational complexity to a manageable level, we first establish a geometric direct sum decomposition of $\VConv_k(\C^n)^{\U(n)}$ into closed subspaces mirroring the vanishing behaviour of the hermitian intrinsic volumes. To make this precise, note that for a linear map $T:V\rightarrow W$ between finite dimensional real vector spaces, we can define the push-forward $T_*:\VConv(V)\rightarrow\VConv(W)$ by
\begin{align*}
	T_*\mu(f)=\mu(T^*f)\quad\text{for}~\text{for}~\mu\in\VConv(V),~f\in\Conv(W,\R).
\end{align*}
 If $E\subset \C^n$ is a subspace and $T=\pi_E:\C^n\rightarrow E$ is the orthogonal projection, we will also call this the \emph{restriction} of $\mu$ to $E$. The following spaces were introduced in \cite{KnoerrUnitarilyinvariantvaluations2021} for $\max(0,k-n)\le q\le\lfloor\frac{k}{2}\rfloor$:
\begin{align*}
	\VConv_{k,q}(\C^n)^{\U(n)}:=\left\{\mu\in\VConv_k(\C^n)^{\U(n)}:\pi_{E_{k,p}*}\mu=0~\text{for all}~p\ne q\right\}
\end{align*}
It was shown in \cite{KnoerrUnitarilyinvariantvaluations2021} that the sum of these spaces is direct. Our first main result shows that this direct sum provides a decomposition of $\VConv_k(\C^n)^{\U(n)}$.
	\begin{maintheorem}
		\label{maintheorem:decomposition}
		\begin{align*}
			\VConv_{k}(\C^n)^{\U(n)}=\bigoplus_{q=\max(0,k-n)}^{\lfloor\frac{k}{2}\rfloor}\VConv_{k,q}(\C^n)^{\U(n)}
		\end{align*}
	\end{maintheorem}
		Theorem \ref{maintheorem:decomposition} reduces the classification problem to the description of the different components $\VConv_{k,q}(\C^n)^{\U(n)}$. In order to describe the relevant valuations, we use the Monge-Amp\`ere-type operators
	\begin{align*}
		&\Theta^{n}_{k,q},&&\max(0,k-n)\le q\le\left\lfloor\frac{k}{2}\right\rfloor,\\
		&\Upsilon^{n}_{k,q},&&\max(1,k-n)\le q\le\left\lfloor\frac{k-1}{2}\right\rfloor
	\end{align*}
	introduced in \cite{KnoerrUnitarilyinvariantvaluations2021}, which assign to any element of $\Conv(\C^n,\R)$ a signed Radon measure on $\C^n$. As in the case of $\SO(n)$-invariant valuations, the classification requires valuations constructed from singular functions. In addition to the spaces $D^a$ introduced in the previous section, we consider for $a\in\mathbb{N}$ the space
	\begin{align*}
		\tilde{D}^{a+2}:=\left\{\tilde{\zeta}\in C_b((0,\infty)): \lim\limits_{t\rightarrow0}t^{a+2}\tilde{\zeta}(t)=0\right\}.
	\end{align*}
	For $R>0$ we let $D^a_R$ and $\tilde{D}_R^{a+2}$ denote the subspaces of functions with support contained in $(0,R]$, and we equip these spaces with the norms
	\begin{align*}
	&\|\zeta\|_{D^a}:=\sup_{t>0}t^a|\zeta(t)|+\sup_{t>0}\left|\int_t^\infty \zeta(r)r^{a-1}dr\right|,&& \text{for}~\zeta\in D^a_R,\\
		&\|\tilde{\zeta}\|_{\tilde{D}^{a+2}}:=\sup_{t>0}t^{a+2}|\zeta(t)|,  &&\text{for}~\tilde{\zeta}\in \tilde{D}^{a+2}_R.
	\end{align*}
	These spaces contain  $C_{c,R}([0,\infty))=\{\phi\in C_c([0,\infty)):\supp\phi\subset[0,R]\}$ as a dense subspace (compare Section \ref{section:DaTildeDa}). Following the approach in \cite{KnoerrSingularvaluationsHadwiger2022}, we obtain the following extension result.
	\begin{maintheorem}
		\label{maintheorem:ExistenceExtensionIntegration}
		There exist unique continuous maps
		\begin{align*}
			T^n_{k,q}:D^{2n-k}_R&\rightarrow\VConv_{k,q}(\C^n)^{\U(n)},\\
			Y^n_{k,q}:\tilde{D}^{2n-k+2}_R&\rightarrow\VConv_{k,q}(\C^n)^{\U(n)}
		\end{align*}
		such that for $\zeta\in C_c([0,\infty))$ with $\supp\zeta\subset[0,R]$ and $f\in\Conv(\C^n,\R)$,
		\begin{align*}
			&T^n_{k,q}(\zeta)[f]=\int_{\C^n}\zeta(|z|)d\Theta_{k,q}(f),
			&Y^n_{k,q}(\zeta)[f]=\int_{\C^n}\zeta(|z|)d\Upsilon^n_{k,q}(f).
		\end{align*}
	\end{maintheorem}
	Let us remark that for $\zeta\in D^{2n-k}$, the function $z\mapsto \zeta(|z|)$ is in general not integrable with respect to the measures $\Theta^n_{k,q}(f)$, compare Corollary \ref{corollary:ZetaNotIntegrableWRTTheta}, while we do not know if this is the case for $\Upsilon^n_{k,q}$ and functions in $\tilde{D}^{2n-k+2}$. However, in both cases the following principal value representation holds.
	\begin{maintheorem}
		\label{maintheorem:PrincipalValueRepresentation}
		For $\zeta\in D^{2n-k}$, $\tilde{\zeta}\in\tilde{D}^{2n-k+2}$, and every $f\in\Conv(\C^n,\R)$
		\begin{align*}
			T^n_{k,q}(\zeta)[f]=&\lim\limits_{\epsilon\rightarrow0}\int_{\C^n\setminus B_\epsilon(0)}\zeta(|z|)d\Theta_{k,q}(f),\\
			Y^n_{k,q}(\tilde{\zeta})[f]=&\lim\limits_{\epsilon\rightarrow0}\int_{\C^n\setminus B_\epsilon(0)}\tilde{\zeta}(|z|)d\Upsilon^n_{k,q}(f).
		\end{align*}
		Moreover, the convergence is uniform on compact subsets of $\Conv(\C^n,\R)$.
	\end{maintheorem}
	In the first case, we have precise information under which conditions on $\zeta$ the representation is given by a proper integral for every $f\in\Conv(\C^n,\R)$, while we only obtain a sufficient condition in the second case, compare Corollaries \ref{corollary:representationConv0}, \ref{corollary:IntegrabilityTheta}, and \ref{corollary:IntegrabilityUpsilon}.\\
	Using these families of valuations, we obtain the following complete classification of the components $\VConv_{k,q}(\C^n)^{\U(n)}$, where we have once again omitted the components $k=0$ and $k=2n$.
	\begin{maintheorem}
	\label{maintheorem:RepresentationVConvKQ}
	Let	$\mu\in\VConv_{k,q}(\C^n)^{\U(n)}$, $1\le k\le 2n-1$.
	\begin{itemize}
		\item If $k\le n$ and $q=0$, or if $k$ is even and $q=\frac{k}{2}$, then there exists a unique $\phi\in D^{2n-k}$ such that
		\begin{align*}
			\mu=T^n_{k,q}(\phi).
		\end{align*}
		\item If $\max(1,k-n)\le q\le \lfloor\frac{k-1}{2}\rfloor$, then there exist unique $\phi\in D^{2n-k}$, $\psi\in \tilde{D}^{2n-k+2}$, such that
		\begin{align*}
			\mu=T^n_{k,q}(\phi)+Y_{k,q}(\psi).
		\end{align*}
	\end{itemize} 
	\end{maintheorem}
	Combining Theorem \ref{maintheorem:decomposition} and Theorem \ref{maintheorem:RepresentationVConvKQ} we thus obtain the following representation of elements of $\VConv_{k}(\C^n)^{\U(n)}$.
	\begin{maincorollary}
		Let $1\le k\le 2n-1$. For any $\mu\in\VConv_{k}(\C^n)^{\U(n)}$ there exist unique $\phi_{q}\in D^{2n-k}$ for $\max(0,k-n)\le q\le \lfloor\frac{k}{2}\rfloor$ and $\psi_{q}\in \tilde{D}^{2n+2-k}$ for $\max(1,k-n)\le q\le \lfloor\frac{k-1}{2}\rfloor$ such that
		\begin{align*}
			\mu=\sum_{q=\max(0,k-n)}^{\lfloor\frac{k}{2}\rfloor}T^n_{k,q}(\phi_{q})+\sum_{q=\max(1,k-n)}^{\lfloor\frac{k-1}{2}\rfloor}Y_{k,q}(\psi_{q}).
		\end{align*}
	\end{maincorollary} 
	
\subsection{Plan of the article}
	The article can very broadly be divided into two different parts: the construction of the relevant valuations and investigation of their properties, which includes the proofs of Theorem \ref{maintheorem:ExistenceExtensionIntegration} and Theorem \ref{maintheorem:PrincipalValueRepresentation}, and the classification result consisting of Theorem \ref{maintheorem:decomposition} and Theorem \ref{maintheorem:RepresentationVConvKQ}. 
	
	The first part, while technical, involves some relatively straight forward constructions generalizing results from \cite{ColesantiEtAlHadwigertheoremconvex2020,KnoerrSingularvaluationsHadwiger2022,KnoerrMongeAmpereoperators2024}, whereas we feel that the second part warrants some additional comments.\\

	It was shown in \cite{KnoerrSmoothvaluationsconvex2024} that versions of Theorem \ref{maintheorem:decomposition} and Theorem \ref{maintheorem:RepresentationVConvKQ} hold for the dense subspace of smooth valuations. Since every valuation in $\VConv_k(\C^n)^{\U(n)}$ can be approximated by a sequence of smooth $\U(n)$-invariant valuations, one can thus hope to obtain the general case from these classifications. However, there are very simple reasons that prevent a direct approximation argument.\\
	
	For Theorem \ref{maintheorem:decomposition}, the main problem is obvious: Given a valuation in $\VConv_k(\C^n)^{\U(n)}$, we can approximate it by smooth $\U(n)$-invariant valuations, which were explicitly described in \cite{KnoerrUnitarilyinvariantvaluations2021} and which decompose according to the decomposition in Theorem \ref{maintheorem:decomposition}. However, we cannot directly infer that these components actually converge individually, since we do a priori not know if the projections onto these components are continuous. \\
	In addition, the approximation results established in \cite{KnoerrSmoothvaluationsconvex2024} are not directly applicable to the spaces $\VConv_{k,q}(\C^n)^{\U(n)}$. It is thus not a priori clear if elements of $\VConv_{k,q}(\C^n)^{\U(n)}$ can be approximated by smooth valuations belonging to the same space, which prevents us from proving Theorem \ref{maintheorem:RepresentationVConvKQ} directly (following for example the approach taken in \cite{KnoerrSingularvaluationsHadwiger2022}). \\
	
	In order to circumvent these problems, we will establish a slightly more refined version of Theorem \ref{maintheorem:decomposition} involving restrictions on the support of the valuations, which allows us to prove that the projections onto the different components are continuous. After this result is established, the classification of the different components in Theorem \ref{maintheorem:RepresentationVConvKQ} can be obtained by approximation from the classification of smooth valuations.\\
	Our proof of the refined version of Theorem \ref{maintheorem:decomposition} relies on induction on the dimension. The main idea centers around reconstructing the projections onto the different components inductively using the restriction of a given valuation to sums of functions defined on lower dimensional complex orthogonal subspaces. This involves evaluating these restricted valuations on suitable continuous one parameter families of convex functions and inverting the associated integral transforms. It turns out that the image of these transforms does not coincide with the space of all continuous functions in the parameters of the chosen families but only with a closed subspace, compare Corollary \ref{corollary:RegularityUVW}. We were not able to find a direct argument to explain why elements of $\VConv_k(\C^n)^{\U(n)}$ should generate functions belonging to this subspace. For smooth valuations this follows from the explicit description in \cite{KnoerrUnitarilyinvariantvaluations2021}, and in the general case we argue by approximation. In order for this to work, we have to verify that all of the constructions depend continuously on the given valuation. In particular, we have to check that the integral transforms are continuous with respect to the relevant norms.\\
	Having constructed suitable candidates for the projections onto the different components, we verify that these maps are indeed the projections by considering their behavior on the dense subspace of smooth valuations, where we know that the decomposition holds.\\
	This explicit reconstruction of the projections works for every component $\VConv_{k,q}(\C^n)^{\U(n)}$ except for $q=1$, see Remark \ref{remark:ProblematicComponent}. Since the corresponding projection is just the identity minus the remaining projections, this is sufficient to construct all projections.\\

	The article is structured as follows:\\
	In Section \ref{section:Preliminaries} we discuss a variety of background results needed throughout the article, including notions from geometric measure theory, and smooth valuations and curvature measures on convex bodies, and establish some preparatory results. This includes some general estimates for integrals involving smooth curvature measures, which seem to be new. We also review the relevant results on the topology of $\VConv(\R^n)$ and give slightly improved versions of the approximation results from \cite{KnoerrSmoothvaluationsconvex2024}.\\
	In Section \ref{section:IntegralTransforms} we introduce the integral transforms for the spaces $D^a$ and $\tilde{D}^{a+2}$ that will occur in our version of the template method. Since our proofs of the classification results heavily rely on approximation arguments, this section mostly centers around establishing suitable bounds for the operator norms of these integral transforms between the relevant classes of functions.\\
	Section \ref{section:MAOperators} is concerned with the properties of the Monge-Amp\`ere-type operators $\Theta^n_{k,q}$ and $\Upsilon^n_{k,q}$. We first extend some of the results from \cite{KnoerrMongeAmpereoperators2024} used in the construction of these measure-valued functionals and derive general bounds for integrals with respect to these measures in terms of the Hessian measures. We also consider the behaviour of $\Theta^n_{k,q}$ and $\Upsilon^n_{k,q}$ on two families of convex functions, which forms the basis for the template method.\\
	In Section \ref{section:SingularValuations} we prove Theorem \ref{maintheorem:ExistenceExtensionIntegration} and Theorem \ref{maintheorem:PrincipalValueRepresentation} by establishing suitable estimates for integrals of rotation invariant functions with respect to $\Theta^n_{k,q}$ and $\Upsilon^n_{k,q}$. We also discuss how the valuations obtained from $T^n_{k,q}$ and $Y^n_{k,q}$ are related to the hermitian intrinsic volumes.\\
	Finally, Section \ref{section:CharacterizationResults} contains the proofs of the characterization results in Theorem \ref{maintheorem:decomposition} and Theorem \ref{maintheorem:RepresentationVConvKQ}.

\section{Preliminaries}
	\label{section:Preliminaries}
	
	\subsection{Convex functions}
	\label{section:PreliminariesConvex}
	We refer to the monograph by Rockafellar and Wets \cite{RockafellarWetsVariationalanalysis1998} for a comprehensive background on convex functions and only collect the results needed in this article. For simplicity, we state all results in this section for $V=\R^n$.\\
	
	The space $\Conv(\R^n,\R)$ of all finite convex functions carries a natural topology induced by epi-convergence. We will not need the precise definition of this notion, since it is equivalent to locally uniform convergence or pointwise convergence in our setting (compare \cite[Theorem 7.17]{RockafellarWetsVariationalanalysis1998}). In particular, this topology on $\Conv(\R^n,\R)$ is metrizable.\\
	It is easy to see that any function in $\Conv(\R^n,\R)$ is locally bounded. 
	For later use, we remark that this implies that the function $f\mapsto \sup_{x\in A} |f(x)|$ is continuous on $\Conv(\R^n,\R)$ for any $A\subset\R^n$ compact.
	Moreover, the following result, which is a special case of \cite[9.14]{RockafellarWetsVariationalanalysis1998}, shows that these functions are in particular locally Lipschitz continuous.
		\begin{lemma}
			\label{lemma:LipschitzConstant}
			Let $U\subset \R^n$ be a convex open subset and $f:U\rightarrow\R$ a convex function. If $A\subset U$ is a set with $A+\epsilon B_1(0)\subset U$ and $f$ is bounded on $A+ \epsilon B_1(0)$, then $f$ is Lipschitz continuous on $A$ with Lipschitz constant $\frac{2}{\epsilon}\sup_{x\in A+\epsilon B_1(0)}|f(x)|$.
		\end{lemma}
	
	For $f\in\Conv(\R^n,\R)$, its convex conjugate or Legendre transform $\mathcal{L}f:\R^n\rightarrow(-\infty,\infty]$ is defined by
	\begin{align*}
		\mathcal{L}f(y)=\sup_{x\in\R^n}\left(\langle y,x\rangle -f(x)\right).
	\end{align*}
	If $f$ is of class $C^\infty$ and such that there exists a constant $\lambda>0$ such that $x\mapsto f(x)-\lambda\frac{|x|^2}{2}$ is convex, then $df:\R^n\rightarrow\R^n$ is a diffeomorphism and
	\begin{align*}
		\mathcal{L}f(df(x))=\langle df(x),x\rangle -f(x) \quad\text{for all }x\in\R^n.
	\end{align*} 
	Using the inverse function theorem, it is easy to see that this also implies that $\mathcal{L}f$ is of class $C^\infty$ with positive definite Hessian.\\
	
	Let us consider the following two subspaces of $\Conv(\R^n,\R)$:
	\begin{align*}
		\Conv_0(\R^n,\R):=&\{f\in\Conv(\R^n,\R): f(0)=0< f(x)\quad\text{for all } 0\ne x\in\R^n\},\\
		\Conv_0^+(\R^n,\R):=&\left\{f\in\Conv_0(\R^n,\R):f-\lambda\frac{|\cdot|^2}{2}\in\Conv(\R^n,\R) \text{ for some }\lambda>0\right\}.
	\end{align*}
	Note that $f\in\Conv_0^+(\R^n,\R)$ implies that the sublevel sets
	$\{y\in\R^n: \mathcal{L}f(y)\le t\}$ are compact and non-empty for $t\ge 0$. Moreover, if $f\in \Conv_0^+(\R^n,\R)$ is of class $C^1$, then $df(x)=0$ if and only if $x=0$, as $f$ has a unique minimum in $x=0$. If $f$ is in addition smooth, this also implies $d\mathcal{L}f(y)=0$ if and only if $y=0$.\\
	
	We will need the following estimates.
	\begin{lemma}[\cite{KnoerrSingularvaluationsHadwiger2022} Lemma 2.5]
			\label{lemma:EstimatesLegendreSublevelSetsConv0}
			Let $f\in\Conv^+_0(\R^n,\R)$ be a smooth function. If $0\ne y=df(x_0)$ for $x_0\in\R^n$, then
			\begin{enumerate}
				\item $|\mathcal{L}f(y)|\le (1+2|x_0|)\sup_{|x|\le |x_0|+1}|f(x)|$,
				\item $|y|\le 2\sup_{|x|\le |x_0|+1}|f(x)|$.
			\end{enumerate}
			Moreover, $\mathcal{L}f(y)\le (1+2R)\sup_{|x|\le R+1}|f(x)|$ implies $|y|\le 2\sup_{|x|\le R+1}|f(x)|$ for $R>0$.
		\end{lemma}

	\subsection{Some notions from geometric measure theory}
		We refer to the monograph \cite{FedererGeometricmeasuretheory1969} by Federer for a comprehensive background on geometric measure theory and only summarize the notions required in the constructions.\\
		
		First, all measures considered in this article will be signed Radon measures on $\R^n$, which we will consider as elements of topological dual $\mathcal{M}(\R^n):=(C_c(\R^n))'$. Here $C_c(\R^n)$ is equipped with its standard inductive topology, that is, a sequence converges if and only if the functions converge uniformly and there exists a compact subset of $\R^n$ that contains the supports of these functions. In particular, these measures are inner and outer regular. To any signed measure $\nu\in\mathcal{M}(\R^n)$ we associate its variation $|\nu|\in\mathcal{M}(\R^n)$, which for open sets $U\subset\R^n$ is given by
		\begin{align*}
			|\nu|(U)=\sup_{\psi\in C_c(\R^n), |\psi|\le 1_U} \left|\int_{\R^n}\psi d\nu\right|.
		\end{align*}
		Now let $M$ denote a smooth manifold,  $\Omega^k(M)$ the space of smooth differential forms of degree $k$ on $M$, $\Omega_c^k(M)$ the subspace of forms with compact support. For simplicity we assume that $M$ carries a Riemannian metric, so the tangent spaces and spaces of $k$-forms carry natural norms. A $k$-current on $M$ is a continuous linear map $T:\Omega_c^k(M)\rightarrow\R$. Given a $k$-current $T$ and $\tau\in \Omega^l(M)$, $l\le k$, we define
		\begin{align*}
			&\partial T:\Omega^{k-1}_c(M)\rightarrow\R, &&[\partial T](\omega)=T(d\omega),\\
			&T\llcorner \tau:\Omega^{k-l}_c(M)\rightarrow\R, &&[T\llcorner \tau](\omega)=T(\tau\wedge\omega).
		\end{align*} 
		If $M$ and $N$ are two smooth manifolds and $T$ is a $k$-current on $M$, $S$ is an $l$-current on $N$, their product $T\times S$ is a $(k+l)$-current on $M\times N$ and uniquely determined by the property $(T\times S)[\pi_M^*\phi\wedge\pi_N^*\psi]=T(\phi)S(\psi)$ for $\phi\in \Omega^k_c(M)$, $\psi\in \Omega^l_c(N)$, where $\pi_M$ and $\pi_N$ denote the natural projections.\\	
		If $F:M\rightarrow N$ is a smooth map that is proper on the support of a $k$-current $T$ on $M$, then we denote the pushforward of $T$ by $F_*T$. It is defined by $(F_*T)[\omega]=T(\psi F^*\omega)$, where $\psi\in C^\infty(M)$ is any function with $\psi=1$ on a neighborhood of $\supp T\cap F^{-1}(\supp\omega)$, which is a compact set as $F$ is proper.\\
		The mass of a $k$-current $T$ on an open subset $U\subset M$ is given by
		\begin{align*}
			M_U(T)=\sup \{|T(\omega)|: \omega\in \Omega^k_c(U), |\omega|_p|\le 1\quad\forall p\in U\},
		\end{align*}
		where we use the natural norms on $\Lambda^k(T_pM)^*$ induced by the Riemannian structure. If $M_U(T)$ is finite for every relatively compact open subset $U\subset M$, then $T$ is said to have locally finite mass. In this case, $T$ extends to all compactly supported forms whose coefficients are bounded Borel functions.\\
		If both $T$ and $\partial T$ have locally finite mass, then $T$ is called locally normal. The space $N_{k,loc}(M)$ of all locally normal currents on $M$ becomes a complete locally convex vector space with respect to the family of semi-norms $N_U(T):=M_U(T)+M_U(\partial T)$ for $U\subset M$ open and relatively compact. Given $T\in N_{k,loc}(M)$ and a smooth map $h:M\rightarrow\R$ that is proper on $\supp T$, there exists a unique Lebesgue-almost everywhere defined measurable map $\langle T,h,\cdot\rangle: \R\rightarrow N_{k-1,loc}(M)$ such that for every $\phi\in C^\infty_c(\R)$
		\begin{align}
			\label{equation:slicingFormula}
			\int_\R\phi(t)\langle T,h,t\rangle[\omega]dt=T(h^*(\phi dt)\wedge \omega)\quad\text{for}~\omega\in \Omega^{k-1}_c(M),
		\end{align}
		compare \cite[Theorem 4.3.2]{FedererGeometricmeasuretheory1969}. This map is unique up to sets of Lebesgue measure zero and given by
		\begin{align*}
			\langle T,h,t\rangle =\partial (T\llcorner 1_{\{h<t\}})-\partial  T\llcorner 1_{\{h>t\}}=-\partial (T\llcorner 1_{\{h>t\}})+\partial  T\llcorner 1_{\{h<t\}}\quad \text{for almost all}~t\in\R.
		\end{align*}
		We refer to \eqref{equation:slicingFormula} as the \emph{slicing formula}.\\
		
		Finally, let us remark that most of the currents considered in this article are integral currents. We refer to \cite{FedererGeometricmeasuretheory1969} for the definition and only note that integral currents are locally normal. In particular, we can apply the slicing formula to these currents. 
	\subsection{Differential forms on $S\R^n$ and $T^*\R^n$}
		\label{section:DiffFormsCotangentBundel}
		The product structure $T^*\R^n=\R^n\times (\R^n)^*$ induces a bigrading on the space $\Omega^*(T^*\R^n)$ of all smooth differential forms on $T^*\R^n$, and we denote the corresponding subspaces by $\Omega^{k,l}(T^*\R^n)$. We also write $\pi_1:T^*\R^n\rightarrow\R^n$ and $\pi_2:T^*\R^n\rightarrow(\R^n)^*$ for the two projections.\\
	
		Recall that the cotangent bundle $T^*\R^n$ carries a natural symplectic form $\omega_s$. If we choose linear coordinates $(x_1,\dots,x_n)$ on $\R^n$ with induced coordinates $(\xi_1,\dots,\xi_n)$ on $(\R^n)^*$, then this form is given by $\omega_s=-d\alpha$ for the $1$-form
		\begin{align*}
			\alpha=\sum_{j=1}^{n}\xi_j dx_j.
		\end{align*} 
		Since $\omega_s$ is a symplectic form, it is non-degenerate, so to any $1$-form $\delta\in\Omega^1(T^*\R^n)$ we can associate a unique vector field $X_\delta$ on $T^*\R^n$ such that $i_{X_\delta}\omega_s=\delta$.
		
		We consider $\R^n$ equipped with its standard inner product and consequently identify $(\R^n)^*\cong\R^n$. Consider the forms
		\begin{align*}
			\beta :=&d\langle x,\xi\rangle -\alpha=\sum_{j=1}^{n}x_j d\xi_j,\\
			\gamma:=&\frac{1}{2}d|x|^2=\sum_{j=1}^{n}x_j dx_j.
		\end{align*}
		Let us call a form $\omega\in \Omega^*(T^*\R^n)$ homogeneous of degree $m$ if $G_t^*\omega=t^m \omega$, where $G_t:T^*\R^n\rightarrow T^*\R^n$, $G_t(x,y)=(tx,y)$.
		\begin{lemma}
			\label{lemma:ChangeFormsPolarCoordinates}
			Let $\omega\in\Omega^{k,l}(T^*\R^n)$ by homogeneous of degree $m$ and 
				\begin{align*}
				G:\R\times\R^n\times\R^n&\rightarrow\R^n\times\R^n\\
				(t,x,\xi)&\mapsto (tx,\xi).
			\end{align*} Then
			\begin{align*}
				G^*(i_{X_{\beta}}\omega)=t^mi_{X_{\beta}}\omega
			\end{align*}
		\end{lemma}
		\begin{proof}
			By linearity, we may assume that $\omega=\phi(x,\xi) dx^I\wedge d\xi^J$, $|I|=k$, $|J|=l$, where $\phi$ is homogeneous of degree $m-k$ in the first argument. Then
			\begin{align*}
				i_{X_{\beta}}\omega=\sum_{i\in I}\phi(x,\xi) \epsilon_ix_idx^{I\setminus\{i\}}\wedge d\xi^J,
			\end{align*}
			where $\epsilon_i\in\{\pm 1\}$ is determined by $dx_i\wedge dx^{I\setminus\{i\}}=\epsilon_i dx^I$. Consequently,
			\begin{align*}
				G^*i_{X_{\beta}}\omega=&\sum_{i\in I}\phi(tx,\xi) \epsilon_itx_i\left(\bigwedge_{j\in I\setminus\{i\}} d(tx_j)\right) \wedge d\xi^J\\
				=&t^{m}\sum_{i\in I}\phi(x,\xi) \epsilon_ix_i dx^{I\setminus\{i\}}\wedge d\xi^J\\
				&+t^{m-1}dt\wedge \sum_{i,j\in I,i\ne j}\phi(x,\xi) \epsilon_i\tilde{\epsilon}_{i,j}x_ix_j dx^{I\setminus\{i,j\}}\wedge d\xi^J,
			\end{align*}
			where $\tilde{\epsilon}_{i,j}\in \{\pm1\}$ is determined by $dx_j\wedge dx^{I\setminus \{i,j\}}= \tilde{\epsilon}_{i,j}dx^{I\setminus\{i\}}$. As
			\begin{align*}
				dx_i\wedge dx_j\wedge dx^{I\setminus \{i,j\}}=-dx_j\wedge dx_i\wedge dx^{I\setminus \{i,j\}},
			\end{align*}
			we obtain $\epsilon_i\tilde{\epsilon}_{i,j}=-\epsilon_j\tilde{\epsilon}_{j,i}$. In particular,
			\begin{align*}
				\sum_{i,j\in I,i\ne j}\epsilon_i\tilde{\epsilon}_{i,j}x_ix_j dx^{I\setminus\{i,j\}}
				=\frac{1}{2}\sum_{i,j\in I,i\ne j}\epsilon_i\tilde{\epsilon}_{i,j}x_ix_j dx^{I\setminus\{i,j\}}+\frac{1}{2}\sum_{i,j\in I,i\ne j}\epsilon_j\tilde{\epsilon}_{j,i}x_jx_i dx^{I\setminus\{i,j\}}=0.
			\end{align*}
			Thus
			\begin{align*}
				G^*i_{X_{\beta}}\omega=t^{m}\sum_{i\in I}\phi(x,\xi) \epsilon_ix_i dx^{I\setminus\{i\}}\wedge d\xi^J=t^{m}i_{X_{\beta}}\omega.
			\end{align*}
		\end{proof}
		In this article we will consider $S\R^n=\R^n\times S^{n-1}$ as a submanifold of $T^*\R^n$ using the map
		\begin{align*}
			j:S\R^n\rightarrow T^*\R^n\cong \R^n\times\R^n\\
			(\xi,v)\mapsto (v,\xi).
		\end{align*}
		Then the restriction of $\gamma$ to $S\R^n$ vanishes and $\beta$ restricts to a \emph{contact form} on $S\R^n$, that is, the restriction $j^*d\beta$ is non-degenerate on the contract distribution $H:=\ker j^*\beta$. This implies that there exists a unique vector field $T$ on $S\R^n$ such that $i_T(j^*\beta)=1$ and $i_T(j^*d\beta)=0$, which is called the \emph{Reeb vector field}. Finally, note that the product structure induces a bigrading on $\Omega^*(S\R^n)$, and we denote the corresponding components by $\Omega^{k,l}(S\R^n)$.
		
		\subsection{Smooth valuations on convex bodies and smooth curvature measures}
			\label{section:valuationsBodies}
			Let $\mathcal{K}(\R^n)$ denote the space of convex bodies in $\R^n$, that is, the set of all non-empty, convex, and compact subsets of $\R^n$ equipped with the Hausdorff metric. For $K\in\mathcal{K}(\R^n)$, the set
			\begin{align*}
				\nc(K):=\{(\xi,v)\in\R^n\times S^{n-1}:v\text{ outer normal to } K \text{ in }\xi\in\partial K\}
			\end{align*}
			is a compact Lipschitz submanifold of the sphere bundle $S\R^n$ of dimension $n-1$ that carries a natural orientation. We may thus consider $\nc(K)$ as an integral $(n-1)$-current on $S\R^n$, called the \emph{normal cycle} of $K$, which we will denote by $\tau\mapsto \nc(K)[\tau]$, where $\tau\in\Omega^{n-1}(S\R^n)$ is a smooth $(n-1)$-form.
			
			\medskip
			Given $\tau\in \Omega^{n-1}(S\R^n)$, we may in particular associate to any $K\in\mathcal{K}(\R^n)$ the signed measure 
			\begin{align*}
				B\mapsto \Phi_\tau(K)[B]:=\nc(K)[1_{(\pi_2\circ j)^{-1}(B)}\tau]\quad\text{for Borel sets }B\subset \R^n.
			\end{align*}
			Here $\pi_2\circ j:S\R^n\rightarrow\R^n$ is the projection onto the first factor and $1_A$ denotes the indicator function of a set $A$. The map $\Phi_\tau$ is called a smooth curvature measure. Then $\Phi_\tau:\mathcal{K}(\R^n)\rightarrow \mathcal{M}(\R^n)$ is continuous with respect to the weak*-topology (compare, for example, the proof of \cite[Lemma 2.1.3]{AleskerFuTheoryvaluationsmanifolds.2008}).\\
		
			Similarly, we can associate to any $\tau\in \Omega^{n-1}(S\R^n)$ the functional $\phi_\tau:\mathcal{K}(\R^n)\rightarrow \R$,
			\begin{align*}
				\phi_\tau(K)=\nc(K)[\tau].
			\end{align*}
			Then $\phi_\tau$ is a continuous valuation on $\mathcal{K}(\R^n)$ (compare \cite{AleskerTheoryvaluationsmanifolds.2006}), and functionals of this type are called \emph{smooth valuations}. 
			
			Consider the differential forms
			\begin{align*}
				\kappa_k=\frac{1}{k!(n-1-k)!}\sum_{\sigma\in S_n}\sign(\sigma)v_{\sigma(1)}d\xi_{\sigma(2)}\dots d\xi_{\sigma(k+1)}\wedge dv_{\sigma(k+2)}\dots dv_{\sigma(n)},
			\end{align*}
			where $(\xi_1,\dots,\xi_n)$ are coordinates on $\R^n$ with induced coordinates $(v_1,\dots,v_n)$ on $S^{n-1}\subset\R^n$ and $S_n$ denotes the symmetric group of degree $n$.
			Then for $0\le i\le n-1$ the curvature measures 
			\begin{align*}
				C_i(K)[B]:=\frac{1}{(n-i)\omega_{n-i}}\nc(K)\left[1_{(\pi_2\circ j)^{-1}(B)}\kappa_k\right]
			\end{align*}
			are the \emph{Federer curvature measures} \cite{FedererCurvaturemeasures1959}, and the smooth valuations $\mu_k(K):=C_k(K,\R^n)$ are the \emph{intrinsic volumes}, compare \cite[Section 2.1]{FuAlgebraicintegralgeometry2014}.
		
			\begin{proposition}
				\label{proposition:BoundCurvatureMeasure}
				For every translation invariant differential form $\tau\in \Omega^{k,n-1-k}(S\R^n)$ there exists a constant $C(\tau)\ge0$ such that
				\begin{align*}
					|\Phi_\tau(K)|(U)\le C(\tau)  C_k(K)[U]
				\end{align*}
				for all $K\in\mathcal{K}(\R^n)$, $U\subset \R^n$ Borel set.
			\end{proposition}
			\begin{proof}
				If $K\in\mathcal{K}(\R^n)$ is a smooth convex body with strictly positive curvature, then $\nc(K)$ is given by the submanifold $\{(\xi,\nu_K(\xi)): \xi\in \partial K\}$, where $\nu_K:\partial K\rightarrow S^{n-1}$ denotes the Gauss map. In this case it is thus sufficient to consider the restriction of $\tau$ to the tangent spaces of $\nc(K)$. As $\tau$ is translation invariant, it is sufficient to assume that we are given a point $(0,v)\in\nc(K)$. Since the differential of $\nu_K$ is self adjoint, there exists an oriented orthonormal basis $e_1,\dots,e_{n-1}$ of $v^{\perp}$ and $\lambda_1,\dots,\lambda_{n-1}>0$ such that the tangent space $T_{(0,v)}\nc(K)$ is spanned by $(e_1,\lambda_1e_1)$,\dots, $(e_{n-1},\lambda_{n-1}e_{n-1})$. Define $\bar{e}_i=(e_i,0)$, $\bar{f}_i=(0,e_i)$. Then the restriction of $\tau$ to $T_{(0,v)}\nc(K)$ is a multiple of the volume form given by
				\begin{align*}
					&\tau|_{(0,v)}(\bar{e}_1+\lambda_1\bar{f}_1,\dots,\bar{e}_{n-1}+\lambda_i\bar{f}_{n-1})\\
					=&\frac{1}{k!(n-k-1)!}\sum_{\pi\in S_{n-1}}\sign(\pi)\lambda_{\pi(k+1)}\dots\lambda_{\pi(n-1)} \tau|_{(0,v)}(\bar{e}_{\pi(1)},\dots,\bar{e}_{\pi(k)},\bar{f}_{\pi(k+1)},\dots,\bar{f}_{\pi(n-1)}),
				\end{align*}
				where we have used that $\tau$ is a $(k,n-k-1)$-form. Consider the fiber bundle $\mathcal{S}_{k,n-k-1}$ over $S^{n-1}$ with fiber over $v\in S^{n-1}$ given in the notation above by the space of all elements $(\bar{e}_1,\dots,\bar{e}_k,\bar{f}_{k+1},\dots,\bar{f}_{n-1})\in (v^\perp\oplus v^\perp)^{n-1}$ for an oriented orthonormal basis $e_1,\dots,e_{n-1}$ of $v^\perp$. Then it is easy to see that $\mathcal{S}_{k,n-k-1}$ is naturally a compact manifold. As the evaluation map is continuous, the map
				\begin{align*}
					\mathcal{S}_{k,n-k-1}\rightarrow&\R\\
					(v, (\bar{e}_1,\dots,\bar{e}_k,\bar{f}_{k+1},\dots,\bar{f}_{n-1}))\mapsto&|\tau|_{(0,v)}(\bar{e}_{1},\dots,\bar{e}_{k},\bar{f}_{k+1},\dots,\bar{f}_{n-1})|
				\end{align*}
				is thus bounded by some $c(\tau)\ge 0$. Now note that 
				\begin{align*}
					\sign(\pi) \kappa_k|_{(0,v)}(\bar{e}_{\pi(1)},\dots,\bar{e}_{\pi(k)},\bar{f}_{\pi(k+1)},\dots,\bar{f}_{\pi(n-1)})=1
				\end{align*}
				for every choice of positively oriented orthonormal basis $e_1,\dots,e_{n-1}$ of $v^\perp$ and every $\pi\in S_{n-1}$. Thus
				\begin{align*}
					&\pm\tau|_{(0,v)}(\bar{e}_1+\lambda_1\bar{f}_1,\dots,\bar{e}_{n-1}+\lambda_i\bar{f}_{n-1})\\
					=&\pm\frac{1}{k!(n-k-1)!}\sum_{\pi\in S_{n-1}}\sign(\pi)\lambda_{\pi(k+1)}\dots\lambda_{\pi(n-1)} \tau|_{(0,v)}(\bar{e}_{\pi(1)},\dots,\bar{e}_{\pi(k)},\bar{f}_{\pi(k+1)},\dots,\bar{f}_{\pi(n-1)})\\
					\le& c(\tau)\frac{1}{k!(n-k-1)!}\sum_{\pi\in S_{n-1}}\sign(\pi)\lambda_{\pi(k+1)}\dots\lambda_{\pi(n-1)} \kappa_k|_{(0,v)}(\bar{e}_{\pi(1)},\dots,\bar{e}_{\pi(k)},\bar{f}_{\pi(k+1)},\dots,\bar{f}_{\pi(n-1)})\\
					=&c(\tau)\kappa_k|_{(0,v)}(\bar{e}_1+\lambda_1\bar{f}_1,\dots,\bar{e}_{n-1}+\lambda_i\bar{f}_{n-1}).
				\end{align*}
				As $\tau$ and $\kappa_k$ are translation invariant, it follows that the restriction of $c(\tau)\kappa_k\pm\tau$ to any tangent space $T_{(x,v)}\nc(K)$ is non-negative for any $K\in\mathcal{K}(\R^n)$ smooth with positive Gauss curvature. Thus for $C(\tau):=c(\tau)(n-k)\omega_{n-k}$,
				\begin{align*}
					C(\tau)C_k(K)\pm\Phi_\tau(K)\ge 0
				\end{align*}
				for every smooth convex body with strictly positive Gauss curvature. As smooth curvature measures are continuous with respect to the weak* topology, this holds for all convex bodies, so for $U\subset\R^n$ open, we obtain
				\begin{align*}
					&|\Phi_\tau(K)|(U)=\sup_{\psi\in C_c(\R^n),|\psi|\le 1_U} \left|\int_{\R^n}\psi(x)d\Phi_\tau(K,x)\right|\\
					=&\frac{1}{2}\sup_{\psi\in C_c(\R^n),|\psi|\le 1_U} \left|\int_{\R^n}\psi(x)d[C(\tau) C_k(K)+\Phi_\tau(K)]-\int_{\R^n}\psi(x)d[C(\tau) C_k(K)-\Phi_\tau(K)]\right|\\
					\le& \frac{1}{2}[C(\tau) C_k(K)+\Phi_\tau(K)](U)+[C(\tau) C_k(K)-\Phi_\tau(K)](U)\\
					=&C(\tau) C_k(K)[U]
				\end{align*}
				for all $K\in\mathcal{K}(\R^n)$. Since $|\Phi_\tau(K)|$ is a Radon measure, the claim follows.
				
			\end{proof}
		
			\subsection{Properties of the differential cycle}
			In this section we collect some results on the differential cycle, which was introduced by Fu in \cite{FuMongeAmperefunctions.1989}. We also refer to \cite{JerrardSomerigidityresults2010} for a generalization of this notion. Let $\vol\in\Lambda^n(\R^n)^*$ denote a volume form.
		\begin{theorem}[\cite{FuMongeAmperefunctions.1989} Theorem 2.0]
			\label{theorem:FuUniquenessDifferentialCycle}
			Let $f:\R^n\rightarrow\mathbb{R}$ be a locally Lipschitzian function. There exists at most one integral current $S\in I_n(T^*\R^n)$ such that
			\begin{enumerate}
				\item $S$ is closed, i.e. $\partial S=0$,
				\item $S$ is Lagrangian, i.e. $S\llcorner \omega_s=0$,
				\item $S$ is locally vertically bounded, i.e. $\supp S\cap \pi^{-1}(A)$ is compact for all $A\subset \R^n$ compact,
				\item $S(\phi(x,y)\pi^*\vol)=\int_{\R^n}\phi(x,df(x))d\vol(x)$ for all $\phi\in C^\infty_c(T^*\R^n)$.
			\end{enumerate}
			Note that the right hand side of the last equation is well defined due to Rademacher's theorem.
		\end{theorem}
		If this current exists, then it is called the \emph{differential cycle} of $f$, denoted by $D(f)$, and $f$ is called Monge-Amp\`ere. By \cite[Theorem 2.0 and Proposition 3.1]{FuMongeAmperefunctions.1989} any convex function is Monge-Amp\`ere. Note that this current only depends on the choice of orientation of $\R^n$ and not on the choice of the volume form $\vol$.
		\begin{proposition}
			\label{proposition:DifferentialCycleSumMAOrthogonalCylinderFunctions}
			Let $V$, $W$ be two finite dimensional oriented real vector spaces, $\pi_V:V\times W\rightarrow V$, $\pi_W:V\times W\rightarrow W$ the natural projections. If $f:V\rightarrow\R$ and $g:W\rightarrow\R$ are Monge-Amp\`ere, then so is the function $\pi_V^*f+\pi_W^*g:V\oplus W\rightarrow\R$ and
			\begin{align*}
				D(\pi_V^*f+\pi_W^*g)=G_*(D(f)\times D(g)),
			\end{align*}
			where $G:T^*V\times T^*W\rightarrow T^*(V\times W)$ is given by $G((x,\xi),(y,\eta))=((x,y),(\xi,\eta))$.
		\end{proposition}
		\begin{proof}
			Due Theorem \ref{theorem:FuUniquenessDifferentialCycle}, it suffices to show that the current $G_*(D(f)\times D(g))$ satisfies the defining properties of $D(\pi_V^*f+\pi_W^*g)$, which is easily verified.
		\end{proof}	
		
%
		
		\begin{lemma}[\cite{KnoerrSingularvaluationsHadwiger2022} Lemma 4.12]
			\label{lemma:relationLegendreBeta}
			For $f\in \Conv(\R^n,\R)\cap C^\infty(\R^n)$,
			\begin{align*}
				D(f)\llcorner\pi_2^*d\mathcal{L}f=D(f)\llcorner \beta.
			\end{align*}
		\end{lemma}
		For $f\in\Conv_0^+(\R^n,\R)$ and $t\ge 0$ we set
		\begin{align*}
			K_{\mathcal{L}f}^t:=\{y\in \R^n: \mathcal{L}f(y)\le t\}\in\mathcal{K}(\R^n).
		\end{align*}
		\begin{proposition}[\cite{KnoerrSingularvaluationsHadwiger2022} Proposition 4.11]
			\label{proposition:RelationNCSlicesDiffCycle}
			Let $f\in \Conv_0^+(\R^n,\R)\cap C^\infty(\R^n)$. Consider the map
			\begin{align*}
				F_f:\R^n\times S^{n-1}&\rightarrow \R^n\times\R^n\\
				(y, v)&\mapsto \left(|\mathcal{L}f(y)|\cdot v,y\right).
			\end{align*}
			Then $(F_f)_{*}\nc(K_{\mathcal{L}f}^t)=\langle D(f),\pi_2^*\mathcal{L}f,t\rangle $ for almost all $t>0$.
		\end{proposition}
	
			\begin{lemma}
			\label{lemma:pullbackSpericalFormsFf}
			Let $\omega\in\Omega^{k,l}(T^*\R^n)$ by homogeneous of degree $m$. Then
			\begin{align*}
				F_f^*(i_{X_{\beta}}\omega)=|d\mathcal{L}f(\xi)|^{m} j^*(i_{X_{\beta}}\omega)\quad\text{in}~(\xi,v)\in \R^n\times S^{n-1}
			\end{align*}
		\end{lemma}
		\begin{proof}
			Note that $F_f= G\circ F\circ j$ for 
			\begin{align*}
				G:\R\times\R^n\times\R^n&\rightarrow\R^n\times\R^n\\
				(t,x,\xi)&\mapsto (tx, \xi)\\
				F:\R^n\times \R^{n}&\rightarrow\R\times \R^n\times\R^n\\
				(x,\xi)&\mapsto (|d\mathcal{L}f(\xi)|,x,\xi).
			\end{align*}
			Thus the claim follows from Lemma \ref{lemma:ChangeFormsPolarCoordinates}.
		\end{proof}
		On $S\R^n$ there exists a natural second order differential operator $D$ defined on $(n-1)$-forms, called the \emph{Rumin differential}. We refer to \cite{RuminFormesdifferentiellessur1994} for its definition and to \cite{BernigBroeckerValuationsmanifoldsRumin2007} for its role in the description of smooth valuations on convex bodies. In the following result, $T$ denotes the Reeb vector field, compare Section \ref{section:valuationsBodies}.
		\begin{lemma}[\cite{KnoerrSingularvaluationsHadwiger2022} Corollary 4.10]
			\label{lemma:estimatePartialIntegrationSublevelSets}
			Let $f\in\Conv_0^+(\R^n,\R)\cap C^\infty(\R^n)$. For $\phi\in C^1_c((0,\infty))$ and $c>0$
			\begin{align*}
				\int_0^c\phi(t)\nc(K_{\mathcal{L}f}^t)\left[\frac{1}{|d\mathcal{L}f(y)|}i_TD\omega\right]dt=\phi(c)\nc(K_{\mathcal{L}f}^c)[\omega]-\int_0^c \phi'(t)\nc(K_{\mathcal{L}f}^t)[\omega]dt.
			\end{align*}
			Moreover, 
			\begin{align*}
				\left|\int_0^c	|\phi(t)|\nc(K_{\mathcal{L}f}^t)\left[\frac{1}{|d\mathcal{L}f(y)|}i_TD\omega\right]dt\right|\le2\int_0^c|\phi'(t)|dt\sup_{t\in [0,c]}|\nc(K_{\mathcal{L}f}^t)[w]|.
			\end{align*}
		\end{lemma}
		\begin{corollary}
			\label{corollary:IntrinsicVolumesEstimateSublevelSets}
			Let $f\in\Conv_0^+(\R^n,\R)\cap C^\infty(\R^n)$. For $\phi\in C^1_c((0,\infty))$ and $c>0$
			\begin{align*}
				\int_0^c	|\phi(t)|\int_{\R^n}\frac{1}{|d\mathcal{L}f(y)|}dC_{k-1}(K^t_{\mathcal{L}f(w)})dt\le&\frac{2(n-k+1)\omega_{n-k+1}}{\omega_{n-k}}\mu_k(K_{\mathcal{L}f}^c)\int_0^c|\phi'(t)|dt.
			\end{align*}
		\end{corollary}
		\begin{proof}
			This follows by applying Lemma \ref{lemma:estimatePartialIntegrationSublevelSets} to the differential form $\kappa_k$, which satisfies 
			\begin{align*}
				i_TD\kappa_k=\frac{(n-k+1)\omega_{n-k+1}}{\omega_{n-k}}\kappa_{k-1},
			\end{align*}
		compare \cite[Lemma 3.1]{FuSomeremarksLegendrian1998}.
		\end{proof}
		
%
	
	\subsection{Dually epi-translation invariant valuations}
	\label{section:VConv}
	\subsubsection{Topology and support}
	We equip $\VConv_k(\R^n)$ and all of its subspaces with the topology of uniform convergence on compact subsets in $\Conv(\R^n,\R)$ (see \cite[Proposition 2.4]{Knoerrsupportduallyepi2021} for a description of these sets). In other words, $\VConv(\R^n)$ is a locally convex vector space with the topology induced by the semi-norms
	\begin{align}
		\label{eq:definitionSeminorms}
		\|\mu\|_K:=\sup_{f\in K}|\mu(f)|
	\end{align}
	for all compact subsets $K\subset \Conv(\R^n)$. It is easy to see that $\VConv(\R^n)$ is complete with respect to this topology. Recall that we have the homogeneous decomposition
	\begin{align*}
		\VConv(\R^n)=\bigoplus_{k=0}^n\VConv_k(\R^n).
	\end{align*}
	For any valuation  $\mu\in\VConv_k(\R^n)$, we can thus define its \emph{polarization} $\bar{\mu}:\Conv(\R^n,\R)^k\rightarrow\R$ by
	\begin{align*}
		\bar{\mu}(f_1,\dots,f_k)=\frac{1}{k!}\frac{\partial^k}{\partial \lambda_1\dots\lambda_k}\Big|_0\mu\left(\sum_{j=1}^{k}\lambda_jf_j\right).
	\end{align*}
	Note that this is well defined since the right hand side is a polynomial in $\lambda_1,\dots,\lambda_k\ge 0$. It is then easy to see that $\bar{\mu}$ is a symmetric functional which is additive in each argument. By \cite[Corollary 4.12]{Knoerrsupportduallyepi2021}, $\bar{\mu}$ is jointly continuous for every $\mu\in\VConv_k(\R^n)$. Moreover, the polarization depends continuously on the valuation in the follow sense.
	\begin{lemma}[\cite{Knoerrsupportduallyepi2021} Lemma 4.13]
		\label{lemma:continuityPolarization}
		There exists a constant $c_{k}>0$ depending on $0\le k\le n$ only such that the following holds: If $K\subset \Conv(\R^n,\R)$ is compact and convex with $0\in K$, then 
		\begin{align*}
			\|\bar{\mu}\|_{K}:=\sup_{f_1,\dots,f_k\in K}|\bar{\mu}(f_1,\dots,f_k)|\le c_{k}\|\mu\|_{K}
		\end{align*}
		for all $\mu\in\VConv_k(\R^n)$.
	\end{lemma}
	\begin{remark}
		The inequality also holds if $K$ is not compact. In this case, both sides of the inequality can be infinite.
	\end{remark}
	In \cite{Knoerrsupportduallyepi2021} a notion of support was introduced for elements of $\VConv(\R^n)$. The main results show that the support $\supp\mu\subset \R^n$ of $\mu\in\VConv(\R^n)$ is a compact subset. Moreover, the support is characterized by the following property, which may also be used as a definition.
	\begin{proposition}[\cite{Knoerrsupportduallyepi2021} Proposition 6.3]
		\label{proposition:characterizationSupport}
		The support of $\mu\in\VConv(\R^n)$ is minimal (with respect to inclusion) among the closed sets $A\subset \R^n$ with the following property: If $f,g\in\Conv(\R^n,\R)$ satisfy $f=g$ on an open neighborhood of $A$, then $\mu(f)=\mu (g)$.
	\end{proposition}
	It will be useful to extend this notion to more general functionals. Assume that $V_1,\dots,V_k$ are finite dimensional real vector spaces and let $\pi_i:V_1\times\dots V_k\rightarrow V_i$ denote the $i$th projection. Let $\nu:\Conv(V_1,\R)\times\dots\times\Conv(V_k,\R)\rightarrow\R$ be given. We will say that a closed set $A\subset V_1\times\dots\times V_k$ supports $\nu$ if for all $f_i,h_i\in \Conv(V_i,\R)$, $1\le i\le k$, such that $f_i\equiv h_i$ on a neighborhood of $\pi_i(A)$ we have the equality
	\begin{align*}
		\nu(f_1,\dots,f_k)=\nu(h_1,\dots,h_k).
	\end{align*}
	In this case we will also say that $\nu$ is supported on $A$.
	\begin{lemma}
		\label{lemma:SupportPolarization}
		If $\mu\in\VConv_k(\R^n)$, then $\bar{\mu}:\Conv(\R^n,\R)^k\rightarrow\R$ is supported on $(\supp \mu)^k$.
	\end{lemma}
	\begin{proof}
		Let $h_1,\dots,f_k,h_1,\dots h_k\in\Conv(\R^n,\R)$ satisfy $f_i=h_i$ on a neighborhood of $\supp\mu$. Then for any $\lambda_1,\dots,\lambda_k\ge 0$, 
		\begin{align*}
			\sum_{i=1}^{k}\lambda_if_i\equiv \sum_{i=1}^{k}\lambda_ih_i\quad\text{on a neighborhood of}~\supp\mu.
		\end{align*}
		The claim now follows immediately from the definition of $\bar{\mu}$ and the characterization of $\supp\mu$ in Proposition \ref{proposition:characterizationSupport}.
	\end{proof}
	
	For a compact subset $A\subset \R^n$ let $\VConv_A(\R^n)$ denote the subspace of valuations with support contained in $A$. Then it is easy to see that $\VConv_A(\R^n)$ is a closed subspace of $\VConv(\R^n)$ (compare \cite[Lemma 6.7]{Knoerrsupportduallyepi2021}). It turns out that the topology of these spaces is rather simple. The following is a special case of \cite[Proposition 6.8 and Corollary 6.10]{Knoerrsupportduallyepi2021}.
	\begin{lemma}
		\label{lemma:NormsValuationsBoundedSupport}
		Let $R>0$. The relative topology of $\VConv_{B_R(0)}(\R^n)$ is induced by the norm
		\begin{align*}
			\|\mu\|_{R,1}:=\sup\left\{|\mu(f)|:~f\in\Conv(\R^n,\R),~\sup_{|x|\le R+1}|f(x)|\le 1\right\}.
		\end{align*}
		Moreover, $\VConv_{B_R(0)}(\R^n)$ is a Banach space with respect to the norm.
	\end{lemma}
	Consequently, we will usually work directly with the norms $\|\cdot\|_{R,1}$ as soon as we are dealing with valuations with bounds on their support.
	\begin{remark}
		Note that for $\mu\in\VConv_k(\R^n)$, the inclusion $\supp\mu\subset B_R(0)$ implies that 
		\begin{align*}
			|\mu(f)|\le \left(\sup_{|x|\le R+1}|f(x)|\right)^k \|\mu\|_{R,1}\quad\text{for all}~f\in\Conv(\R^n,\R).
		\end{align*}
		We will encounter estimates of this form in Section \ref{section:SingularValuations}.
	\end{remark}
	
	Let $T:V\rightarrow W$ be linear and let $T^*:\Conv(W,\R)\rightarrow\Conv(V,\R)$ denote the pullback. Then it is easy to see that $T^*$ is continuous. This implies the following for the pushforward defined in the introduction
	\begin{lemma}
		\label{lemma:continuityPushForward}
		Let $T:V\rightarrow W$ be linear. Then $T_*:\VConv(V)\rightarrow\VConv(W)$ is continuous. More precisely, if $K\subset \Conv(W,\R)$ is compact, then so is $T^*(K)$ and
		\begin{align*}
			\|T_*\mu\|_K=\|\mu\|_{T^*(K)}.
		\end{align*}
	\end{lemma}
	\begin{proof}
		$T^*(K)\subset \Conv(V,\R)$ is compact because $T^*$ is continuous. Thus the claim follows directly from the definition of the semi-norms $\|\cdot\|_K$ and the push forward $T_*$.
	\end{proof}

	\subsubsection{Smooth valuations on convex functions}
	The notion of smooth valuations on convex functions was introduced in \cite{KnoerrSmoothvaluationsconvex2024}. By definition, a valuation $\mu\in\VConv_k(\R^n)$ is called \emph{smooth} if there exists a differential form $\omega\in \Omega^{k}_c(\R^n)\otimes \Lambda^{n-k}((\R^n)^*)^*$ such that
	\begin{align*}
		\mu(f)=D(f)[\omega]\quad\text{for}~f\in\Conv(\R^n,\R).
	\end{align*}
	We denote the subspace of smooth valuations by $\VConv_k(\R^n)^{sm}$, which is dense in $\VConv_k(\R^n)$ by the main results of \cite{KnoerrSmoothvaluationsconvex2024}. 
	\begin{theorem}[\cite{KnoerrSmoothvaluationsconvex2024} Theorem 6.5]
		\label{theorem:approximation-by-smooth-valuations}
		$\VConv_k(\R^n)^{sm}\subset \VConv_k(\R^n)$ is sequentially dense. More precisely, the following holds: For every compact set $A \subset \R^n$ , $\mu \in \VConv_k(\R^n)\cap \VConv_{A}(\R^n)$, and every compact neighborhood
		$B \subset\R^n$ of $A$ there exists a sequence $(\mu_j)_j$ in $\VConv_k(\R^n)\cap \VConv_{B}(\R^n)$ of smooth valuations  such that $(\mu_j)_j$ converges to $\mu$.
	\end{theorem}
	Given a compact set $A\subset\R^n$, the subspace $\VConv_A(\R^n)\cap \VConv_k(\R^n)^{sm}$ is in general not dense in $\VConv_A(\R^n)\cap \VConv_k(\R^n)$ (it can in fact be trivial). However, we have the following result.
	\begin{proposition}
		\label{proposition:approximationUnderSupportRestriction}
		Let $A\subset\R^n$ be compact and convex with non-empty interior. Then $\VConv_A(\R^n)\cap \VConv_k(\R^n)^{sm}$ is dense in $\VConv_A(\R^n)$.
	\end{proposition}
	\begin{proof}	
		Without loss of generality we may assume that the origin is contained in the interior of $A$. For $\mu\in \VConv_A(\R^n)$  and $\delta \in (0,1)$ consider the valuation $\mu_\delta$ given by
		\begin{align*}
			\mu_\delta(f):=\mu\left(f\left(\frac{\cdot}{\delta}\right)\right).
		\end{align*}
		From the characterization of the support given in Proposition \ref{proposition:characterizationSupport}, it is easy to see that $\mu_\delta$ is supported on $\delta A$. From the fact that the map
		\begin{align*}
			(0,\infty)\times \Conv(\R^n,\R)&\rightarrow\Conv(\R^n,\R)\\
			(\delta,f)&\mapsto f\left(\frac{\cdot}{\delta}\right)
		\end{align*} is continuous, one easily deduces that $\mu_\delta$ converges to $\mu$ for $\delta\rightarrow1$. It is thus sufficient to show the claim under the additional assumption that the support of $\mu$ is contained in the interior of $A$. In this case, the claim follows directly from Theorem \ref{theorem:approximation-by-smooth-valuations}.
	\end{proof}
		If $G\subset \GL(n,\R)$ is a compact subgroup, let $\VConv_k(\R^n)^{G}$ denote the space of all $G$-invariant valuations. We will be interested in the subspace
		\begin{align*}
			\VConv_k(\R^n)^{G,sm}:=\VConv_k(\R^n)^{G}\cap \VConv_k(\R^n)^{sm}
		\end{align*}
		of all $G$-invariant smooth valuations
	\begin{corollary}\label{corollary:approximationUnderSupportRestrictionInvariantCase}
		Let $G\subset\SO(n,\R)$ be a compact subgroup. Then smooth valuations are dense in $\VConv_k(\R^n)^G\cap \VConv_{B_R(0)}(\R^n)$ for every $R>0$.
	\end{corollary}
	\begin{proof}
		By Proposition \ref{proposition:approximationUnderSupportRestriction} we can approximate any valuation $\mu\in\VConv_k(\R^n)^G\cap \VConv_{B_R(0)}(\R^n)$ by smooth valuations $(\mu_j)_j$ with the same support restriction. Replacing $\mu_j$ by $\tilde{\mu}_j\in\VConv_k(\R^n)^G$ defined by
		\begin{align*}
			\tilde{\mu}(f):=\int_G\mu(f\circ g)dg,
		\end{align*}
		we obtain a sequence of $G$-invariant valuations converging to $\mu$. It is now easy to see that $\tilde{\mu}_j$ is a smooth valuation with support in $B_R(0)$, compare the proof of \cite[Proposition 6.6]{KnoerrSmoothvaluationsconvex2024}.
	\end{proof}

	\section{Three families of integral transforms}
		\label{section:IntegralTransforms}
		In this section we will discuss three families of integral transforms that will occur in the proofs of our main results. The spaces of functions discussed in this section are listed below.\\
		 Recall that $C_b((0,\infty))$ and $C_b((0,\infty))$ denote the spaces of continuous functions with bounded support.\\
		In the following definition, $\tilde{S}^{a,b}:C_c((0,\infty)^2)\rightarrow C_b((0,\infty)^2)$ denotes the map defined in \eqref{eq:DefTildeSabDerivative} below. We consider the following spaces of continuous functions for $a,b\in\mathbb{N}$:
		\begin{align*}
			D^a=&\left\{\zeta\in C_b((0,\infty)):\lim_{t\rightarrow0}t^a\zeta(t)=0,~\lim\limits_{t\rightarrow0}\int_0^\infty \zeta(r)r^{a-1}dr~\text{exists and is finite}\right\},\\
			\tilde{D}^{a+2}=&\left\{\tilde{\zeta}\in C_b((0,\infty)): \lim\limits_{t\rightarrow0}t^{a+2}\tilde{\zeta}(t)=0\right\},\\
			C_0:=&\{\phi\in C_c([0,\infty)):\phi(0)=0\},\\
			C_{(a,b)}:=&\{\phi\in C_c([0,\infty)^2):\tilde{S}^{a,b}(\phi)[s,t]~\text{only depends on the value of}~s^2+t^2~\text{for}~s,t> 0\}.
		\end{align*}
		We denote by $D^a_R$, $\tilde{D}^{a+2}_R$, and $C_{0,R}$ the corresponding subspaces of functions supported on $(0,R]$ and $[0,R]$, and by $C_{(a,b),R}$ the subspace of $C_{(a,b)}$ of all functions supported on $B_R(0)$. Similarly, we will use the notation $C_{c,R}([0,\infty))$ and $C_{c,R}([0,\infty)^2)$ for the subspaces of $C_c([0,\infty))$ and $C_c([0,\infty)^2)$ of all functions with support contained in $[0,R]$ and $B_R(0)$ respectively. Note that $C_{0,R}\subset C_{c,R}([0,\infty))\subset C_c([0,\infty))$ and $C_{(a,b),R}\subset C_{c,R}([0,\infty)^2)\subset C_c([0,\infty)^2)$ are closed subspaces, and in fact Banach spaces with respect to the maximum norm.
		
	\subsection{The spaces $D^a$ and $\tilde{D}^{a+2}$}
	\label{section:DaTildeDa}
	By \cite[Lemma 2.7]{KnoerrSingularvaluationsHadwiger2022}, $D^a_R$ is a Banach space with respect to the norm $\|\cdot\|_{D^a}$. Moreover, $C_c([0,\infty))\cap D^{a}_R$ is dense in $D^a_R$, as the following result shows.
	\begin{lemma}[\cite{KnoerrSingularvaluationsHadwiger2022} Lemma 2.6]
		\label{lemma:DensityDNJ_continuousFunctions}
		For $\zeta\in D^{a}_{R}$ and $r>0$ define $\zeta^r\in C_c([0,\infty))\cap D^{a}_{R}$ by
		\begin{align*}
			\zeta^r(t):=\begin{cases}
				\zeta(t) & \text{for }t> r,\\
				\zeta(r) & \text{for }0\le t\le r.
			\end{cases}
		\end{align*}
		Then $\lim\limits_{r\rightarrow0}\|\zeta-\zeta^r\|_{D^{a}}=0$.
	\end{lemma}
	We will need the corresponding statements for the spaces $\tilde{D}^{a+2}_R$.
	\begin{lemma}
		\label{lemma:tildeDcomplete}
		$\tilde{D}^{a+2}_R$ is complete with respect to $\|\cdot\|_{\tilde{D}^{a+2}}$.
	\end{lemma}
	\begin{proof}
		Let $(\zeta_j)_j$ be a Cauchy sequence in $\tilde{D}^{a+2}_R$ with respect to $\|\cdot\|_{\tilde{D}^{a+2}}$. Set $\phi_j(t):=\zeta(t)t^{a+2}$. Then $\phi_j$ extends by continuity to an element of $C_c([0,\infty))$ supported on $[0,R]$ with $\phi_j(0)=\lim\limits_{t\rightarrow0}\zeta(t)t^{a+2}=0$. Moreover, 
		\begin{align*}
			\|\phi_j-\phi_k\|_\infty =\|\zeta_j-\zeta_k\|_{\tilde{D}^{a+2}},
		\end{align*}
		so $(\phi_j)_j$ is a Cauchy sequence in $C_c([0,\infty))$ with $\supp\phi_j\subset [0,R]$ for all $j\in\mathbb{N}$. Thus the sequence converges uniformly to some $\phi\in C_c([0,\infty))$ with $\supp\phi\subset [0,R]$. In particular, $\phi(0)=\lim\limits_{j\rightarrow\infty}\phi_j(0)=0$. Set $\zeta(t):=\frac{\phi(t)}{t^{a+2}}$ for $t>0$. Then it is easy to see that $\zeta\in \tilde{D}^{a+2}_R$ with 
		\begin{align*}
			\|\zeta-\zeta_j\|_{\tilde{D}^{a+2}}=\|\phi-\phi_j\|_\infty.
		\end{align*}
		Thus $(\zeta_j)_j$ converges to $\zeta\in \tilde{D}^{a+2}_R$.
	\end{proof}
	The following result shows that $\tilde{D}^{a+2}_R$ is the completion of $C_c([0,\infty])\cap \tilde{D}^{a+2}_R$ with respect to $\|\cdot\|_{\tilde{D}^{a+2}}$.
	\begin{lemma}
		\label{lemma:DensityTildeD_continuousFunctions}
		For $\tilde{\zeta}\in \tilde{D}^{a+2}_R$ and $r>0$ consider the function $\tilde{\zeta}_r(t)\in C_c([0,\infty))\cap \tilde{D}^{a+2}_R$ given by
		\begin{align*}
			\tilde{\zeta}_r(t):=\max\left(\frac{t^{a+2}}{r^{a+2}},1\right)\tilde{\zeta}(t).
		\end{align*}
		Then $\lim_{r\rightarrow0}\|\tilde{\zeta}-\tilde{\zeta}_r\|_{\tilde{D}^{a+2}}=0$. In particular, $C_c([0,\infty))\cap \tilde{D}_R^a$ is dense in $\tilde{D}^{a+2}_R$.
	\end{lemma}
	\begin{proof}
		Note that we have the estimate
		\begin{align*}
			t^{a+2}|\tilde{\zeta}(t)-\tilde{\zeta}_r(t)|\le 2\sup_{s\in (0,r]} s^{a+2}|\tilde{\zeta}(s)|\quad \text{for all}~t>0,
		\end{align*}
		so $\|\tilde{\zeta}-\tilde{\zeta}_r\|_{\tilde{D}^{a+2}}\le 2\sup_{s\in (0,r]} s^{a+2}|\tilde{\zeta}(s)|$, which converges to $0$ for $r\rightarrow0$.
	\end{proof}

	\subsection{The integral transforms $\mathcal{R}^a$ and $\mathcal{P}^a$}
	For $a\in\mathbb{N}$ consider the map $\mathcal{R}^a:C_b((0,\infty))\rightarrow C_b((0,\infty))$ given by
	\begin{align*}
		\mathcal{R}^a(\zeta)[t]=\zeta(t)t^a+a\int_t^\infty \zeta(s)s^{a-1}ds=-t^{a+1}\frac{d}{dt}\left(\frac{1}{t^a}\int_t^\infty\zeta(s)s^{a-1}ds\right) \quad \text{for}~t>0.
	\end{align*}
	Note that $\mathcal{R}^a$ does not increase the support: If $\supp\zeta\subset(0,R]$, then the same holds for $\mathcal{R}^a(\zeta)$.
	\begin{lemma}[\cite{ColesantiEtAlHadwigertheoremconvex} Lemma 3.7]
		\label{lemma:InjectivityTransformR}
		The map $\mathcal{R}^a: D^a_R\rightarrow C_{c,R}([0,\infty))$ is a bijection with inverse $(\mathcal{R}^a)^{-1}$ given by
		\begin{align*}
			(\mathcal{R}^a)^{-1}(\phi)[t]=\frac{\phi(t)}{t^a}-a\int_t^\infty \frac{\phi(s)}{s^{a+1}}ds=-\frac{1}{t^{a-1}}\frac{d}{dt}\left(t^a\int_t^\infty \frac{\phi(s)}{s^{a+1}}ds\right) \quad \text{for}~t>0.
		\end{align*}
	\end{lemma}
	
	\begin{lemma}
		\label{lemma:RTopIsomorphism}
		$\mathcal{R}^a:D^a_R\rightarrow C_{c,R}([0,\infty))$ is a topological isomorphism. In particular, its inverse is continuous.
	\end{lemma}
	\begin{proof}
		It follows from the definition of the norm $\|\cdot\|_{D^a}$ that $\mathcal{R}^a:D^a_R\rightarrow C_{c,R}([0,\infty))$ is continuous. $D^a_R$ is a Banach space by \cite[Lemma 2.7]{KnoerrSingularvaluationsHadwiger2022}  and $C_{c,R}([0,\infty))$ is a Banach space as well, so by Lemma \ref{lemma:InjectivityTransformR}, $\mathcal{R}^a$ establishes a bijective continuous map between Banach spaces. The claim follows from the open mapping theorem.
	\end{proof}

	Let $C_0\subset C_c([0,\infty))$ denote the subspace of all $\phi\in C_c([0,\infty))$ with $\phi(0)=0$. Denote by $C_{0,R}\subset C_0$ the subspace of all functions with support in $[0,R]$. Consider the map $\mathcal{P}^a:\tilde{D}^{a+2}\rightarrow C_0$ given by
	\begin{align*}
		\mathcal{P}(\tilde{\zeta})[t]=\begin{cases}
			t^{a+2}\tilde{\zeta}(t) & t>0,\\
			0 & t=0.
		\end{cases}
	\end{align*}
	The following is obvious.
	\begin{lemma}
		\label{lemma:PaIsomorphism}
			$\mathcal{P}^{a+2}:\tilde{D}_R^a\rightarrow C_{0,R}$ is a well defined topological isomorphism.
	\end{lemma}

	\subsection{The integral transform $\mathcal{R}^{a,b}$}
	
	Let $C_b((0,\infty)^2)$ denote the space of all continuous functions $\zeta: (0,\infty)^2\rightarrow\R$ with bounded support. For $a,b\in\mathbb{N}$, define $\mathcal{R}^{a,b}:C_b((0,\infty))\rightarrow C_b((0,\infty)^2)$ by
	\begin{align*}
		\mathcal{R}^{a,b}(\zeta)[s,t]=&s^{a+1}t^{b+1}\frac{d}{ds}\frac{d}{dt}\left(\frac{1}{s^at^b}\int_{t}^\infty\int_s^\infty\zeta(\sqrt{u^2+v^2})u^{a-1}v^{b-1}dudv\right) \\
		=&t^bs^a \zeta(\sqrt{s^2+t^2})+at^b\int_s^\infty\zeta(\sqrt{u^2+t^2})u^{a-1}du\\
		&+b\int_t^\infty v^{b-1}s^a\zeta(\sqrt{s^2+v^2})dv+ab\int_t^\infty\int_s^\infty \zeta(\sqrt{u^2+v^2})u^{a-1}v^{b-1}dudv
	\end{align*}
	for $s,t>0$.
	\begin{lemma}
		\label{lemma:ContinuityRab}
		For every $\zeta\in D_R^{a+b}$, $\mathcal{R}^{a,b}(\zeta)$ extends to a continuous function on $[0,\infty)^2$ with support contained in  $B_R(0)$. Moreover, the map $\mathcal{R}^{a,b}:D^{a+b}_R\rightarrow C_{c,R}([0,\infty)^2)$ is continuous.
	\end{lemma}
	\begin{proof}
		It is easy to see that $\mathcal{R}^{a,b}(\zeta)$ is supported on $B_R(0)$. We thus only have to show that $\mathcal{R}^{a,b}(\zeta)$ extends by continuity to a function in $C_c([0,\infty)^2)$. First note that $\mathcal{R}^{a,b}(\zeta)$ is well defined for $(s,t)\ne (0,0)$ and continuous on  $[0,\infty)^2\setminus\{(0,0)\}$. 
		Let us consider the limit $(s,t)\rightarrow (0,0)$. With the changes of coordinates $x=\frac{u}{t}$, $y=\frac{v}{s}$ for the middle terms, we obtain
		\begin{align}
			\label{eq:Rab}
			\begin{split}
				\mathcal{R}^{a,b}(\zeta)[s,t]=&t^bs^a \zeta(\sqrt{s^2+t^2})+a\int_\frac{s}{t}^\infty t^{a+b}\zeta(t\sqrt{x^2+1})x^{a-1}dx\\
				&+b\int_{\frac{t}{s}}^\infty s^{a+b}\zeta(s\sqrt{1+y^2})y^{b-1}dy+ab\int_t^\infty\int_s^\infty \zeta(\sqrt{u^2+v^2})u^{a-1}v^{b-1}dudv.
			\end{split}
		\end{align}
		Note that 
		\begin{align*}
			|t^bs^a\zeta(\sqrt{s^2+t^2})|\le\sqrt{s^2+t^2}^{a+b} |\zeta(\sqrt{s^2+t^2})|,
		\end{align*}
		which converges to $0$ for $(s,t)\rightarrow0$. Similarly, the integrand of the second term satisfies
		\begin{align*}
			a|t^{a+b}\zeta(t\sqrt{x^2+1})x^{a-1}|\le |(t\sqrt{x^2+1})^{a+b}\zeta(t\sqrt{x^2+1})|\cdot \frac{a|x|^{a-1}}{\sqrt{x^2+1}^{a+b}}\le \|\zeta\|_{D^{a+b}}\frac{a|x|^{a-1}}{\sqrt{x^2+1}^{a+b}},
		\end{align*}
		and similar for the integrand of the third term, where the function on the right are integrable for $a,b\in \mathbb{N}$. Dominated convergence thus implies that the first three terms in \eqref{eq:Rab} converge to zero for $(s,t)\rightarrow(0,0)$. Moreover, we obtain the estimate
		\begin{align}
			\label{eq:estimate1Rab}
			\begin{split}
			|\mathcal{R}^{a,b}(\zeta)[s,t]|\le &\|\zeta\|_{D^{a+b}}\left(1+\int_0^\infty\frac{ax^{a-1}+bx^{b-1}}{\sqrt{x^2+1}^{a+b}}dx\right)\\
			&+\left|ab\int_t^\infty\int_s^\infty \zeta(\sqrt{u^2+v^2})u^{a-1}v^{b-1}dudv\right|.
			\end{split}
		\end{align}
		If either $a=0$ or $b=0$, this implies the claim. If both are non-zero, we switch to polar coordinates in the last integral to obtain
		\begin{align*}
			&ab\int_t^\infty\int_s^\infty \zeta(\sqrt{r^2+u^2})r^{a-1}u^{b-1}drdu\\
			=&\frac{1}{\omega_a\omega_b}\int_{\R^b\setminus B^b_t(0)}\int_{\R^a\setminus B^a_s(0)}\zeta(\sqrt{|x_a|^2+|x_b|^2})d\vol_{a}(x_a)d\vol_{b}(x_b)
			=& \frac{1}{\omega_a\omega_b}\left(I_1(s,t)+I_2(s,t)\right),
		\end{align*}
		where $B^b_t(0)$ and $B^a_s(0)$ denote the balls of radius $t,s$ in $\R^b$, $\R^a$ centered at the origin and
		\begin{align*}
			I_1(s,t):=&\int\limits_{\R^{a+b}\setminus B^{a+b}_{\min(s,t)}(0)}\zeta(|x|)d\vol_{a+b}(x)=(a+b)\omega_{a+b}\int_{\min(s,t)}^\infty \zeta(r)r^{a+b-1}dr,\\
			I_2(s,t):=&\int\limits_{\R^b\setminus B^b_t(0)}\int\limits_{\R^a\setminus B^a_s(0)}\zeta(\sqrt{|x_a|^2+|x_b|^2})d\vol_{a}(x_a)d\vol_{b}(x_b)-I_1(s,t).
		\end{align*}
		In particular, $|I_1(s,t)|\le (a+b)\omega_{a+b}\|\zeta\|_{D^{a+b}}$. Assume that $s\le t$. Then
		\begin{align*}
			|I_2(s,t)|\le&\int_{B_s^a(0)\times (B_t^b(0)\setminus B_s^b(0))}|\zeta(|x|)|d\vol_{a+b}(x)+\int_{B_{\sqrt{2}s}^{a+b}(0)\setminus B_{s}^{a+b}(0)}|\zeta(|x|)|d\vol_{a+b}(x).
		\end{align*}
		For the first term, we obtain the estimate
		\begin{align*}
			&\int_{B_s^a(0)\times (B_t^b(0)\setminus B_s^b(0))}|\zeta(|x|)|d\vol_{a+b}(x)=a\omega_ab\omega_b\int_s^t\int_0^s |\zeta(\sqrt{u^2+v^2})|u^{a-1}v^{b-1}dudv\\
			\le &a\omega_ab\omega_b\left(\max_{r\in (0,\sqrt{t^2+s^2}]}|\zeta(r)|r^{a+b}\right)\int_s^t\int_0^s\frac{u^{a-1}v^{b-1}}{\sqrt{u^2+v^2}^{a+b}}dudv,
		\end{align*}
		where
		\begin{align*}
			&\int_s^t\int_0^s\frac{u^{a-1}v^{b-1}}{\sqrt{u^2+v^2}^{a+b}}dudv\le \int_s^t\int_0^s\frac{1}{u^2+v^2}dudv\le\int_s^t\frac{s}{v^2}dv=s\left(\frac{1}{s}-\frac{1}{t}\right)\le 1,
		\end{align*}
		because $s\le t$. Similarly,
		\begin{align*}
			&\int_{B_{\sqrt{2}s}^{a+b}(0)\setminus B_{s}^{a+b}(0)}|\zeta(|x|)|d\vol_{a+b}(x)=(a+b)\omega_{a+b}\int_{s}^{\sqrt{2}s}|\zeta(r)|r^{a+b-1}dr\\
			\le &(a+b)\omega_{a+b}\sup_{r\in (0,\sqrt{2}s)}|\zeta(r)r^{a+b}\int_s^{\sqrt{2}s}\frac{1}{u}du\le(a+b)\omega_{a+b}\sup_{r\in (0,\sqrt{s^2+t^2})}|\zeta(r)|r^{a+b}\ln(\sqrt{2}).
		\end{align*}
		If $s\ge r$, we obtain similar estimates, so in total we see that
		\begin{align}
			\label{eq:RabIntegral2}
			|I_2(s,t)|\le (a\omega_ab\omega_b+(a+b)\omega_{a+b}\ln 2)\sup_{r\in (0,\sqrt{s^2+t^2})}|\zeta(r)|r^{a+b},
		\end{align}
		which converges to $0$ for $(s,t)\rightarrow0$, as $r\mapsto \zeta(r)r^{a+b}$ extends to a continuous function on $[0,\infty)$ that vanishes in $r=0$, which is therefore locally uniformly continuous. In particular, the limit $\lim_{(s,t)\rightarrow0}\mathcal{R}^{a,b}(\zeta)[s,t]$ exists if and only if the limit
		\begin{align*}
			\lim\limits_{(s,t)\rightarrow(0,0)}I_{1}(s,t)
			=&\lim\limits_{r\rightarrow0}(a+b)\omega_{a+b}\int_{r}^\infty\zeta(u)u^{a+b-1}du
		\end{align*}
		exists, which is the case as $\zeta\in D^{a+b}$. Since $|I_1(s,t)|$ is bounded by a multiple of $\|\zeta\|_{D^{a+b}}$, \eqref{eq:RabIntegral2} implies the estimate
		\begin{align*}
				\left| ab\int_t^\infty\int_s^\infty \zeta(\sqrt{u^2+v^2})u^{a-1}v^{b-1}dudv\right|
				\le \frac{|I_1(s,t)|+|I_2(s,t)|}{\omega_a\omega_b}\le C\|\zeta\|_{D^{a+b}}
		\end{align*}
		for some constant $C>0$ depending on $a,b$ only. Combining this estimate with \eqref{eq:estimate1Rab} , we see that $\mathcal{R}^{a,b}$ satisfies
		\begin{align*}
			\mathcal{R}^{a,b}(\zeta)[s,t]|\le&\tilde{C}\|\zeta\|_{D^{a+b}}
		\end{align*}
		with a constant $\tilde{C}$ depending on $a,b$ only. Thus $\mathcal{R}^{a,b}$ is continuous.
	\end{proof}
	Next, we are going to construct the inverse of $\mathcal{R}^{a,b}$. Consider the map $
	\tilde{\mathcal{S}}^{a,b}:C_b((0,\infty)^2)\rightarrow C_b((0,\infty)^2)$ given by
	\begin{align}
		\label{eq:DefTildeSabDerivative}
		\tilde{\mathcal{S}}^{a,b}(\phi)[s,t]
		=&\frac{1}{s^{a-1}t^{b-1}}\frac{d}{ds}\frac{d}{dt}\left( s^at^b\int_t^\infty \int_s^\infty\frac{\phi(u,v)}{u^{a+1}v^{b+1}}dudv\right) \\
		\notag
		=&\frac{\phi(s,t)}{s^{a}t^{b}}+\frac{a}{t^{b}}\int_s^\infty\frac{\phi(u,t)}{u^{a+1}}du+\frac{b}{t^{a}}\int_t^\infty\frac{\phi(s,v)}{v^{b+1}}dv
		+\frac{ab}{s^{a}b^{b}}\int_t^\infty\int_s^\infty\frac{\phi(u,v)}{u^{a+1}v^{b+1}}dudv	
	\end{align}
	for $s,t>0$. The following result follows directly from the definition.
	\begin{corollary}
		\label{corollary:relationTildeSabRab}
		For $\zeta\in C_b((0,\infty))$, $\tilde{\mathcal{S}}^{a,b}\circ \mathcal{R}^{a,b}(\zeta)[s,t]=\zeta(\sqrt{s^2+t^2})$.
	\end{corollary}
	Let $C_{(a,b)}$ denote the space of all $\phi\in C_c([0,\infty)^2)$ such that $\tilde{\mathcal{S}}^{a,b}(\phi)[s,t]$ only depends on $s^2+t^2$. We denote by $C_{(a,b),R}$ the subspace of functions supported on $B_R(0)$. 
	The previous corollary and Lemma \ref{lemma:ContinuityRab} imply the following.
	\begin{corollary}
		For $\zeta\in D^{a+b}_R$, $\mathcal{R}^{a,b}(\zeta)\in C_{(a,b),R}$.
	\end{corollary}	
	Consider the map $\mathcal{S}^{a,b}:C_c([0,\infty)^2)\rightarrow C_b((0,\infty))$ given by
	\begin{align*}
		\mathcal{S}^{a,b}(\phi)[t]=\tilde{\mathcal{S}}^{a,b}(\phi)\left[\frac{t}{\sqrt{2}},\frac{t}{\sqrt{2}}\right].
	\end{align*}
	\begin{theorem}
		$\mathcal{S}^{a,b}:C_{(a,b),R}\rightarrow D^{a+b}_R$ is well defined and continuous.
	\end{theorem}
	\begin{proof}
		Let $\phi\in C_{(a,b),R}$. It is easy to see that the support of $\mathcal{S}^{a,b}(\phi)$ is contained in $(0,R]$, so it suffices to show that this function belongs to $D^{a+b}$. With the changes of coordinates $u\mapsto \frac{tu}{\sqrt{2}}$, $v\mapsto \frac{tv}{\sqrt{2}}$, we obtain
		\begin{align*}
		\frac{t^{a+b}}{\sqrt{2}^{a+b}}\mathcal{S}^{a,b}(\phi)[t]=&\phi\left(\frac{t}{\sqrt{2}},\frac{t}{\sqrt{2}}\right)+a\int_1^\infty\frac{\phi\left(\frac{tu}{\sqrt{2}},\frac{t}{\sqrt{2}}\right)}{u^{a+1}}du\\
		&+b\int_1^\infty\frac{\phi\left(\frac{t}{\sqrt{2}},\frac{tv}{\sqrt{2}}\right)}{v^{b+1}}dv
		+ab\int_1^\infty\int_1^\infty\frac{\phi\left(\frac{tu}{\sqrt{2}},\frac{tv}{\sqrt{2}}\right)}{u^{a+1}v^{b+1}}dudv.
		\end{align*}
		As $u\mapsto\frac{1}{x^{c+1}}$ is integrable on $[1,\infty)$ for $c>0$ and $\phi$ is bounded, dominated convergence implies
		\begin{align}
			\label{eq:SabInD_DecayCondition}
			\lim\limits_{t\rightarrow0}t^{a+b}\mathcal{S}^{a,b}(\phi)[t]=0.
		\end{align}
		By estimating $\phi$ by $\|\phi\|_\infty$, we also obtain the estimate
		\begin{align}
			\label{eq:continuitySLimit}
			|t^{a+b}\mathcal{S}^{a,b}(\phi)[t]|\le 4\sqrt{2}^{a+b}\|\phi\|_\infty.
		\end{align}
		Let us now consider the integral
		\begin{align}
			\label{eq:SabInD_IntegralCondition}
			\int_t^\infty \mathcal{S}^{a,b}(\phi)[r]r^{a+b-1}dr=&\int_t^\infty\tilde{\mathcal{S}}^{a,b}(\phi)\left[\frac{r}{\sqrt{2}},\frac{r}{\sqrt{2}}\right]r^{a+1}dr\\
			\notag
			=&\frac{1}{(a+b)\omega_{a+b}}\int_{\R^{a+b}\setminus B_t^{a+b}(0)}\tilde{\mathcal{S}}^{a,b}(\phi)\left[\frac{|x|}{\sqrt{2}},\frac{|x|}{\sqrt{2}}\right]d\vol_{a+b}(x),
		\end{align}
		where we changed to polar coordinates in the second step. As $\phi\in C_{(a,b)}$, the integrand is rotation invariant, so we obtain
		\begin{align*}
			(a+b)\omega_{a+b}\int_t^\infty \mathcal{S}^{a,b}(\phi)[r]r^{a+b-1}dr=&\int_{\R^{a+b}\setminus B_t^{a+b}(0)}\tilde{\mathcal{S}}^{a,b}(\phi)\left[|x_1|,|x_2|\right]d\vol_a(x_1)d\vol_{b}(x_2)\\
			=&I_1(t)+I_2(t),
		\end{align*}
		where 
		\begin{align*}
			I_1(t)=\int_{\R^{b}\setminus B_t^{b}(0) }\int_{\R^a\setminus B^a_t(0) }\tilde{\mathcal{S}}^{a,b}(\phi)\left[|x_1|,|x_2|\right]d\vol_a(x_1)d\vol_{b}(x_2),\\
			I_2(t)=\int_{\R^{a+2}\setminus B_t^{a+b}(0)}\tilde{\mathcal{S}}^{a,b}(\phi)\left[|x_1|,|x_2|\right]d\vol_a(x_1)d\vol_{b}(x_2)-I_1(t).
		\end{align*}
		First,
		\begin{align*}
			|I_2(t)|\le& \int_{ B_{\sqrt{2}}^{a+b}(0)\setminus B_t^{a+b}(0)}|\tilde{\mathcal{S}}^{a,b}(\phi)\left[|x_1|,|x_2|\right]|d\vol_a(x_1)d\vol_{b}(x_2)\\
			=&(a+b)\omega_{a+b}\int_t^{\sqrt{2}t}\left|\tilde{\mathcal{S}}^{a,b}(\phi)\left[\frac{r}{\sqrt{2}},\frac{r}{\sqrt{2}}\right]\right|r^{a+b-1}dr,
		\end{align*}
		where we have used that $\tilde{\mathcal{S}}^{a,b}(\phi)$ only depends on the norm of its argument. Thus
		\begin{align}
				\label{eq:estimateI2}
				\begin{split}
						|I_2(t)|\le&(a+b)\omega_{a+b}\sup_{r\in (0,t)}\left|r^{a+b}\mathcal{S}^{a,b}(\phi)[r]\right|\int_t^{\sqrt{2}t}\frac{1}{r}dr\\
					=&\ln\sqrt{2}(a+b)\omega_{a+b}\sup_{r\in (0,t)}\left|r^{a+b}\mathcal{S}^{a,b}(\phi)[r]\right|,
				\end{split}
		\end{align}
		which converges to $0$ for $t\rightarrow0$ by \eqref{eq:SabInD_DecayCondition}. Let us turn to $I_1(t)$. As $(x_1,x_2)\mapsto\tilde{\mathcal{S}}(\phi)[|x_1,|x_2|]$ is rotation invariant, a change to polar coordinates implies
		\begin{align*}
			I_1(t)=&a\omega_{a}b\omega_{b}\int_t^\infty\int_t^\infty \tilde{\mathcal{S}}^{a,b}(\phi)\left[s,r\right]s^{a-1}ds r^{b-1}dt
			=a\omega_{a}b\omega_{b}t^{a+b}\int_t^\infty\int_t^\infty \frac{\phi(u,v)}{u^{a+1}v^{b+1}}dudv,
		\end{align*}
		where we used \eqref{eq:DefTildeSabDerivative}. Changing coordinates, we thus obtain
		\begin{align*}
			I_1(t)=a\omega_{a}b\omega_{b}\int_1^\infty\int_1^\infty \frac{\phi(tu,tv)}{u^{a+1}v^{b+1}}dudv.
		\end{align*}
		Dominated convergence implies $\lim_{r\rightarrow0} I_1(t)=\omega_a\omega_b\phi(0)$. In particular, the limit $t\rightarrow0$ in \eqref{eq:SabInD_IntegralCondition} exists and is finite, which together with \eqref{eq:SabInD_DecayCondition} implies $\mathcal{S}^{a,b}(\phi)\in D^{a+b}$. Moreover, we obtain the estimate
		\begin{align*}
			|I_1(t)|\le \omega_a\omega_b \|\phi\|_\infty.
		\end{align*}
		Combining this estimate with \eqref{eq:continuitySLimit} and \eqref{eq:estimateI2}, we obtain for $\phi\in C_{(a,b)}$ the estimate
		\begin{align*}
			\|\mathcal{S}^{a,b}(\phi)\|_{D^{a+b}}\le C\|\phi\|_\infty
		\end{align*}
		for a constant $C>0$ depending on $a,b$ only. Thus $\mathcal{S}^{a,b}$ is continuous.
	\end{proof}

	\begin{corollary}
		\label{corollary:relationSabRab}
		$\mathcal{R}^{a,b}:D^{a+b}_R\rightarrow C_{(a,b),R}$ and $\mathcal{S}^{a,b}:C_{(a,b),R}\rightarrow D^{a+b}_R $ are mutual inverses.
	\end{corollary}
	\begin{proof}
		It is easy to see that $\tilde{\mathcal{S}}^{a,b}:C_{(a,b),R}\rightarrow C_b((0,\infty)^2)$ is injective. Since $\tilde{\mathcal{S}}(\phi)[s,t]$ only depends on $s^2+t^2$ for $\phi\in C_{(a,b),R}$, $\mathcal{S}^{a,b}$ must also be injective.	It follows from Corollary \ref{corollary:relationTildeSabRab} that $\mathcal{S}^{a,b}\circ \mathcal{R}^{a,b}(\zeta)=\zeta$ for all $\zeta\in D^{a+b}_R$. Hence, $\mathcal{S}^{a,b}:C_{(a,b),R}\rightarrow D^{a+b}_R$ is also surjective. The claim follows.
	\end{proof}

\section{$\U(n)$-equivariant Monge-Amp\`ere operators}	
	\label{section:MAOperators}
	\subsection{Monge-Amp\`ere operators defined in terms of the differential cycle}
	As shown in \cite{KnoerrMongeAmpereoperators2024}, the differential cycle may be used to define a large class of Monge-Amp\`ere-type operators on $\Conv(\R^n,\R)$. In this section we extend some these results and obtain some general bounds for the variation of these functionals in terms of the Hessian measures. The constructions rely on the following result.
	\begin{theorem}[\cite{KnoerrMongeAmpereoperators2024} Theorem 4.10.]
		\label{theorem:MAOperatorsDefinedDifferentialCycle}
		Let $\tau\in\Omega^n(T^*\R^n)$ and define $\Psi_\tau(f)\in\mathcal{M}(\R^n)$ for $f\in\Conv(\R^n)$ by 
		\begin{align*}
			\Psi_\tau(f)[B]:=D(f)[1_{\pi^{-1}(B)}\tau]\quad\text{for all bounded Borel sets } B\subset\R^n.
		\end{align*}
		Then $\Psi_\tau:\Conv(\R^n,\R)\rightarrow\mathcal{M}(\R^n)$ is a continuous valuation. If $\tau$ is invariant with respect to translations in the second factor of $T^*\R^n=\R^n\times(\R^n)^*$, then $\Psi_\tau$ is dually epi-translation invariant. If $\tau$ is invariant with respect to translations in the first factor, then $\Psi_\tau$ is translation equivariant.
	\end{theorem}
	A special case of this construction yields the \emph{Hessian measures} $\Phi_k$, $0\le k\le n$, mentioned in the introduction. These satisfy $\Phi_k=\binom{n}{k}^{-1}\Psi_{\tilde{\kappa}_k}$  for the differential form
	\begin{align*}
		\tilde{\kappa}_k:=\frac{1}{k!(n-k)!}\sum\limits_{\sigma\in S_{n}}\sign(\sigma)dx_{\sigma(1)}\dots dx_{\sigma(n-k)}\wedge dy_{\sigma(n-k+1)}\dots dy_{\sigma(n)}, \quad 0\le k\le n,
	\end{align*} 
	compare \cite[Proposition 5.3]{KnoerrSingularvaluationsHadwiger2022}.\\

	For $f\in\Conv(\R^n,\R)\cap C^2(\R^n)$ consider the map
	\begin{align*}
		G_f:\R^n&\rightarrow T^*\R^n\\
		x&\mapsto (x,df(x)).
	\end{align*}
	
	\begin{definition}
		We call $\tau\in \Omega^n(T^*\R^n)$ a \emph{positive} form if the $n$-form $G_f^*\tau\in\Omega^n(\R^n)$ is a non-negative multiple of the volume form for every $f\in\Conv(\R^n,\R)\cap C^2(\R^n)$. 
	\end{definition}
	Since $\Psi_\tau$ is continuous with respect to the weak*-topology by Theorem \ref{theorem:MAOperatorsDefinedDifferentialCycle}, we directly obtain the following.
	\begin{corollary}
		$\tau\in \Omega^n(T^*\R^n)$ is a positive form if and only if $\Psi_\tau(f)$ is a positive measure for every $f\in\Conv(\R^n,\R)$.
	\end{corollary}
	
	Let $E\in \Gr_k(\R^n)$ be a subspace. We call a positively oriented orthonormal basis $e_1,\dots,e_n$ adapted to $E$ if $e_1,\dots, e_k$ span $E$. For any oriented orthonormal basis $e_1,\dots,e_n$ of $\R^n$ we set $\bar{e}_i=(e_i,0)$, $\epsilon_i=(0,e_i)$, and consider these vectors as a basis of $\R^n\times \R^n\cong \R^n\times (\R^n)^*$.
	
	\begin{lemma}
		\label{lemma:CharacterizationPositiveForms}
		$\tau\in \Omega^{n-k,k}(T^*\R^n)^*$ is a positive form if and only if for every $(x,\xi)\in T^*\R^n$ and every subspace $E\in \Gr_k(\R^n)$ there exists a positively oriented orthonormal basis $e_1,\dots,e_n$ of $\R^n$ adapted to $E$ such that
		\begin{align*}
			\tau|_{(x,\xi)}(\bar{e}_1,\dots,\bar{e}_{n-k},\epsilon_{n-k+1},\dots,\epsilon_n)\ge 0.
		\end{align*}
	\end{lemma}
	\begin{proof}
		Let $(x_0,\xi_0)\in T^*\R^n$. If $\tau$ is a positive form and $E\in \Gr_k(\R^n)$, we choose a positively oriented basis $e_1,\dots,e_n$ adapted to $E$. Let $(x_{0,1},\dots, x_{0,n})$, $(\xi_{i_{0,1}},\dots, \xi_{0,n})$ denote the coordinates of $x_0$ and $\xi$ with respect to this basis. Set $f(x):=\sum_{i=1}^{k}\lambda_i x_i^2 +\xi_{0,i}x_i-2\lambda_ix_{0,i}$ for $\lambda_i\ge 0$ for $1\le i\le k$. Then it is easy to see that
		\begin{align*}
			G_f^*\tau|_{x_0}= k!\lambda_1\dots\lambda_k\tau|_{(x_0,\xi_0)}(\bar{e}_1,\dots,\bar{e}_{n-k},\epsilon_{n-k+1},\dots,\epsilon_n)\vol_n,
		\end{align*}
		which implies the claim.\\
		
		Conversely, let us assume that $\tau$ satisfies the second property, and note that it does not depend on the specific choice of the positively oriented orthonormal basis.\\		
		Let $f\in\Conv(\R^n,\R)\cap C^2(\R^n)$. Then $G_f^* dx_i= dx_i$ and 
		\begin{align*}
			G_f^* dy_i=d \left(\frac{\partial f(x)}{\partial x_i}\right)=\sum_{j=1}^{n}\frac{\partial^2 f(x)}{\partial x_i\partial x_j}dx_j.
		\end{align*}
		In particular, for any $x\in \R^n$, $(G_f^*\tau)|_x$ only depends on the Hessian of $f$ in $x$ (as well as $x$ and $df(x)$). It is thus sufficient to assume that $\tau$ is a constant differential form and to consider $f_A(x)=\langle x,Ax\rangle$, where $A$ is a symmetric and positive semi-definite $(n\times n)$-matrix. We may choose a positively oriented orthonormal basis $e_1,\dots,e_n$ such that $A$ is diagonal with respect to this basis, that is, $Ae_i=\lambda_ie_i$ for $\lambda_i\ge 0$. Note that $(G_{f_A})_*e_i=\bar{e}_i+\lambda_i\epsilon_i$. Thus
		\begin{align*}
			&G_{f_A}^*\tau (e_1,\dots,e_n)\\
			=&\tau (\bar{e}_1+\lambda_1\epsilon_1,\dots, \bar{e}_n+\lambda_n\epsilon_n)\\
			=&\frac{1}{k!(n-k)!}\sum_{\sigma\in S_n}\mathrm{sign}(\sigma)\lambda_{\sigma(n-k+1)}\dots\lambda_{\sigma(n)}\tau (\bar{e}_{\sigma(1)},\dots,\bar{e}_{\sigma(n-k)},\epsilon_{\sigma(n-k+1)},\dots,\epsilon_{\sigma(n)}),
		\end{align*}
		where we used that $\tau$ is a $(n-k,k)$-form. Since $\mathrm{sign}(\sigma)e_{\sigma(1)},e_{\sigma(1)},\dots,e_{\sigma(n)}$ is a positively oriented basis adapted to $E_\sigma:=\mathrm{span} (e_{\sigma(1)},e_{\sigma(1)},\dots,e_{\sigma(k)})$, all terms in this sum are non-negative by assumption. Thus $G_{f_A}^*\tau$ is non-negative.
	\end{proof}
	
	\begin{proposition}
		\label{proposition:GeneralBoundMA}
		Let $\omega\in \Omega^{n-k}(\R^n)\otimes \Lambda^k((\R^n)^*)^*$ be $(n-k+l)$-homogeneous, $l>0$. Then there exists $C(\omega)>0$ such that for $f\in\Conv(\R^n,\R)$ and all Borel subsets $U\subset\R^n$,
		\begin{align*}
			|\Psi_\omega(f)|(U)\le C(\omega)\int_U|x|^l d\Phi_k(f).
		\end{align*}	
	\end{proposition}
	\begin{proof}
		As in the proof of Proposition \ref{proposition:BoundCurvatureMeasure}, it is sufficient to find $c(\omega)>0$ such that $c(\omega)|x|^l\tilde{\kappa}_k\pm\omega $ defines a positive form, since $\tilde{\kappa}_k$ induces a scalar multiple of the Hessian measure $\Phi_k$. Because the map
		\begin{align*}
			\SO(n)\times S^{n-1}&\rightarrow \R\\
			((e_1,\dots,e_n),x)&\mapsto| \omega|_{(x,0)}(\bar{e}_1,\dots,\bar{e}_{n-k},\epsilon_{n-k+1},\dots,\epsilon_n)|
		\end{align*}
		is continuous and $\SO(n)\times S^{n-1}$ is compact, this map is bounded by some $c(\omega)>0$. Since
		\begin{align*}
			\tilde{\kappa}_k|_{(x,0)}(\bar{e}_1,\dots,\bar{e}_{n-k},\epsilon_{n-k+1},\dots,\epsilon_n)=1,
		\end{align*}
		the form $c(\omega)|x|^l\tilde{\kappa}_k\pm \omega$ is positive on $S^{n-1}\times (\R^n)^*$ as it is translation invariant in the second factor. Since it is homogeneous of degree $n-k+l>0$, it is thus positive on $T^*\R^n$. Here we use that $\omega|_{(0,0)}=0$ since it is $(n-k+l)$-homogeneous.
	\end{proof}

	\begin{corollary}\label{corollary:GeneralIntegrabilityMA}
		Let $0\le k\le n-1$ and $\omega\in \Omega^{n-k}\otimes\Lambda^k((\R^n)^*)^*$ be homogeneous of degree $n-k+l$, $l\ge 0$. Then there exists a constant $A(\omega)$ such that for every $R>0$, for every Baire function $\phi:[0,\infty)\rightarrow[0,\infty]$ with $\supp\phi\subset[0,R]$, and every $f\in\Conv(\R^n,\R)$,
		\begin{align*}
			\int_{\R^n}\phi(|x|) d|\Psi_\omega(f)|\le A(\omega) \left(\sup_{|x|\le R+1}|f(x)|\right)^k \| \phi\|_{D^{n-k+l}}
		\end{align*}
		In particular, $\{0\}$ is a set of $\Psi_\omega(f)$-measure $0$. Moreover, if $\zeta\in C_b((0,\infty))$ satisfies $|\zeta|\in D^{n-k+l}$, then $x\mapsto \zeta(|x|)$ is integrable with respect to $\Psi_\omega(f)$ for every $f\in\Conv(\R^n,\R)$.
	\end{corollary}
	\begin{proof}
		From the definition of the norm $\|\cdot\|_{D^{2n-k}}$, it easily follows from monotone convergence that it is sufficient to establish the inequality for bounded Baire functions. By using a continuous cut-off function, we see that every bounded Baire function with support in $[0,R]$ is the pointwise limit of a sequence of continuous functions with support in $[0,R+\epsilon]$ for arbitrary $\epsilon>0$. It is thus sufficient to consider continuous functions. In this case, Proposition \ref{proposition:GeneralBoundMA} implies
		\begin{align*}
			\int_{\R^n}\phi(|x|) d|\Psi_\omega(f)|\le C(\omega)\int_{\R^n}|x|^l\phi(|x|)d\Phi_k(f)\le C(\omega) C\left(\sup_{|x|\le R+1}|f(x)|\right)^k \| \psi\|_{D^{n-k}}
		\end{align*}
		by \cite[Proposition 6.6]{KnoerrSingularvaluationsHadwiger2022}, where $C>0$ is a constant depending on $n,k$ only and $\psi(t)=t^l\phi(t)$. The inequality follows since $\|\psi\|_{D^{n-k}}=\|\phi\|_{D^{n-k+l}}$. The two other claims are a simple consequence of this result.
	\end{proof}
	
	\begin{remark}
		For $k=n$ all such differential forms induce measures that are absolutely continuous with respect to the real Monge-Amp\`ere operator. In this case, discrete sets can have non-zero measure. 
	\end{remark}

	\subsection{The Monge-Amp\`ere operators $\Theta^n_{k,q}$, $\mathcal{B}^n_{k,q}$, $\mathcal{C}^n_{k,q}$, and $\Upsilon^n_{k,q}$}
		Let $z_1=x_1+iy_1,\dots, z_n=x_n+iy_n$ be the standard coordinates on $\C^n$, where $x_1,y_1,\dots,x_n,y_n$ are real coordinates. We consider $T^*\C^n=\C^n\times (\C^n)^*$ with coordinates $(z_1,\dots,z_n,w_1,\dots,w_n)$, where $w_1=\xi_1+i\eta_1,\dots, w_n=\xi_n+i\eta_n$ are the induced coordinates on $(\C^n)^*$. Consider the following differential forms
	\begin{align*}
		\beta_1=&\sum_{j=1}^{n}x_jd\xi_j+y_jd\eta_j, &&\gamma_1=\sum_{j=1}^{n}x_jdx_j+y_jdy_j,\\
		\beta_2 =&\sum_{j=1}^{n}x_jd\eta_j-y_jd\xi_j, &&\gamma_2=\sum_{j=1}^{n}x_jdy_j-y_jdx_j,\\
		\theta_0=& \sum_{j=1}^{n}dx_j\wedge dy_j, && \theta_1=\sum_{j=1}^{n}dx_j\wedge d\eta_j-dy_j\wedge d\xi_j,\\
		\theta_2=&\sum_{j=1}^{n}d\xi_j\wedge d\eta_j, &&  \omega_s=\sum_{j=1}^{n}dx_j\wedge d\xi_j+dy_j\wedge d\eta_j.
	\end{align*} 
	Note that the forms $\beta$ and $\gamma$ from Section \ref{section:DiffFormsCotangentBundel} correspond to $\beta_1$ and $\gamma_1$. Let us further note that $\theta_0$ is the pullback of the natural symplectic form on $\C^n$, whereas $\omega_s$ is the symplectic form on $T^*\C^n$. We also note the following relations:
	\begin{align*}
		d\beta_1=&\omega_s, && d\gamma_1=0,\\
		d\beta_2=& \theta_1, && d\gamma_2=2\theta_0.
	\end{align*}
	We will set $\theta^n_{k,q}:=\theta_0^{n-k+q}\wedge \theta_1^{k-2q}\wedge \theta_2^q$ whenever this form is well-defined.
	\begin{definition}
		\label{definition:MAOperatirsThetaUpsilonBC}
		For $0\le k\le 2n$, we define  $\Theta^n_{k,q},\mathcal{B}^n_{k,q},\mathcal{C}^n_{k,q},\Upsilon^n_{k,q}:\Conv(\C^n,\R)\rightarrow\mathcal{M}(\C^n)$ by
		\begin{align*}
			\Theta^n_{k,q}:=&c_{n,k,q}\Psi_{\theta^n_{k,q}}&&\text{for}~\max(0,k-n)\le q\le \left\lfloor\frac{k}{2}\right\rfloor\\
			\mathcal{B}^n_{k,q}:=&c_{n,k,q}\Psi_{\beta_1\wedge\beta_2\wedge \theta^{n-1}_{k-2,q-1}} &&\text{for}~2\le k\le 2n-1,~\max(1,k-n)\le q\le \left\lfloor\frac{k}{2}\right\rfloor\\
			\mathcal{C}^n_{k,q}:=&\frac{c_{n,k,q}}{2}\Psi_{\beta_1\wedge\gamma_2\wedge \theta^{n-1}_{k-1,q}} &&\text{for}~1\le k\le 2n-1,~\max(0,k-n)\le q\le \left\lfloor\frac{k-1}{2}\right\rfloor\\
			\Upsilon^n_{k,q}:=&\mathcal{B}^n_{k,q}-\mathcal{C}^n_{k,q}&& \text{for}~2\le k\le 2n-1,~\max(1,k-n)\le q\le \left\lfloor\frac{k-1}{2}\right\rfloor,
		\end{align*}
		where $c_{n,k,q}:=\frac{1}{(n-k+q)!(k-2q)!q!}$.
	\end{definition}

	\begin{corollary}
		\label{corollary:ThetaPositivemeasure}
		$\Theta^{n}_{k,q}(f)$ is a positive measure for every $f\in\Conv(\C^n,\R)$.
	\end{corollary}
	\begin{proof}
		Let $E\in\Gr_k(\C^n)$ and $e_1,\dots,e_{2n}$ be a positively oriented basis of $\C^n$ adapted to $E$. It follows from the discussion in \cite[Section 2.7]{BernigFuHermitianintegralgeometry2011} that $\theta^n_{{k,q}}(\bar{e}_1,\dots,\bar{e}_{n-k},\epsilon_{n-k+1},\dots,\epsilon_n)$ is a positive multiple of the Klain function $\mathrm{Kl}_{\mu_{k,q}}(E)$ of the Hermitian intrinsic volume $\mu_{k,q}$ evaluated in $E$. Since this function is positive by \cite[Proposition 4.1]{BernigFuHermitianintegralgeometry2011}, the claim follows from Lemma \ref{lemma:CharacterizationPositiveForms}.
	\end{proof}
	\begin{remark}
		In general, $\mathcal{B}^n_{k,q}(f)$ and $\mathcal{C}^n_{k,q}(f)$ are neither positive nor negative measures. The same holds true for $\Upsilon^n_{k,q}(f)$. This follows from a simple calculation using Lemma \ref{lemma:CharacterizationPositiveForms}.
	\end{remark}
	
	\begin{proposition}
		Let $f\in \Conv(E_{k,p},\R)$ . Then for $q\ne p$
		\begin{align*}
			\Theta^n_{k,q}(\pi_{E_{k,p}}^*f)=\mathcal{B}^n_{k,q}(\pi_{E_{k,p}}^*f)=\mathcal{C}^n_{k,q}(\pi_{E_{k,p}}^*f)=\Upsilon^n_{k,q}(\pi_{E_{k,p}}^*f)=0.
		\end{align*}
	\end{proposition}
	\begin{proof}
		It was shown in \cite[Corollary 4.11]{KnoerrUnitarilyinvariantvaluations2021} that the integral of any element of $C_c(\C^n)$ with respect to the first three measures vanishes, so these measures have to vanish identically. Since $\Upsilon^n_{k,q}$ is a linear combination of $\mathcal{B}^n_{k,q}$ and $\mathcal{C}^n_{k,q}$, it has the same property.
	\end{proof}

\subsection{Behavior under complex orthogonal decompositions}
	let $\C^n=X\oplus Y$ be a complex orthogonal direct sum decomposition. Identifying $X$ and $Y$ with $\C^a$ using a complex orthonormal basis, we may define the corresponding Monge-Amp\`ere-type operators
	\begin{align*}
		\Theta^X_{k,q},\Upsilon^X_{k,q}:\Conv(X,\R)\rightarrow\mathcal{M}(X),\\
		\Theta^Y_{k,q},\Upsilon^Y_{k,q}:\Conv(Y,\R)\rightarrow\mathcal{M}(Y).
	\end{align*} 
	In order to simplify the notation, we will suppress the pullback by the projections $\pi_X,\pi_Y$ onto the two factors, that is, we will $f_X\in\Conv(X,\R)$ and $f_Y\in\Conv(Y,\R)$ as a function on $\C^n$ using the constant extension.
		\begin{lemma}
			\label{lemma:DecompositionMeasuresDirectSum}
			Assume that $\C^n=X\oplus Y$ is an complex orthogonal direct sum decomposition. For $f_X\in \Conv(X,\R)$, $f_Y\in \Conv(Y,\R)$,
			\begin{align*}
				\Theta^n_{k,q}(f_X+f_Y)=&\sum_{\substack{m_1+m_2=k\\ p_1+p_2=q}}\Theta^{ X}_{m_1,p_1}(f_X)\otimes \Theta^{Y}_{m_1,p_2}(f_Y),\\
				\Upsilon^n_{k,q}(f_X+f_Y)=&\sum_{\substack{m_1+m_2=k\\ p_1+p_2=q}}\left(\frac{p_1}{q}\mathcal{B}^{X}_{m_1,p_1}(f_X)-\frac{m_1-2p_1}{k-2q}\mathcal{C}^X_{m_1,p_1}(f_X)\right)\otimes \Theta^{Y}_{m_2,p_2}(f_Y)\\
				&+\sum_{\substack{m_1+m_2=k\\ p_1+p_2=q}}\Theta^{ X}_{m_1,p_1}(f_X)\otimes\left(\frac{p_2}{q}\mathcal{B}^{Y}_{m_2,p_2}(f_Y)-\frac{m_2-2p_2}{k-2q}\mathcal{C}^Y_{m_2,p_2}(f_Y)\right).
			\end{align*}
		\end{lemma}
		\begin{proof}
			It is enough to compare the integrals of the product of $\phi_1\in C^\infty_c(X)$ and $\phi_2\in C^\infty_c(Y)$ for both sides. Write $z=(z_X,z_Y)\in X\oplus Y$. By Proposition \ref{proposition:DifferentialCycleSumMAOrthogonalCylinderFunctions}
			\begin{align*}
				\int_{\C^n}\phi_1(z_X)\phi_2(z_Y)d\Theta^n_{k,q}(f_X+f_Y)=&c_{n,k,q}D(f_X+f_Y)[\phi_1(z_X)\phi_2(z_Y)\theta^{n}_{k,q}]\\
				=&c_{n,k,q}\left(D(f_X)\times D(f_Y)\right)[\phi_1(z_X)\phi_2(z_Y)G^*\theta^{n}_{k,q}]
			\end{align*}
			and similar for $\Upsilon^n_{k,q}(f_X+f_Y)$. It is easy to see that the pullback of the forms $\beta_1,\beta_2,\gamma_1,\gamma_2,\theta_0,\theta_1,\theta_2$ by $G$ is given by the sum of the corresponding forms on $T^*X$, $T^*Y$ respectively. Rearranging these terms, we obtain a product of two differential forms on $T^*X$ and $T^*Y$ corresponding to scalar multiples of the Monge-Amp\`ere operators on the right hand side of the equations above. The claim follows from the defining property of the product of two currents.
		\end{proof}
		
	\subsection{Behavior on smooth rotation invariant functions}
		For the following calculation, it is convenient to use the complex differential operators $\partial$ and $\bar{\partial}$ on $\C^n$ induced by the complex structure. We refer to \cite{HuybrechtsComplexgeometry2005} for their basic properties. The following lemma follows from a simple calculation.
		\begin{lemma}
			\label{lemma:pullbackFormsGraph}
			Let $f\in \Conv(\C^n),\R)$ be a convex function given by $f(z)=h(|z|^2)$ for $h\in C^\infty(\R)$. Then
			\begin{align*}
				G_{f}^*\beta_1=&\left[2|z|^2h''(|z|^2)+h'(|z|^2)\right]d|z|^2,&&
				G_{f}^*\beta_2=-ih'(|z|^2)(\partial|z|^2-\bar{\partial}|z|^2),\\
				G_{f}^*\gamma_1=&d|z|^2,&&
				G_{f}^*\gamma_2=-\frac{i}{2}(\partial|z|^2-\bar{\partial}|z|^2),\\
				G_{f}^*\theta_0=&\frac{i}{2}\sum_{j=1}^{n}dz_j\wedge d\bar{z}_j=:\tilde{\theta_0},\\
				G_{f}^*\theta_1=&2ih''(|z|^2)\bar{\partial} |z|^2\wedge\partial |z|^2+4h'(|z|^2)\tilde{\theta}_0,\\
				G_{f}^*\theta_2=&4ih'(|z|^2)h''(|z|^2)\bar{\partial }|z|^2\wedge \partial |z|^2+4h'^2(|z|^2)\tilde{\theta_0}.
			\end{align*}
		\end{lemma}
		\begin{lemma}
			\label{lemma:BehaviorThetaBCOnRotationInvariantSmooth}
			Let $f\in\Conv(\C^n,\R)$ be given by $f(z)=h(|z|^2)$ for $h\in C^\infty(\R)$. For $1\le k\le 2n$ and $\phi\in C_c(\C^n)$,
			\begin{align*}
				&\int_{\C^n}\phi(z)d\Theta^n_{k,q}(f,z)\\
				&\quad=4^{k-q}\binom{n}{k-2q,q}\int_{0}^\infty \left(\int_{S^{2n-1}} \phi(r v)dv\right)r^{2n-1}\left[\frac{k}{n}r^2h''(r^2)h'(r^2)^{k-1}+h'(r^2)^{k}\right]dr,\\
				&\int_{\C^n}\phi(z)d\mathcal{B}^n_{k,q}(f,z)=\int_{\C^n}\phi(z)d\mathcal{C}^n_{k,q}(f,z)\\
				&\quad=\frac{2^{2k-2q-1}}{n}\binom{n}{k-2q,q}\int_{0}^\infty \left(\int_{S^{n2-1}} \phi(rv)dv\right)r^{2n+1}\left[2r^2h''(r^2)h'(r^2)^{k-1}+h'(r^2)^k\right]dr.
			\end{align*}
		\end{lemma}
		\begin{proof}
			As $f$ is a smooth function, we have for $\phi\in C_c(\C^n)$
			\begin{align*}
				\int_{\C^n}\phi d\Theta^n_{k,q}(f)=&c_{n,k,q}D(f)[\phi(z)\theta^{n-1}_{k,q}]
				=c_{n,k,q}\int_{\C^n}\phi(z)G_f^*(\theta_0^{n-k+q}\wedge \theta_1^{k-2q}\wedge\theta_2^q).
			\end{align*}
			A short calculation using Lemma \ref{lemma:pullbackFormsGraph} shows that this reduces to
			\begin{align*}
				\int_{\C^n}\phi d\Theta^n_{k,q}(f)=&c_{n,k,q}2^{2k-2q-1}\int_{\C^n}\phi(z) kh''(|z|^2)h'(|z|^2)^{k-1}\tilde{\theta}_0^{n-1}\wedge i\bar{\partial}|z|^2\wedge\partial|z|^2\\
				&+c_{n,k,q}2^{2k-2q}\int_{\C^n}\zeta(|z|) h'(|z|^2)^{k}\tilde{\theta}_0^{n}.
			\end{align*}
			As $i\bar{\partial}|z|^2\wedge\partial|z|^2\wedge\tilde{\theta}_0^{n-1}=\frac{2}{n}|z|^2\tilde{\theta}_0$ and $\frac{\tilde{\theta}_0^n}{n!}$ is the volume form on $\C^n$, we obtain
			\begin{align*}
				\int_{\C^n}\phi d\Theta^n_{k,q}(f)=&c_{n,k,q}2^{2k-2q}n!\int_{\C^n}\phi(z)|z|^2 \frac{k}{n}h''(|z|^2)h'(|z|^2)^{k-1} d\vol_{2n}(z)\\
				&+c_{n,k,q}2^{2k-2q}n!\int_{\C^n}\phi(z) h'(|z|^2)^{k}d\vol_{2n}(z)\\
				=&c_{n,k,q}4^{k-q}n!\int_{0}^\infty \left(\int_{S^{2n-1}}\phi(rv)dv \right)r^{2n-1}\left[\frac{k}{n}r^2h''(r^2)h'(r^2)^{k-1}+h'(r^2)^{k}\right]dr
			\end{align*}
			by a change to polar coordinates. Now the claim follows from $n!c_{n,k,q}=\frac{n!}{(n-k+q)!(k-2q)!q!}=\binom{n}{k-2q,q}$. The other two cases follow from a similar calculation.
		\end{proof}
		\begin{corollary}
			\label{corollary:VanishingUpsilonRotationInvariantFunctions}
			Let $f\in\Conv(\C^n,\R)$ be rotation invariant. Then $\Upsilon^n_{k,q}(f)=0$.
		\end{corollary}
		\begin{proof}
			As $\mathcal{B}^{n}_{k,q}$ and $\mathcal{C}^n_{k,q}$ coincide on smooth rotation invariant functions by Lemma \ref{lemma:BehaviorThetaBCOnRotationInvariantSmooth} and $\Upsilon^n_{k,q}=\mathcal{B}^{n}_{k,q}-\mathcal{C}^{n}_{k,q}$, $\Upsilon^n_{k,q}$ vanishes on smooth rotation invariant functions. Since every rotation invariant function in $\Conv(\C^n,\R)$ can be approximated uniformly on compact subsets by a sequence in $\Conv(\C^n,\R)$ of smooth rotation invariant functions and $\Upsilon^n_{k,q}$ is continuous, the claim follows.
		\end{proof}
		Next, we will explicitly calculate these measures for the family of functions $u_t\in \Conv(\C^n,\R)$ defined for $t\ge 0$ by
		\begin{align*}
			u_t(z)=\max(0,|z|-t)=\frac{1}{2}\left(\sqrt{(|z|-t)^2}+|z|-t\right).
		\end{align*}
		The argument will be based on approximating $u_t$ by the following smooth convex functions.
		\begin{lemma}
			\label{lemma:LimitsCalculationUt}
			Consider the functions $h_\epsilon\in C^\infty([0,\infty))$ defined by
			\begin{align*}
				h_\epsilon(s):=&\frac{1}{2}\left[\sqrt{(\sqrt{s+\epsilon^2}-t)^2+\epsilon^2}+\sqrt{s+\epsilon^2}-t\right]&& \text{for } s\ge0.
			\end{align*}
			Then for $\zeta\in C_c([0,\infty))$ and $a\ge k$,
			\begin{align*}
				\lim\limits_{\epsilon\rightarrow0}\int_0^\infty \zeta(r) r^ah_\epsilon'(r^2)^{k}dr=&2^{-k}\int_t^\infty \zeta(r)r^{a-k}dr,\\
				\lim\limits_{\epsilon\rightarrow0}\int_0^\infty \zeta(r) r^{a+2}h''_\epsilon(r^2)h_\epsilon'(r^2)^{k-1}dr=&-2^{-(k+1)}\int_t^\infty \zeta(r)r^{a-k}dr+\frac{2^{-(k+1)}}{k}\zeta(t)t^{a+1-k}.
			\end{align*}
		\end{lemma}
		\begin{proof}
			Explicitly, we have
			\begin{align*}
				h'_\epsilon(s)=&\frac{1}{4}\frac{1}{\sqrt{s+\epsilon^2}\sqrt{(\sqrt{s+\epsilon^2}-t)^2+\epsilon^2}}\left[\sqrt{(\sqrt{s+\epsilon^2}-t)^2+\epsilon^2}+\sqrt{s+\epsilon^2}-t\right],\\
				h''_\epsilon(s)=&-\frac{1}{8}\frac{(\sqrt{s+\epsilon^2}-t)^3+\sqrt{(\sqrt{s+\epsilon^2}-t)^2+\epsilon^2}^3-(t-\sqrt{s+\epsilon^2})\epsilon^2-\sqrt{s+\epsilon^2}\epsilon^2}{\sqrt{s+\epsilon^2}^3\sqrt{(\sqrt{s+\epsilon^2}-t)^2+\epsilon^2}^3}.
			\end{align*}
			Note that 
			\begin{align*}
				g_{1,\epsilon}(s):=\sqrt{s+\epsilon^2}h'_\epsilon(s), \quad s\ge 0,
			\end{align*} 
			is bounded by $1$ and converges pointwise to $\frac{1}{2}1_{[t^2,\infty)}(s)$. Thus
			\begin{align*}
					\lim\limits_{\epsilon\rightarrow0}\int_0^\infty \zeta(r) r^ah_\epsilon'(r^2)^{k}dr=\lim\limits_{\epsilon\rightarrow0}\int_0^\infty \zeta(r) \frac{r^a}{\sqrt{r^2+\epsilon^2}^k}g_{1,\epsilon}(r^2)^{k}dr=2^{-k}\int_t^\infty \zeta(r)r^{a-k}dr
			\end{align*}
			by dominated convergence, which shows the first equation. Similarly, 
			\begin{align*} g_{2,\epsilon}(s):=\sqrt{s+\epsilon^2}^3\left(h''_\epsilon(s)-\frac{\sqrt{s+\epsilon^2}\epsilon^2}{8\sqrt{s+\epsilon^2}^3\sqrt{(\sqrt{s+\epsilon^2}-t)^2+\epsilon^2}^3}\right), \quad s\ge 0,
			\end{align*} is bounded by $1$ and converges pointwise to $-\frac{1}{4}1_{[t^2,\infty)}(s)$. Dominated convergence thus implies
			\begin{align}
				\notag
				&\lim\limits_{\epsilon\rightarrow0}\int_0^\infty\zeta(r)r^{a+2}\left[h_\epsilon''(r^2)-\frac{\sqrt{r^2+\epsilon^2}\epsilon^2}{8\sqrt{r^2+\epsilon^2}^3\sqrt{(\sqrt{r^2+\epsilon^2}-t)^2+\epsilon^2}^3}\right]h_\epsilon'(r^2)^{k-1}dr\\
				\notag
				=&\lim\limits_{\epsilon\rightarrow0}\int_0^\infty\zeta(r)r^{a-k}\frac{r^{k+2}}{\sqrt{r^2+\epsilon^2}^{k+2}}g_{2,\epsilon}(r^2)g_{1,\epsilon}(r^2)^{k-1}dr\\
				\label{eq:limit1}
				=&-2^{-(k+1)}\int_{t}^{\infty}\zeta(r)r^{a-k}dt.
			\end{align}
			Now consider the term we subtracted from $h_\epsilon''$ in the last integral, namely the integral
			\begin{align*}
				I_\epsilon:=&\frac{1}{8}\int_0^\infty\zeta(r)r^{a+2}\frac{\sqrt{r^2+\epsilon^2}\epsilon^2}{\sqrt{r^2+\epsilon^2}^3\sqrt{(\sqrt{r^2+\epsilon^2}-t)^2+\epsilon^2}^3}h_\epsilon'(r^2)^{k-1}dr\\
				=&\frac{1}{8}\int_0^\infty\zeta(r)\frac{r^{a+2}}{\sqrt{r^2+\epsilon^2}^{k+1}}\frac{\epsilon^2}{\sqrt{(\sqrt{r^2+\epsilon^2}-t)^2+\epsilon^2}^3}g_{1,\epsilon}(r^2)^{k-1}dr.
			\end{align*}
			With the substitution $x=\frac{\sqrt{r^2+\epsilon^2}-t}{\epsilon}$, that is, $r=\sqrt{(\epsilon x+t)^2-\epsilon^2}$, we thus obtain
			\begin{align*}
				I_\epsilon=&\frac{1}{8}\int_{\frac{\epsilon-t}{\epsilon}}^\infty\zeta(\sqrt{(\epsilon x+t)^2-\epsilon^2})\frac{\sqrt{(\epsilon x+t)^2-\epsilon^2}^{a+2}}{(\epsilon x+t)^{k+1}}\frac{1}{\sqrt{x^2+1}^3} g_{1,\epsilon}((\epsilon x+t)^2-\epsilon^2)^{k-1}dx.
			\end{align*}
			Note that
			$x\mapsto \frac{1}{\sqrt{x^2+1}^3}$ is integrable on $\R$ while the rest of the integrand stays bounded for $\epsilon\rightarrow0$ since $a\ge k$. Moreover, 
			\begin{align*}
				g_{1,\epsilon}((\epsilon x+t)^2-\epsilon^2)=&\frac{1}{4}\frac{\sqrt{\epsilon^2 x^2+\epsilon^2}+\epsilon x}{\sqrt{\epsilon^2x^2+\epsilon^2}}
				=\frac{1}{4}\left[1+\frac{x}{\sqrt{x^2+1}}\right]\quad \text{for}~x>\frac{\epsilon-t}{\epsilon}.
			\end{align*} 
			It is now easy to see that the integrand converges pointwise for $\epsilon\rightarrow0$. Dominated convergence thus implies
			\begin{align}
				\label{eq:limit2}
				\lim_{\epsilon\rightarrow0}I_\epsilon=&\frac{4^{-(k-1)}}{8}\zeta(t)t^{a+1-k}\int_{-\infty}^\infty\left[1+\frac{x}{\sqrt{x^2+1}}\right]^{k-1}\frac{1}{\sqrt{x^2+1}^3}dx.
			\end{align}
			This last integral can be calculated explicitly:
			\begin{align*}
				&\int_{-\infty}^\infty\left[1+\frac{x}{\sqrt{x^2+1}}\right]^{k-1}\frac{1}{\sqrt{x^2+1}^3}dx=\sum_{j=0}^{k-1}\binom{k-1}{j}\int_{-\infty}^\infty \frac{x^{j}}{\sqrt{x^2+1}^{j+3}}dx\\
				=&2\sum_{j=0}^{\lfloor\frac{k-1}{2}\rfloor}\binom{k-1}{2j}\int_0^\infty \frac{x^{2j}}{\sqrt{x^2+1}^{2j+3}}dx=2\sum_{j=0}^{\lfloor \frac{k-1}{2}\rfloor}\binom{k-1}{2j}\left[\frac{1}{2j+1}\frac{x^{2j+1}}{\sqrt{x^2+1}^{2j+1}}\right]_0^\infty\\
				=&2\sum_{j=0}^{\lfloor\frac{k-1}{2}\rfloor}\binom{k-1}{2j}\frac{1}{2j+1}
				=\frac{2}{k}\sum_{j=0}^{\lfloor \frac{k-1}{2}\rfloor}\binom{k}{2j+1}=\frac{2}{k}2^{k-1},
			\end{align*}
			where we have used that the sum contains exactly all binomial coefficients with odd index. Now the second equation follows by combining \eqref{eq:limit1} and \eqref{eq:limit2}.
		\end{proof}

		\begin{theorem}
			\label{theorem:ThetaValuesOnUt}
			Let $u_t(z)=\max(0,|z|-t)$, $1\le k\le 2n$. Then for $\phi\in C_c(\C^n)$
			\begin{align*}
					&\int_{\C^n}\phi(z)d\Theta^n_{k,q}(u_t,z)\\
					&\quad=\frac{2^{k-2q}}{n}\binom{n}{k-2q,q}\left( (2n-k)\int_t^\infty\left(\int_{S^{2n-1}}\phi(rv)dv\right)r^{2n-k-1}dr+t^{2n-k}\int_{S^{2n-1}}\phi(tv)dv\right),\\
					&\int_{\C^n}\phi(z)d\mathcal{B}^n_{k,q}(u_t,z)=\int_{\C^n}\phi(z)d\mathcal{C}^n_{k,q}(u_t,z)=\frac{2^{k-2q-1}}{kn}\binom{n}{k-2q,q}t^{2n-k+2}\left(\int_{S^{2n-1}}\phi(tv)dv\right).
			\end{align*}
		\end{theorem}
		\begin{proof}
			Since $\Theta^n_{k,q}$, $\mathcal{B}^n_{k,q}$, and $\mathcal{C}^n_{k,q}$ are continuous and $f_\epsilon(z):=h_\epsilon(|z|^2)$ converges to $u_t$ in $\Conv(\R^n,\R)$, the claim follows from Lemma \ref{lemma:BehaviorThetaBCOnRotationInvariantSmooth} by applying Lemma \ref{lemma:LimitsCalculationUt} to the function $\zeta(t):=\int_{S^{2n-1}}\phi(tv)dv$.
		\end{proof}
	
		\begin{corollary}
			\label{corollary:ZetaNotIntegrableWRTTheta}
			Let $\zeta\in C_b((0,\infty))$, $0\le k\le 2n-1$. Then $z\mapsto \zeta(|z|)$ is integrable with respect to $\Theta^n_{k,q}(|\cdot|)$ if and only if 
			\begin{align*}
				\int_0^\infty |\zeta(r)|r^{2n-k-1}dr<\infty.
			\end{align*}
			Moreover, there exists $\zeta\in D^{2n-k}$ that $z\mapsto \zeta(|z|)$ is not integrable with respect to $\Theta^n_{k,q}(|\cdot|)$.
		\end{corollary}
		\begin{proof}
			Let $\zeta\in C_b((0,\infty))$. By approximating $|\zeta|$ with an increasing sequence of continuous functions and using monotone convergence, we deduce from Theorem \ref{theorem:ThetaValuesOnUt} that
			\begin{align*}
				\int_{\C^n}|\zeta(|z|)|d\Theta^n_{k,q}(|\cdot|)=2\omega_{2n}2^{k-2q}\binom{n}{k-2q,q}(2n-k)\int_0^\infty|\zeta(r)|r^{{2n-k-1}}dr.
			\end{align*}
			In particular, the function $z\mapsto \zeta(|z|)$ is integrable with respect to $\Theta^n_{k,q}(|\cdot|)$ if and only if $\int_0^\infty |\zeta(r)|r^{2n-k-1}dr<\infty$.\\
			In order to show the second claim, it is therefore sufficient to construct a function $\zeta\in D^{2n-k}$ such that $r\mapsto \zeta(r)r^{2n-k-1}$ is not integrable on $(0,\infty)$.\\
			
			Define $\phi\in C([0,\infty))$ by
			\begin{align*}
				\phi(x)=\begin{cases}
					0 & x\in[0,2),\\
					\frac{(-1)^n}{n}\cdot\frac{1}{2^{n-2}} \cdot\frac{x-2^{n}}{2^{n}}& x\in [2^{n},\frac{3}{2}\cdot 2^{n}),~n\ge 1,\\
					\frac{(-1)^n}{n} \cdot\frac{1}{2^{n-2}}\cdot\frac{2^{n+1}-x}{2^{n}}& x\in [\frac{3}{2}\cdot 2^{n},2^{n+1}),~n\ge 1.
				\end{cases}
			\end{align*}
			Then for $a>2$, 
			\begin{align*}
				\left|\int_0^a\phi(x)dx-\sum_{n=1}^{\lfloor\frac{\ln a}{\ln 2}\rfloor-1}\frac{(-1)^n}{n}\right|\le \int_{2^{\lfloor\frac{\ln a}{\ln 2}\rfloor}}^a|\phi(x)|dx\le \frac{1}{\lfloor\frac{\ln a}{\ln 2}\rfloor+1},
			\end{align*}
			which converges to $0$ for $a\rightarrow\infty$. As the sum $\sum_{n=1}^{\infty}\frac{(-1)^n}{n}$ converges, so does $\lim\limits_{a\rightarrow\infty}\int_0^a\phi(x)dx$. Similarly, 
			\begin{align*}
				\int_0^\infty |\phi(t)|dt=\sum_{n=1}^{\infty}\frac{1}{n}=\infty,
			\end{align*}
			so $\phi$ is not integrable on $[0,\infty)$. Moreover,
			\begin{align*}
				|\phi(x)|\le \frac{1}{\lfloor\frac{\ln x}{\ln 2}\rfloor}\cdot \frac{8}{x}\cdot \frac{1}{2}\quad\text{for}~x\ge 2, 
			\end{align*}
			so $\lim\limits_{x\rightarrow\infty}x \phi(x)=0$. Now consider the function $\zeta(t)=\frac{\phi(\frac{1}{t})}{t^{2n-k+1}}$ for $t>0$, which is continuous and has bounded support. Then
			\begin{align*}
				\lim\limits_{t\rightarrow0}t^{2n-k}\zeta(t)=\lim\limits_{t\rightarrow0}\frac{1}{t}\phi\left(\frac{1}{t}\right)=0,
			\end{align*}
			and
			\begin{align*}
				\int_t^\infty\zeta(s)s^{2n-k-1}ds=\int_t^{\infty}\phi\left(\frac{1}{s}\right)\frac{1}{s^2}ds=\int_0^{\frac{1}{t}}\phi(x)dx,
			\end{align*}
			so the limit $\lim\limits_{t\rightarrow0}\int_t^\infty\zeta(s)s^{2n-k-1}ds$ exists and is finite. In particular $\zeta\in D^{2n-k}$. The same calculation shows that
			\begin{align*}
				\int_t^\infty|\zeta(s)|s^{2n-k-1}ds=\int_0^{\frac{1}{t}}|\phi(x)|dx,
			\end{align*}
			so $r\mapsto \zeta(r)r^{2n-k-1}$ is not integrable on $(0,\infty)$. Thus $\zeta\in D^{2n-k}$ is the desired example.
		\end{proof}

\section{Singular valuations associated to the Monge-Amp\`ere operators $\Theta^n_{k,q}$ and $\Upsilon^n_{k,q}$}
	\label{section:SingularValuations}

\subsection{Estimates for $\mathcal{B}^n_{k,q}$, $\mathcal{C}^n_{k,q}$, and $\Upsilon^n_{k,q}$}
	In this section we establish the following estimates for $\mathcal{B}^n_{k,a}$ and $\mathcal{C}^n_{k,q}$.
	\begin{theorem}
		\label{theorem:Estimate_BC}
		There exists a constant $C_{n,k,q}$ with the following property:
		For every $R>0$ and every $\phi\in C_c([0,\infty))$ with $\supp \phi\subset [0,R]$
		\begin{align*}
			&\left|\int_{\C^n}\phi(|z|) d\mathcal{B}^n_{k,q}(f)\right|\le C_{n,k,q} \left(\sup_{|z|\le R+1}|f(z)|\right)^k\|\phi\|_{\tilde{D}^{2n-k+2}},\\
			&\left|\int_{\C^n}\phi(|z|) d\mathcal{C}^n_{k,q}(f)\right|\le C_{n,k,q} \left(\sup_{|z|\le R+1}|f(z)|\right)^k\|\phi\|_{\tilde{D}^{2n-k+2}}.
		\end{align*}
	\end{theorem}
	We will deduce Theorem \ref{theorem:Estimate_BC} from the following estimate for smooth functions in $\Conv_{0}^+(\C^n,\R)$.
	\begin{proposition}
		\label{proposition:EstimateDifferentialCycleBC}
		There exists a constant $\tilde{C}_{n,k,q}$ with the following property:
		For every $R>0$, $\phi\in C_c((0,\infty))$ with $\supp \phi\subset[0,R]$, $\psi\in C_c^1((0,\infty))$, and $f\in\Conv_0^+(\C^n,\R)\cap C^\infty(\C^n)$,
		\begin{align*}
			&\left| D(f)\left[\psi(\mathcal{L}f(w))\phi(|z|) \beta_1\wedge\beta_2\wedge \theta^{n-1}_{k-2,q-1}\right]\right| 
			\le \tilde{C}_{n,k,q} \left(\sup_{|z|\le R+1}|f(z)|\right)^k\|\phi\|_{\tilde{D}^{2n-k+2}}\int_0^{c(f)}|\psi'(t)|dt,\\
			&\left| D(f)\left[\psi(\mathcal{L}f(w))\phi(|z|) \beta_1\wedge\gamma_2\wedge \theta^{n-1}_{k-2,q-1}\right]\right|\le \tilde{C}_{n,k,q} \left(\sup_{|z|\le R+1}|f(z)|\right)^k\|\phi\|_{\tilde{D}^{2n-k+2}}\int_0^{c(f)}|\psi'(t)|dt,
		\end{align*}
		for $c(f)=(1+2R)\sup_{|z|\le R+1}|f(z)|$.
	\end{proposition}
	\begin{proof}
		We will show the first inequality, the second follows with the same argument.\\		
		As $D(f)\llcorner \beta_1=D(f)\llcorner d\pi_2^*\mathcal{L}f$ by Lemma \ref{lemma:relationLegendreBeta}, the slicing formula \eqref{equation:slicingFormula} implies
		\begin{align}
			\notag
			&D(f)\left[\psi(\mathcal{L}f(w))\phi(|z|) \beta_1\wedge\beta_2\wedge \theta^{n-1}_{k-2,q-1}\right]=\int_0^\infty \psi(t)\langle D(f), \pi_2^*\mathcal{L}f,t\rangle [\phi(|z|)\beta_2\wedge \theta^{n-1}_{k-2,q-1}]dt\\
			\label{eq:estimateBSlicing1}
			=&\int_0^\infty \psi(t)\langle D(f), \pi_2^*\mathcal{L}f,t\rangle \left[\frac{\phi(|z|)}{|z|^2}\left[i_{X_{\beta_1}}(\gamma_1\wedge\beta_2\wedge\theta^{n-1}_{k-2,q-1})+\gamma_1\wedge i_{X_{\beta_1}}(\beta_2\wedge\theta^{n-1}_{k-2,q-1})\right]\right]dt,
		\end{align}
		where we have used $i_{X_{\beta_1}}\gamma_1=|z|^2$ in the second step. We will treat the terms corresponding to the two different differential forms separately. For the first term, note that
		\begin{align*}
			&\int_0^\infty \psi(t)\langle D(f), \pi_2^*\mathcal{L}f,t\rangle \left[\frac{\phi(|z|)}{|z|^2}i_{X_{\beta_1}}\left(\gamma_1\wedge\beta_2\wedge\theta^{n-1}_{k-2,q-1}\right)\right]dt\\
			=&\int_0^\infty \psi(t)\nc(K_{\mathcal{L}f}^t) [\phi(|d\mathcal{L}f(w)|)|d\mathcal{L}f(w)|^{2n-k+1}j^*i_{X_{\beta_1}}\left(\gamma_1\wedge\beta_2\wedge\theta^{n-1}_{k-2,q-1}\right)]dt,
		\end{align*}
		where we have used the relation $F_{f*}\nc(K_{\mathcal{L}f}^t)=\langle D(f), \pi_2^*\mathcal{L}f,t\rangle$ from Proposition \ref{proposition:RelationNCSlicesDiffCycle}, Lemma \ref{lemma:pullbackSpericalFormsFf}, and that $\gamma_1\wedge\beta_2\wedge\theta^{n-1}_{k-2,q-1}$ is homogeneous of degree $2n-k+3$. Note that $w\mapsto \phi(|d\mathcal{L}f(w)|)$ vanishes unless $|d\mathcal{L}f(w)|\le R$, which implies $\mathcal{L}f(w)\le c(f)$ by Lemma \ref{lemma:EstimatesLegendreSublevelSetsConv0}, that is, $w\in K_{\mathcal{L}f}^{c(f)}$. Since $\nc(K_{\mathcal{L}f}^t)$ is supported on $\{(x,v)\in S\R^n: v~\text{outer normal to}~K^t_{\mathcal{L}f}~\text{in}~x\in\partial K^t_{\mathcal{L}f}\}$, the integrand vanishes identically for $t> c(f)$. Moreover,
		\begin{align*}
			\left|\phi(|d\mathcal{L}f(w)|)\right|\cdot |d\mathcal{L}f(w)|^{2n-k+1}\le \|\phi\|_{\tilde{D}^{2n-k+2}}\frac{1}{|d\mathcal{L}f(w)|}.
		\end{align*}
		We apply Proposition \ref{proposition:BoundCurvatureMeasure} to $\omega:=j^*i_{X_{\beta_1}}\left(\gamma_1\wedge\beta_2\wedge\theta^{n-1}_{k-2,q-1}\right)\in\Omega^{k-1,2n-k}(S\C^n)$ and the corresponding curvature measure $\Phi_\omega$ to obtain $C>0$ depending on $n,k,q$ only such that
		\begin{align*}
			&\left|\int_0^\infty \psi(t)\langle D(f), \pi_2^*\mathcal{L}f,t\rangle 	\left[\frac{\phi(|z|)}{|z|^2}i_{X_{\beta_1}}(\gamma_1\wedge\beta_2\wedge\theta^{n-1}_{k-2,q-1})\right]dt\right|\\
			\le &\|\phi\|_{\tilde{D}^{2n-k+2}}\int_0^{c(f)} |\psi(t)| \left[\int_{\R^n}\frac{1}{|d\mathcal{L}f(w)|}d|\Phi_\omega(K_{\mathcal{L}f}^t,w)|\right] dt\\
			\le &C\|\phi\|_{\tilde{D}^{2n-k+2}}\int_0^{c(f)} |\psi(t)| \left[\int_{\R^n}\frac{1}{|d\mathcal{L}f(w)|}dC_{k-1}(K_{\mathcal{L}f}^t)\right] dt.
		\end{align*}
		Corollary \ref{corollary:IntrinsicVolumesEstimateSublevelSets} implies
		\begin{align*}
			\int_0^{c(f)} |\psi(t)| \left[\int_{\R^n}\frac{1}{|d\mathcal{L}f(w)|}d\Phi_{k-1}(K_{\mathcal{L}f}^t)\right] dt=&\int_0^{c(f)} |\psi(t)| \nc(K_{\mathcal{L}f}^t)\left[\frac{1}{|d\mathcal{L}f(w)|}\omega_{k-1}\right] dt\\
			\le& 2\frac{(n-k+1)\omega_{n-k+1}}{\omega_{n-k}}\mu_k\left(K_{\mathcal{L}f}^{c(f)}\right)\int_0^{c(f)}|\psi'(t)|dt.
		\end{align*}
		Using that $|\mathcal{L}f(w)|\le c(f)$ implies $|w|\le2\sup_{|z|\le R+1}|f(z)|$ by Lemma \ref{lemma:EstimatesLegendreSublevelSetsConv0}, we obtain $K^{c(f)}_{\mathcal{L}f}\subset \left(2\sup_{|z|\le R+1}|f(z)|\right)\cdot B_1(0)$. Since the intrinsic volume $\mu_k$ is $k$-homogeneous and monotone, we obtain a constant $C'$ depending on $n,k,q$ only such that
		\begin{align}
			\label{eq:estimateB1}
			\begin{split}
				&\left|\int_0^\infty \psi(t)\langle D(f), \pi_2^*\mathcal{L}f,t\rangle 	\left[\frac{\phi(|z|)}{|z|^2}i_{X_{\beta_1}}(\gamma_1\wedge\beta_2\wedge\theta^{n-1}_{k-2,q-1})\right]dt\right|\\
				\le & C'\left(\sup_{|z|\le R+1}|f(z)|\right)^k\|\phi\|_{\tilde{D}^{2n-k+2}}\int_0^{c(f)}|\psi'(t)|dt.
			\end{split}
		\end{align}
		
		Let us now turn to the second term in \eqref{eq:estimateBSlicing1}. Let $\pi:(\C^n\setminus\{0\})\times\C^n\rightarrow S^{2n-1}\times \C^n$ denote the radial projection in the first argument. As $\beta_2\wedge\theta^{n-1}_{k-2,q-1}$ is homogeneous of degree $2n-k+1$, we obtain
		\begin{align*}
			i_{X_{\beta_1}}(\gamma_1\wedge\beta_2\wedge\theta^{n-1}_{k-2,q-1})=|z|^{2n-k+1}\pi^*i_{X_{\beta_1}}(\gamma_1\wedge\beta_2\wedge\theta^{n-1}_{k-2,q-1})|_{S^{2n-1}\times\C^n}
		\end{align*}
		as in the proof of Proposition \ref{lemma:ChangeFormsPolarCoordinates}.
		Consider the map $r:(\C^n\times\setminus\{0\})\times\C^n\rightarrow\R$ given by $r(z,w)=|z|$. Then $\gamma_1=rdr$, so applying the slicing formula \eqref{equation:slicingFormula} once again, we obtain 		
		\begin{align*}
			&\int_0^\infty \psi(t)\langle D(f), \pi_2^*\mathcal{L}f,t\rangle \left[\frac{\phi(|z|)}{|z|^2}\gamma_1\wedge i_{X_{\beta_1}}(\beta_2\wedge\theta^{n-1}_{k-2,q-1})\right]dt\\
			=&\int_0^\infty \psi(t)\int_0^\infty \left\langle\langle D(f), \pi_2^*\mathcal{L}f,t\rangle, r,s\right\rangle  \left[\frac{\phi(s)}{s}s^{2n-k+1}\pi^*i_{X_{\beta_1}}(\beta_2\wedge\theta^{n-1}_{k-2,q-1})|_{S^{2n-1}\times\C^n}\right]dsdt\\
			=&\int_0^\infty \psi(t) \int_0^\infty\langle D(f), \pi_2^*\mathcal{L}f,t\rangle  \left[\phi(s)s^{2n-kds}1_{\{|z|\le s\}}d\pi^*i_{X_{\beta_1}}(\beta_2\wedge\theta^{n-1}_{k-2,q-1})|_{S^{2n-1}\times\C^n}\right]dsdt\\
			=&\int_0^\infty \psi(t) \langle D(f), \pi_2^*\mathcal{L}f,t\rangle  \left[\int_{|z|}^\infty\phi(s)s^{2n-kds}ds~ d\pi^*i_{X_{\beta_1}}(\beta_2\wedge\theta^{n-1}_{k-2,q-1})|_{S^{2n-1}\times\C^n}\right]dt\\
			=&\int_0^\infty \psi(t) \nc \left(K_{\mathcal{L}f}^t\right) \left[\int_{|d\mathcal{L}f(w)|}^\infty\phi(s)s^{2n-kds}ds~dF_f^*\pi^*i_{X_{\beta_1}}(\beta_2\wedge\theta^{n-1}_{k-2,q-1})|_{S^{2n-1}\times\C^n}\right]dt,
		\end{align*}
		where we used the relation $F_{f*}\nc(K_{\mathcal{L}f}^t)=\langle D(f), \pi_2^*\mathcal{L}f,t\rangle$  in the last step. Note that 
		\begin{align*}
			dF_f^*\pi^*i_{X_{\beta_1}}(\beta_2\wedge\theta^{n-1}_{k-2,q-1})|_{S^{2n-1}\times\C^n}= dj^*i_{X_{\beta_1}}(\beta_2\wedge\theta^{n-1}_{k-2,q-1}).
		\end{align*}
		We apply Proposition \ref{proposition:BoundCurvatureMeasure} to the curvature measure $\Phi_{\omega'}$ for $\omega':=dj^*i_{X_{\beta_1}}(\beta_2\wedge\theta^{n-1}_{k-2,q-1})$, and obtain $C>0$ depending on $n,k,q$ only such that
		\begin{align*}
			&\left|	\int_0^\infty \psi(t)\langle D(f), \pi_2^*\mathcal{L}f,t\rangle \left[\frac{\phi(|z|)}{|z|^2}\gamma_1\wedge i_{X_{\beta_1}}(\beta_2\wedge\theta^{n-1}_{k-2,q-1})\right]dt\right|\\
			=&\left|\int_0^\infty \psi(t) \int_{\C^n} \left[\int_{|d\mathcal{L}f(w)|}^{\infty}\phi(s)s^{2n-kds}ds\right] d\Phi_{\omega'}\left(K_{\mathcal{L}f}^t,w\right)dt\right|\\
			\le &C\int_0^\infty |\psi(t)| \int_{\C^n} \left|\int_{|d\mathcal{L}f(w)|}^\infty\phi(s)s^{2n-kds}ds\right| dC_k\left(K_{\mathcal{L}f}^t,w\right)dt
		\end{align*}
		As $\phi$ is supported on $[0,R]$, the inner integrand vanishes identically unless $|d\mathcal{L}f(w)|\le R$, which implies $|\mathcal{L}f(w)|\le c(f)$ by Lemma \ref{lemma:EstimatesLegendreSublevelSetsConv0}, that is, $w\in K_{\mathcal{L}f}^t$ for $t\le c(f)$. As in the previous case, this implies that the innermost integral vanishes for for $w\in \supp C_k(K_{\mathcal{L}f}^t)$ for $t> c(f)$. We thus obtain
		\begin{align*}
			&\left|	\int_0^\infty \psi(t)\langle D(f), \pi_2^*\mathcal{L}f,t\rangle \left[\frac{\phi(|z|)}{|z|^2}\gamma_1\wedge i_{X_{\beta_1}}(\beta_2\wedge\theta^{n-1}_{k-2,q-1})\right]dt\right|\\
			\le &C\sup_{r>0}\left|r\int_r^\infty \phi(s)s^{2n-k}ds\right|\int_0^{c(f)} |\psi(t)| \int_{\C^n} \frac{1}{|d\mathcal{L}f(w)|} dC_k\left(K_{\mathcal{L}f}^t,w\right)dt\\
			\le&C'\|\phi\|_{\tilde{D}^{2n-k+2}}\mu_k\left(\mathcal{K}^{c(f)}_{\mathcal{L}f}\right)\int_0^{c(f)} |\psi'(t)| dt 
		\end{align*}
		for a constant $C'$ depending on $n,k,q$ only, where we used Corollary \ref{corollary:IntrinsicVolumesEstimateSublevelSets} in the last step, as well as the estimate
		\begin{align*}
			\left|r\int_r^\infty \phi(s)s^{2n-k}ds\right|\le 	\|\phi\|_{\tilde{D}^{2n-k+2}}r\int_r^\infty \frac{1}{s^2}ds=\|\phi\|_{\tilde{D}^{2n-k+2}}.
		\end{align*}
		As in the previous case, $K^{c(f)}_{\mathcal{L}f}\subset \left(2\sup_{|z|\le R+1}|f(z)|\right)\cdot B_1(0)$, so we obtain
		\begin{align*}
			&\left|	\int_0^\infty \psi(t)\langle D(f), \pi_2^*\mathcal{L}f,t\rangle \left[\frac{\phi(|z|)}{|z|^2}\gamma_1\wedge i_{X_{\beta_1}}(\beta_2\wedge\theta^{n-1}_{k-2,q-1})\right]dt\right|\\
			\le &C''\left(\sup_{|z|\le R+1}|f(z)|\right)^k \|\phi\|_{\tilde{D}^{2n-k+2}}\int_0^{c(f)} |\psi'(t)| dt
		\end{align*}
		for a constant $C''>0$ depending on $n,k,q$ only. Combining this estimate with \eqref{eq:estimateB1}, we obtain the desired inequality.
	\end{proof}
	
	\begin{proof}[Proof of Theorem \ref{theorem:Estimate_BC}]
		We will prove the first inequality, the second follows with the same argument. Assume first that $f\in\Conv_0^+(\C^n,\R)\cap C^\infty(\C^n)$.\\
		
		Let $\psi_1\in C^\infty((0,\infty))$ be a function with $\psi_1=0$ on $[0,\frac{1}{2}]$, $\psi_1=1$ on $[1,\infty)$ and let $\psi_2\in C^\infty((0,\infty))$ be a function with $\psi_2=1$ on $(0,1]$, $\psi_2=0$ on $[2,\infty)$. We may assume that $\psi_1$ is non-decreasing. For $\delta\in (0,1)$ we define $\psi_\delta\in C^\infty_c(0,\infty)$ by
		\begin{align*}
			\psi_\delta(t):=\psi_1\left(\frac{t}{\delta}\right)\psi_2(\delta t).
		\end{align*}
		Then $\psi_\delta\equiv 1$ on $[\delta,\frac{1}{\delta}]$. Moreover, as in the proof of \cite[Proposition 6.6]{KnoerrSingularvaluationsHadwiger2022}, 
		\begin{align*}
			\lim_{\delta\rightarrow0}\int_0^{c(f)}|\psi_\delta'(t)|dt=1.
		\end{align*}
		As $f\in\Conv_0^+(\C^n,\R)\cap C^\infty(\C^n)$, $\mathcal{L}f$ is a smooth function and $\mathcal{L}f(w)=\langle df(z),z\rangle -f(z)$ for all $(z,w)\in\supp D(f)$. Moreover, $\psi_\delta(\langle df(z),z\rangle -f(z))$ converges to $1$ for $\delta\rightarrow0$ for all $z\in\C^n$ such that $\langle df(z),z\rangle -f(z)\ne0$, that is, $z\ne 0$. In addition, $\mathcal{B}^n_{k,q}(f)$ is absolutely continuous with respect to the Lebesgue measure because $f$ is smooth. If $\phi\in C_c([0,\infty))$, then $z\mapsto \phi(|z|)$ is integrable with respect to $\mathcal{B}^n_{k,q}(f)$ and dominated convergence implies
		\begin{align*}
			\left|\int_{\C^n}\phi(|z|)d\mathcal{B}^n_{k,q}(f)\right|=&\left|\lim\limits_{\delta\rightarrow0}\int_{\C^n}\psi_\delta(\langle df(z),z\rangle -f(z))|\phi(|z|)|d\mathcal{B}^n_{k,q}(f)\right|\\
			=&\lim\limits_{\delta\rightarrow0}c_{n,k,q}\left|D(f)\left[\psi_\delta(\mathcal{L}f(w))\phi(|z|) \beta_1\wedge\beta_2\wedge \theta^{n-1}_{k-2,q-1}\right]\right| \\
			\le& \limsup_{\delta\rightarrow0}c_{n,k,q} \tilde{C}_{n,k,q} \left(\sup_{|z|\le R+1}|f(z)|\right)^k\|\phi\|_{\tilde{D}^{2n-k+2}}\int_0^\infty |\psi_\delta'(t)|dt\\
			=&c_{n,k,q} \tilde{C}_{n,k,q} \left(\sup_{|z|\le R+1}|f(z)|\right)^k\|\phi\|_{\tilde{D}^{2n-k+2}},
		\end{align*}
		where we have used Proposition \ref{proposition:EstimateDifferentialCycleBC} in the third step.\\
		Now let $f\in\Conv(\C^n,\R)\cap C^\infty(\C^n)$ be an arbitrary smooth convex function. Then $|df(0)|\le \frac{2}{R+1}\sup_{R+1}|f(z)|$ by Lemma \ref{lemma:LipschitzConstant}. Moreover $z\mapsto f-\langle df(0),\cdot\rangle -f(0)+\lambda |z|^2$ belongs to $\Conv_0^+(\C^n,\R)\cap C^\infty(\C ^n)$ for all $\lambda>0$. As $\mathcal{B}^n_{k,q}$ is dually epi-translation invariant and continuous, we obtain 
		\begin{align*}
			\left|\int_{\C^n}\phi(|z|)\mathcal{B}^n_{k,q}(f)\right|=&\lim\limits_{\lambda\rightarrow0}\left|\int_{\C^n}\phi(|z|)\mathcal{B}^n_{k,q}(f-\langle df(0),\cdot\rangle -f(0)+\lambda|\cdot|^2)\right|\\
			\le&\limsup_{\lambda\rightarrow0}c_{n,k,q} \tilde{C}_{n,k,q} \left(\sup_{|z|\le R+1}|f(z)-\langle df(0),z\rangle -f(0)+\lambda|z|^2|\right)^k\|\phi\|_{\tilde{D}^{2n-k+2}}\\
			\le& c_{n,k,q} \tilde{C}_{n,k,q} \left(4\sup_{|z|\le R+1}|f(z)|\right)^k\|\phi\|_{\tilde{D}^{2n-k+2}}.
		\end{align*}
		Thus for every $f\in\Conv(\C^n,\R)\cap C^\infty(\C^n)$,
		\begin{align*}
			\left|\int_{\C^n}\phi(|z|)\mathcal{B}^n_{k,q}(f)\right|\le C_{n,k,q}\left(\sup_{|z|\le R+1}|f(z)|\right)^k\|\phi\|_{\tilde{D}^{2n-k+2}}
		\end{align*}
		for $C_{n,k,q}=c_{n,k,q}\tilde{C}_{n,k,q}4^k$. Now note that both sides of this inequality depend continuously on $f\in\Conv(\C^n,\R)$. Thus it holds for arbitrary $f\in\Conv(\C^n,\R)$. This finishes the proof.
	\end{proof}
	
	\begin{corollary}
		\label{corollary:estimateUpsilon}
		There exists a constant $B_{n,k,q}$ with the following property:
		For every $R>0$ and every $\phi\in C_c([0,\infty))$ with $\supp \phi\subset [0,R]$
		\begin{align*}
			&\left|\int_{\C^n}\phi(|z|) d\Upsilon^n_{k,q}(f)\right|\le B_{n,k,q} \left(\sup_{|z|\le R+1}|f(z)|\right)^k\|\phi\|_{\tilde{D}^{2n-k+2}}.
		\end{align*}
	\end{corollary}
	\begin{proof}
		Since $\Upsilon^n_{k,q}=\mathcal{B}^n_{k,q}-\mathcal{C}^n_{k,q}$, the claim follows from Theorem \ref{theorem:Estimate_BC} with $B_{n,k,q}=2C_{n,k,q}$.
	\end{proof}
	
	\subsection{Estimates for $\Theta^n_{k,q}$}
		\begin{lemma}
		\label{lemma:TauAndTildeTau}
		Consider the $(2n-1)$-forms on $T^*\C^n$ given by
		\begin{align*}
			\tau_{k,q}&:=i_{X_{\beta_1}}\theta^n_{k,q}=(n-k+q)\gamma_2\wedge\theta^{n-1}_{k,q}+(k-2q)\beta_2\wedge\theta^{n-1}_{k-1,q},\\
			\tilde{\tau}_{k,q}&:=-i_{X_{\gamma_1}}\theta^n_{k,q}=(k-2q)\gamma_2\wedge\theta^{n-1}_{k-1,q}+q\beta_2\wedge\theta^{n-1}_{k-2,q-1}.
		\end{align*}
		Then 
		\begin{align*}
			r^2\theta^n_{k,q}\equiv& \gamma_1\wedge \tau_{k,q}+\beta_1\wedge\tilde{\tau}_{k,q}\quad\mod\omega_s,\\
			d\tau_{k,q}=&(2n-k)\theta^n_{k,q},\\
			d\tilde{\tau}_{k,q}=&2(k-2q)\theta^n_{k-1,q}+q\theta^n_{k-1,q-1}.
		\end{align*}
		\end{lemma}
		\begin{proof}
			The first equation was shown in \cite[Corollary 3.12]{KnoerrUnitarilyinvariantvaluations2021}. The last two relations follow from $d\gamma_2=2\theta_0$ and $d\beta_2=\theta_1$.
		\end{proof}
	
		\begin{corollary}
			\label{corollary:PartialIntegrationFormulaTauKQ}
			For $\Psi\in C^1((0,\infty))$, 
			\begin{align*}
				\left[(2n-k)\Psi(|z|)+|z|\Psi'(|z|)\right]\theta^n_{k,q}\equiv	d\left(\Psi(|z|)\tau_{k,q}\right)+\frac{\Psi'(|z|)}{|z|}\beta_1\wedge \tilde{\tau}_{k,q}\quad\mod \omega_s
			\end{align*}
			on $(\C^n\setminus\{0\})\times\C^n$.
		\end{corollary}
		\begin{proof}
			Note that $d|z|=\frac{\gamma}{|z|}$. Hence
			\begin{align*}
				d\left(\Psi(|z|)\tau_{k,q}\right)=&\frac{\Psi'(|z|)}{|z|}\gamma\wedge\tau_{k,q}+\Psi(|z|)d\tau_{k,q}\\
				\equiv&\frac{\Psi'(|z|)}{|z|}\left(|z|^2\theta^n_{k,q}-\beta_1\wedge\tilde{\tau}_{k,q}\right)+(2n-k)\Psi(|z|)\theta^n_{k,q}\quad\mod\omega_s
			\end{align*}
			on $(\C^n\setminus\{0\})\times\C^n$ by Lemma \ref{lemma:TauAndTildeTau}. Rearranging the equation, we obtain the desired result.
		\end{proof}
		
		\begin{proposition}
			\label{proposition:EstimateThetaDifferentialCycle}
			There exists a constant $\tilde{A}_{n,k,q}>0$ with the following property:
			For every $R>0$, $\zeta\in C_c([0,\infty))$ with $\supp \zeta\subset [0,R)$, $\psi\in C_c^1((0,\infty))$, and $f\in\Conv_0^+(\C^n,\R)\cap C^\infty(\C^n)$,
			\begin{align*}
				\left|D(f)\left[\psi(\mathcal{L}f(w))\zeta(|z|)\theta^n_{k,q}\right]\right|\le \tilde{A}_{n,k,q}\left(\sup_{|z|\le R+1}|f(z)|\right)^k \|\zeta\|_{D^{2n-k}} \int_0^{c(f)}|\psi'(t)|dt
			\end{align*}
			for $c(f)=(1+2R)\sup_{|z|\le R+1}|f(z)|$.
		\end{proposition}
		\begin{proof}
			The function $\Psi\in C^1((0,\infty))$ given by
			\begin{align*}
				\Psi(t)=-\frac{1}{t^{2n-k}}\int_t^\infty \zeta(s)s^{2n-k-1}ds
			\end{align*}
			satisfies
			\begin{align*}
				\Psi'(t)=\frac{2n-k}{t^{2n-k+1}}\int_t^\infty \zeta(s)s^{2n-k-1}ds+\frac{\zeta(t)}{t},
			\end{align*}
			so
			\begin{align*}
				(2n-k)\Psi(t)+t\Psi'(t)=\zeta(t).
			\end{align*}
			Consequently, applying Corollary \ref{corollary:PartialIntegrationFormulaTauKQ}, we obtain
			\begin{align*}
				\zeta(|z|)\theta^n_{k,q}\equiv d\left(\Psi(|z|)\tau_{k,q}\right)+\frac{\Psi'(|z|)}{|z|}\beta_1\wedge \tilde{\tau}_{k,q}\quad\mod\omega_s~ \text{on}~\C^n\setminus\{0\}\times \C^n.
			\end{align*}
			Let $\psi\in C^1_c((0,\infty))$ and $f\in\Conv_0^+(\C^n,\R)\cap C^\infty(\C^n)$. If we multiply the previous equation with the function $(z,w)\mapsto \psi(\mathcal{L}f(w))$ on $\C^n\times\C^n$, we obtain a differential form that is well defined and continuous on a neighborhood of the support of $D(f)$. Since $D(f)$ vanishes on multiples of $\omega_s$ and is closed, we obtain
			\begin{align*}
				&D(f)\left[\psi(\mathcal{L}f(w))\zeta(|z|)\theta^n_{k,q}\right]\\
				=&D(f)\left[\psi(\mathcal{L}f(w))\left(d\left(\Psi(|z|)\tau_{k,q}\right)+\frac{\Psi'(|z|)}{|z|}\beta_1\wedge \tilde{\tau}_{k,q}\right)\right]\\
				=&D(f)\left[-\psi'(\mathcal{L}f(w))\pi_2^*d\mathcal{L}f(w)\wedge \Psi(|z|)\tau_{k,q}+\psi(\mathcal{L}f(w))\frac{\Psi'(|z|)}{|z|}\beta_1\wedge \tilde{\tau}_{k,q}\right]\\
				=&D(f)\left[-\psi'(\mathcal{L}f(w)) \Psi(|z|)\beta_1\wedge\tau_{k,q}+\psi(\mathcal{L}f(w))\frac{\Psi'(|z|)}{|z|}\beta_1\wedge \tilde{\tau}_{k,q}\right],
			\end{align*}
			because $\beta_1$ and $\pi_2^*d\mathcal{L}f(w)$ coincide on $D(f)$ by Lemma \ref{lemma:relationLegendreBeta}. Since $\beta_1\wedge \tilde{\tau}_{k,q}$ is a linear combination of $\beta_1\wedge\beta_2\wedge\theta^{n-1}_{k-2,q-1}$ and $\beta_1\wedge\gamma_2\wedge\theta^{n-1}_{k-1,q}$, we can estimate the second term using Proposition \ref{proposition:EstimateDifferentialCycleBC} to obtain
			\begin{align}
				\label{eq:estimateTheta1}
				\begin{split}
					&|D(f)\left[\psi(\mathcal{L}f(w))\frac{\Psi'(|z|)}{|z|}\beta_1\wedge \tilde{\tau}_{k,q}\right]|\\
					\le& (k-q)\tilde{C}_{n,k,q}\left(\sup_{|z|\le R+1}|f(z)|\right)^k\left(\sup_{t>0}t^{2n-k+2} \left|\frac{\Psi'(t)}{t}\right|\right)\int_0^{c(f)}|\psi'(t)|dt,
				\end{split}
			\end{align}
			where
			\begin{align*}
				\sup_{t>0}t^{2n-k+2} \left|\frac{\Psi'(t)}{t}\right|=\sup_{t>0}t^{2n-k+2}\left|\frac{(2n-k)}{t^{2n-k+2}}\int_t^\infty \zeta(r)r^{2n-k-1}dt+\frac{\zeta(t)}{t^2}\right|\le (2n-k+1)\|\zeta\|_{D^{2n-k}}.
			\end{align*}
			For the first term, we apply the slicing formula \eqref{equation:slicingFormula}, Proposition \ref{proposition:RelationNCSlicesDiffCycle}, and Lemma \ref{lemma:pullbackSpericalFormsFf} to obtain
			\begin{align*}
				D(f)\left[\psi'(\mathcal{L}f(w)) \Psi(|z|)\beta_1\wedge\tau_{k,q}\right]=&\int_0^\infty \psi'(t) \langle D(f), \pi_2^*\mathcal{L}f, t\rangle \left[\Psi(|z|)\tau_{k,q}\right]dt\\
				=&\int_0^\infty \psi'(t) \nc(K_{\mathcal{L}f}^t) \left[\Psi(|d\mathcal{L}f(w)|)|d\mathcal{L}f(w)|^{2n-k}j^*\tau_{k,q}\right]dt\\
				=&\int_0^\infty \psi'(t) \left[\int_{\C^n}\Psi(|d\mathcal{L}f(w)|)|d\mathcal{L}f(w)|^{2n-k}d\Phi_{j^*\tau_{k,q}}(K_{\mathcal{L}f}^t)\right]dt.
			\end{align*}
			Note that the integrand vanishes unless $|d\mathcal{L}f(w)|\le R$, which by Lemma \ref{lemma:EstimatesLegendreSublevelSetsConv0} implies $w\in K_{\mathcal{L}f}^t$ for some $t\le c(f)$. As in the proof of Proposition \ref{proposition:EstimateDifferentialCycleBC}, this implies that the inner integral vanishes identically for $t\ge c(f)$. In addition $|\Psi(t)t^{2n-k}|\le \|\zeta\|_{D^{2n-k}}$.
			We apply Proposition \ref{proposition:BoundCurvatureMeasure} to the curvature measure $\Phi_{j^*\tau_{k,q}}$ and obtain $C>0$ depending on $n,k,q$ only such that
			\begin{align*}
				&\left|D(f)\left[\psi'(\mathcal{L}f(w)) \Psi(|z|)\beta_1\wedge\tau_{k,q}\right]\right|\\
				\le& C\int_0^{c(f)} |\psi'(t)|  \left[\int_{\C^n}|\Psi(|d\mathcal{L}f(w)|)|\cdot|d\mathcal{L}f(w)|^{2n-k}dC_k(K_{\mathcal{L}f}^t)\right]dt\\
				\le& C\|\zeta\|_{D^{2n-k}}\mu_k\left(K^{c(f)}_{\mathcal{L}f}\right)\int_0^{c(f)} |\psi'(t)|dt.
			\end{align*}
			Here we used again that $\mu_k$ is monotone and that $K^{t}_{\mathcal{L}f}\subset K^{c(f)}_{\mathcal{L}f}$ for $t\le c(f)$. As in the proof of Proposition \ref{proposition:EstimateDifferentialCycleBC}, we have $K_{\mathcal{L}f}^{c(f)}\subset \left(2\sup_{|z|\le R+1} |f(z)|\right)\cdot B_1(0)$, so the homogeneity and monotonicity of $\mu_k$ imply
			\begin{align*}
				\left|D(f)\left[\psi'(\mathcal{L}f(w)) \Psi(|z|)\beta_1\wedge\tau_{k,q}\right]\right|\le C'\|\zeta\|_{D^{2n-k}}\left(\sup_{|z|\le R+1}|f(z)|\right)^k\int_0^{c(f)} |\psi'(t)|dt.
			\end{align*}
		for a constant $C'>0$ depending on $n,k,q$ only. Combining this estimate with \eqref{eq:estimateTheta1}, we obtain the desired inequality.
		\end{proof}
		
		\begin{theorem}
			\label{theorem:EstimateTheta}
			There exists a constant $A_{n,k,q}>0$ with the following property: For every $f\in\Conv(\C^n,\R)$, $R>0$, and $\zeta\in C_c([0,\infty))$ with $\supp\zeta\subset[0,R]$
				\begin{align*}
					\left|\int_{\C^n}\zeta(|z|)d\Theta^n_{k,q}(f)\right|\le A_{n,k,q}\left(\sup_{|z|\le R+1}|f(z)|\right)^k\|\zeta\|_{D^{2n-k}}.
				\end{align*}
		\end{theorem}
		\begin{proof}
			This follows from Proposition \ref{proposition:EstimateThetaDifferentialCycle} using the same argument as in the proof of Theorem \ref{theorem:Estimate_BC}.
		\end{proof}

	\subsection{Construction of singular valuations}
		Consider the maps
		\begin{align*}
			T^{n}_{k,q}:C_c([0,\infty))\cap D^{2n-k}_R&\rightarrow\VConv_{k,q}(\C^n)^{\U(n)}\cap \VConv_{B_R(0)}(\C^n)\\
			\notag
			\zeta&\mapsto \left[f\mapsto \int_{\C^n}\zeta(|z|)d\Theta^n_{k,q}\right],\\
			Y^{n}_{k,q}:C_c([0,\infty))\cap \tilde{D}^{2n-k+2}_R&\rightarrow\VConv_{k,q}(\C^n)^{\U(n)}\cap \VConv_{B_R(0)}(\C^n)\\
			\notag
			\tilde{\zeta}&\mapsto \left[f\mapsto \int_{\C^n}\tilde{\zeta}(|z|)d\Upsilon^n_{k,q}\right].
		\end{align*}
		Theorem \ref{maintheorem:ExistenceExtensionIntegration} follows directly from the following result.
		\begin{theorem}
			\label{theorem:ExistenceT}
			$T^n_{k,q}$ and $Y^n_{k,q}$ extend uniquely by continuity to continuous maps 
			\begin{align*}
				T^{n}_{k,q}:D^{2n-k}_R&\rightarrow\VConv_{k,q}(\C^n)^{\U(n)}\cap \VConv_{B_R(0)}(\C^n),\\
				Y^{n}_{k,q}:\tilde{D}^{2n-k+2}_R&\rightarrow\VConv_{k,q}(\C^n)^{\U(n)}\cap \VConv_{B_R(0)}(\C^n).
			\end{align*} Moreover,
			\begin{align*}
				&\|T^n_{k,q}(\zeta)\|_{R,1}\le A_{n,k,q}\|\zeta\|_{D^{2n-k}},
				&&\|Y^n_{k,q}(\tilde{\zeta})\|_{R,1}\le B_{n,k,q}\|\tilde{\zeta}\|_{\tilde{D}^{2n-k+2}}
			\end{align*}
			for every $\zeta\in D^{2n-k}_R$, $\tilde{\zeta}\in \tilde{D}^{2n-k+2}_R$.
		\end{theorem}	
		\begin{proof}
			From Theorem \ref{theorem:EstimateTheta}, we obtain for $\zeta\in D^{2n-k}_R\cap C_c([0,\infty))$ the estimate
			\begin{align*}
				\|T^n_{k,q}(\zeta)\|_{R,1}\le A_{n,k,q}\|\zeta\|_{D^{2n-k}},
			\end{align*}
			so 
			\begin{align*}
				T^n_{k,q}:D^{2n-k}_R\cap C_c([0,\infty))\rightarrow \VConv_{k,q}(\C^n)^{\U(n)}\cap \VConv_{B_R(0)}(\C^n)
			\end{align*}
			is continuous. Since $D^{2n-k}_R\cap C_c([0,\infty))\subset D^{2n-k}_R$ is dense by Lemma \ref{lemma:DensityDNJ_continuousFunctions} and $\VConv_{k,q}(\C^n)^{\U(n)}\cap \VConv_{B_R(0)}(\C^n)$ is complete, $T^n_{k,q}$ thus extends by continuity to $D^{2n-k}_R$. Using Corollary \ref{corollary:estimateUpsilon} and Lemma \ref{lemma:DensityTildeD_continuousFunctions}, the same argument applies to $Y^n_{k,q}$.
		\end{proof}

	\subsection{Representation in terms of principal value integrals}
		Next, we are going to prove the following version of Theorem \ref{maintheorem:PrincipalValueRepresentation}.
			\begin{theorem}
				\label{theorem:PrincipalValue}
				Let $\zeta\in D^{2n-k}_R$, $\tilde{\zeta}\in \tilde{D}_R^{2n-k+2}$. Then for $f\in \Conv(\C^n,\R)$,
				\begin{align*}
					T^n_{k,q}(\zeta)[f]=&\lim\limits_{\epsilon\rightarrow0}\int_{\C^n\setminus B_\epsilon(0)} \zeta(|z|)d\Theta^n_{k,q}(f),\\
					Y^n_{k,q}(\tilde{\zeta})[f]=&\lim\limits_{\epsilon\rightarrow0}\int_{\C^n\setminus B_\epsilon(0)} \tilde{\zeta}(|z|)d\Upsilon^n_{k,q}(f),
				\end{align*}
				where the convergence is uniform on subsets of the form $\{f\in\Conv(\C^n,\R): \sup_{|z|\le R+1} |f(z)|\le M\}$ for $M>0$. In particular the convergence is uniform on compact subsets of $\Conv(\C^n,\R)$.
			\end{theorem}
			\begin{proof}
				We start with the first equation. Choose $\phi\in C_c([0,\infty))$ with $\phi\equiv 1$ on $[0,1]$, $\supp\phi\subset [0,2]$. Consider the function $\zeta^\epsilon$ defined in Lemma \ref{lemma:DensityDNJ_continuousFunctions}. Since $\zeta(t)$ and $\zeta^\epsilon(t)$ coincide for $t\ge \epsilon$, we obtain
				\begin{align*}
					&\left|T^n_{k,q}(\zeta)[f]-\int_{\C^n\setminus B_\epsilon(0)}\zeta(|z|)d\Theta^n_{k,q}(f)\right|\\
					\le& |T^n_{k,q}(\zeta)-T^n_{k,q}(\zeta^\epsilon)[f]|+\left|T^n_{k,q}(\zeta^\epsilon)[f]-\int_{\C^n\setminus B_\epsilon(0)}\zeta(|z|)d\Theta^n_{k,q}(f)\right|\\
					\le & |T^n_{k,q}(\zeta)-T^n_{k,q}(\zeta^\epsilon)[f]|+\int_{B_\epsilon(0)}|\zeta^\epsilon(|z|)|d\Theta^n_{k,q}(f)\\
					\le & |T^n_{k,q}(\zeta)-T^n_{k,q}(\zeta^\epsilon)[f]|+|\zeta(\epsilon)|\int_{\C^n}\phi_{\epsilon}(|z|)d\Theta^n_{k,q}(f),
				\end{align*}
				where we set $\phi_\epsilon(t):=\phi\left(\frac{t}{\epsilon}\right)$ for $t\ge 0$. Since $\phi_\epsilon$ is supported on $(0,R]$ for $\epsilon>0$ small enough, the inequalities in Theorem \ref{theorem:EstimateTheta} and Theorem \ref{theorem:ExistenceT} imply
				\begin{align*}
					&\left|T^n_{k,q}(\zeta)[f]-\int_{\C^n\setminus B_\epsilon(0)}\zeta(|z|)d\Theta^n_{k,q}(f)\right|\\
					\le& A_{n,k,q}\left(\|\zeta-\zeta^\epsilon\|_{D^{2n-k}}+|\zeta(\epsilon)|\cdot\|\phi_\epsilon\|_{D^{2n-k}}\right)\cdot \left(\sup_{|z|\le R+1}|f(z)|\right)^k.
				\end{align*}
				It is easy to see (compare the proof of \cite[Theorem 6.10]{KnoerrSingularvaluationsHadwiger2022}) that 
				\begin{align*}
					\|\phi_\epsilon\|_{D^{2n-k}}\le 4\epsilon^{2n-k}\|\phi\|_{D^{2n-k}}. 
				\end{align*}
				For $f\in \{h\in\Conv(\C^n,\R): \sup_{|z|\le R+1}|h(z)|\le M\}$, we thus obtain
				\begin{align*}
					\left|T^n_{k,q}(\zeta)[f]-\int_{\C^n\setminus B_\epsilon(0)}\zeta(|z|)d\Theta^n_{k,q}(f)\right|\le A_{n,k,q}M^k\left(\|\zeta-\zeta^\epsilon\|_{D^{2n-k}}+|\zeta(\epsilon)|\epsilon^{2n-k}\cdot\|\phi\|_{D^{2n-k}}\right),
				\end{align*}
				which converges to zero as $\zeta^\epsilon$ converges to $\zeta\in D^{2n-k}$. This implies the first claim. Since every compact subset of $\Conv(\C^n,\R)$ is contained in a subset of the form $\{h\in\Conv(\C^n,\R): \sup_{|z|\le R+1}|h(z)|\le M\}$ (compare the characterization of compact subsets in \cite[Proposition 2.4]{Knoerrsupportduallyepi2021}), the claim follows.\\
				
				For the second equation, we consider the function $\tilde{\zeta}_\epsilon$ defined in Lemma \ref{lemma:DensityTildeD_continuousFunctions} and similarly obtain
				\begin{align*}
					&\left|Y^n_{k,q}(\tilde{\zeta})[f]-\int_{\C^n\setminus B_\epsilon(0)}\tilde{\zeta}(|z|)d\Upsilon^n_{k,q}(f)\right|
					\le& 	\left|Y^n_{k,q}(\tilde{\zeta})[f]-Y^n_{k,q}(\tilde{\zeta}_\epsilon)[f]\right| +\left|\int_{ B_\epsilon(0)}\tilde{\zeta}_\epsilon(|z|)d\Upsilon^n_{k,q}(f)\right|.
				\end{align*}
				Since the form inducing $\Upsilon^n_{k,q}$ is homogeneous of degree $2n-k+2$, Proposition \ref{proposition:GeneralBoundMA} implies
				\begin{align*}
					\left|\int_{B_\epsilon(0)}\tilde{\zeta}_\epsilon(|z|)d\Upsilon^n_{k,q}(f)\right|\le&|\tilde{\zeta}(\epsilon)|\int_{ B_\epsilon(0)}d|\Upsilon^n_{k,q}(f)|\le C(\Upsilon^n_{k,q})|\tilde{\zeta}(\epsilon)|\int_{B_\epsilon(0)}\phi_\epsilon(|z|)|z|^2d\Phi_{k}(f)\\
					\le& \tilde{C}\left(\sup_{|z|\le R+1}|f(z)|\right)^k|\tilde{\zeta}(\epsilon)|\cdot\|\tilde{\phi}_\epsilon \|_{D^{2n-k}}
				\end{align*}
				for $\tilde{\phi}_\epsilon(t)=t^2\phi_\epsilon(t)$, for some constant $\tilde{C}$ independent of $R>0$, $f\in\Conv(\C^n,\R)$, $\tilde{\zeta}\in \tilde{D}^{2n-k+2}_R$. Here we used Corollary \ref{corollary:GeneralIntegrabilityMA} in the last step. A simple estimate shows 
				\begin{align*}
					\|\tilde{\phi}_\epsilon\|_{D^{2n-k}}\le 4\epsilon^{2n-k+2}\|\phi\|_{D^{2n-k+2}}.
				\end{align*}
			Combining this estimate with the norm estimate for $Y^n_{k,q}$ in Theorem \ref{theorem:ExistenceT}, we obtain a constant $C>0$ independent of $f\in\Conv(\C^n,\R)$, $R>0$, and $\tilde{\zeta}\in \tilde{D}^{2n-k+2}$ such that
			\begin{align*}
					&\left|Y^n_{k,q}(\tilde{\zeta})[f]-\int_{\C^n\setminus B_\epsilon(0)}\tilde{\zeta}(|z|)d\Upsilon^n_{k,q}(f)\right|\\
				\le& C\left(\sup_{|z|\le R+1}|f(z)\right)^k\left[\|\tilde{\zeta}-\tilde{\zeta}_\epsilon\|_{\tilde{D}^{2n-k+2}}+4\epsilon^{2n-k+2}|\tilde{\zeta}(\epsilon)|\cdot\|\phi\|_{D^{2n-k+2}}\right]
			\end{align*}
			Since $\tilde{\zeta}\in \tilde{D}^{2n-k+2}$, and $\tilde{\zeta}_\epsilon$ converges to $\tilde{\zeta}$ with respect to $\|\cdot\|_{\tilde{D}^{2n-k+2}}$ by Lemma \ref{lemma:DensityTildeD_continuousFunctions}, the right hand side of this equation converges to $0$ for $\epsilon\rightarrow0$. Now the claim follows as in the previous case.
			\end{proof}
		
				\begin{corollary}
				\label{corollary:representationConv0}
				Let $f\in \Conv(\C^n,\R)$, $\zeta\in D^{2n-k}$, $\tilde{\zeta}\in \tilde{D}^{2n-k+2}$. If $z\mapsto \zeta(|z|)$ is integrable with respect to $\Theta^n_{k,q}(f)$, then
				\begin{align*}
					T^n_{k,q}(\zeta)[f]=\int_{\C^n}\zeta(|z|)d\Theta^n_{k,q}(f,z).
				\end{align*}
				Similarly, if $z\mapsto \tilde{\zeta}(|z|)$ is integrable with respect to $\Upsilon^n_{k,q}(f)$, then
				\begin{align*}
					Y^n_{k,q}(\tilde{\zeta})[f]=\int_{\C^n}\tilde{\zeta}(|z|)d\Upsilon^n_{k,q}(f,z)
				\end{align*}
				This is in particular the case if $f$ is of class $C^2$ in a neighborhood of the origin.
			\end{corollary}
			\begin{proof}
				The first two claims are a direct consequence of Theorem \ref{theorem:PrincipalValue}. 
				If $f$ is on class $C^2$ in a neighborhood of the origin, then $\Theta^n_{k,q}(f)$ and $\Upsilon^n_{k,q}(f)$ are absolutely continuous with respect to the Lebesgue measure on this neighborhood with continuous densities. In particular, these densities are bounded on a compact neighborhood of the origin. It is then easy to see that $z\mapsto\zeta(|z|)$ is integrable with respect to $\Theta^n_{k,q}(f)$ and that $z\mapsto \tilde{\zeta}(|z|)$ is integrable with respect to $\Upsilon^n_{k,q}(f)$. 
			\end{proof}
	Corollary \ref{corollary:representationConv0} naturally leads to the question under which conditions on $\zeta\in D^{2n-k}$ the function $z\mapsto \zeta(|z|)$ is integrable with respect to $\Theta^n_{k,q}(f)$ for every $f\in\Conv(\C^n,\R)$. This question is completely answered by the following result.
	\begin{corollary}
		\label{corollary:IntegrabilityTheta} Let $0\le k\le 2n-1$. The following are equivalent for $\zeta\in D^{2n-k}$:
		\begin{enumerate}
			\item The function $z\mapsto \zeta(|z|)$ is integrable with respect to $\Theta^n_{k,q}(f)$ for every $f\in \Conv(\C^n,\R)$.
			\item The function $z\mapsto \zeta(|z|)$ is integrable with respect to $\Theta^n_{k,q}(|\cdot|)$.
			\item $\int_0^\infty |\zeta(r)| r^{2n-k-1}dr<\infty$.
			\item $|\zeta|\in D^{2n-k}$.
		\end{enumerate}
	\end{corollary}
	\begin{proof}
		The implications 1. $\Rightarrow$ 2. and 3. $\Rightarrow$ 4. are trivial, while 2. $\Rightarrow$ 3. follows from Corollary \ref{corollary:ZetaNotIntegrableWRTTheta}. We thus only have to show the implication 4. $\Rightarrow$ 1.\\
		Assume that $|\zeta|\in D^{2n-k}$. Monotone convergence and Theorem \ref{theorem:PrincipalValue} imply for $f\in\Conv(\C^n,\R)$
		\begin{align*}
			\int_{\C^n} |\zeta(|z|)|d\Theta^n_{k,q}(f)=&\lim_{\epsilon\rightarrow0}\int_{\C^n\setminus B_\epsilon(0)} |\zeta(|z|)|d\Theta^n_{k,q}(f)=T^n_{k,q}(|\zeta|)[f]<\infty.
		\end{align*}
		Thus $z\mapsto |\zeta(|z|)|$ is integrable with respect to $\Theta^n_{k,q}(f)$ for every $f\in\Conv(\C^n,\R)$.
	\end{proof}
	For $\Upsilon^n_{k,q}$, we obtain the following partial answer.
	\begin{corollary}
		\label{corollary:IntegrabilityUpsilon}
		If $\tilde{\zeta}\in \tilde{D}^{2n-k+2}$ satisfies $|\tilde{\zeta}|\in D^{2n-k+2}$, then $z\mapsto \tilde{\zeta}(|z|)$ is integrable with respect to $\Upsilon^n_{k,q}(f)$ for every $f\in\Conv(\C^n,\R)$.\\
		If $f\in\Conv(\C^n,\R)$ is a functions such that both $\mathcal{B}^n_{k,q}(f)$ and $\mathcal{C}^n_{k,q}(f)$ are non-negative or non-positive, then $z\mapsto \tilde{\zeta}(|z|)$ is integrable with respect to $\mathcal{B}^n_{k,q}(f)$, $\mathcal{C}^n_{k,q}(f)$, and $\Upsilon^n_{k,q}(f)$ for every $\tilde{\zeta}\in \tilde{D}^{2n-k+2}$.
	\end{corollary}
	\begin{proof}
		Since the form inducing $\Upsilon^n_{k,q}$ is homogeneous of degree $(2n-k+2)$, the first claim follows from Corollary \ref{corollary:GeneralIntegrabilityMA}.\\
		If $\mathcal{B}^n_{k,q}(f)$ and $\mathcal{C}^n_{k,q}(f)$ are non-negative or non-positive, we obtain the estimates
		\begin{align*}
			\int_{\C^n}|\tilde{\zeta}(|z|)|d|\mathcal{B}^n_{k,q}(f)|\le C_{n,k,q}\left(\sup_{|z|\le R+1}|f(z)|\right)^k \|\tilde{\zeta}\|_{\tilde{D}^{2n-k+2}},\\
			\int_{\C^n}|\tilde{\zeta}(|z|)|d|\mathcal{C}^n_{k,q}(f)|\le C_{n,k,q}\left(\sup_{|z|\le R+1}|f(z)|\right)^k \|\tilde{\zeta}\|_{\tilde{D}^{2n-k+2}}
		\end{align*}
		by approximation from Theorem \ref{theorem:Estimate_BC}. Thus $z\mapsto \tilde{\zeta}(|z|)$ is integrable with respect to both $\mathcal{B}^n_{k,q}(f)$, and $\mathcal{C}^n_{k,q}$, and thus also with respect to $\Upsilon^n_{k,q}(f)=\mathcal{B}^n_{k,q}(f)-\mathcal{C}^n_{k,q}(f)$.
	\end{proof}

	\subsection{Relation to the hermitian intrinsic volumes}
		\begin{lemma}
		\label{lemma:YkqVanishesOnRotationInvariantFunctions}
		For every $\zeta\in \tilde{D}^{2n-k+2}$, $Y^n_{k,q}(\zeta)[u_t]=0$ for all $t\ge 0$.
	\end{lemma}
	\begin{proof}
		If $\zeta$ is continuous, this follows directly from \ref{corollary:VanishingUpsilonRotationInvariantFunctions}, because $Y^n_{k,q}(\zeta)[u_t]$ is given by integrating $\zeta$ with respect to $\Upsilon^n_{k,q}(u_t)$, which vanishes as $u_t$ is rotation invariant. Since continuous functions are dense in $\tilde{D}^{2n-k+2}_R$ by Lemma \ref{lemma:DensityTildeD_continuousFunctions} and $Y^n_{k,q}$ is continuous, the claim follows.
	\end{proof}
	\begin{lemma}
		\label{lemma:BehaviorT_u_t}
		\begin{align*}
			T^n_{k,q}(\zeta)[u_t]=& 2n\omega_{2n}\binom{n}{k-2q,q}2^{k-2q}\left( (2n-k)\int_t^\infty\zeta(r)r^{2n-k-1}dr+t^{2n-k}\zeta(t)\right).
		\end{align*}
	\end{lemma}
	\begin{proof}
		By Theorem \ref{theorem:ThetaValuesOnUt} the equation holds if $\zeta\in C_c([0,\infty))$. Obviously the right hand side of the equation depends continuously on $\zeta\in D^{2n-k}$. As $T^n_{k,q}(\zeta)$ depends continuously on $\zeta\in D_R^{2n-k}$ and continuous functions are dense in $D^{2n-k}_R$ by Lemma \ref{lemma:DensityDNJ_continuousFunctions}, the claim follows.
	\end{proof}
	
	In \cite{AleskerValuationsconvexfunctions2019}, Alesker considered the map $T:\VConv(\R^n)\rightarrow\Val(\R^n)$, $\mu\mapsto \left[K\mapsto \mu(h_K)\right]$. For the valuations constructed above, we obtain the following relation to valuations on convex bodies.
	\begin{proposition}
		For $\zeta\in D^{2n-k}$, $\tilde{\zeta}\in \tilde{D}^{2n-k+2}$,
		\begin{align*}
			T^n_{k,q}(\zeta)[h_K]&=\mu_{k,q}(K)(2n-k)\omega_{2n-k}\int_0^\infty \zeta(r)r^{2n-k-1}dr,\\
			Y^n_{k,q}(\tilde{\zeta})[h_K]&=0
		\end{align*}
		for every $K\in\mathcal{K}(\C^n)$.
	\end{proposition}
	\begin{proof}
		The maps $K\mapsto T^n_{k,q}(\zeta)[h_K]$ and $K\mapsto Y^n_{k,q}(\tilde{\zeta})[h_K]$ define elements in $\Val_k(\C^n)^{\U(n)}$ whose restriction to the subspaces $E_{k,p}$ vanishes for $p\ne q$. Thus they are multiples of the hermitian intrinsic volume $\mu_{k,q}$ by Theorem \ref{theorem:BernigFuHermitianIntrinsicVolumes}. In order to obtain the constant, it is sufficient to evaluate them in $B_1(0)$, which corresponds to $h_{B_1(0)}=u_0$. $Y^n_{k,q}(\tilde{\zeta})[u_0]=0$ by Lemma \ref{lemma:YkqVanishesOnRotationInvariantFunctions}, while we can evaluate $T^n_{k,q}(\zeta)[u_0]$ using Lemma \ref{lemma:BehaviorT_u_t}. Since 
		\begin{align*}
			\mu_{k,q}(K)=\frac{1}{\omega_{2n-k}}\int_{B_1(0)} d\Theta^n_{k,q}(h_K)
		\end{align*}
		by \cite[Section 3.2]{BernigFuHermitianintegralgeometry2011} and \cite[Proposition 2.1.7]{AleskerFuTheoryvaluationsmanifolds.2008}, we can calculate $\mu_{k,q}(B_1(0))$ using Theorem \ref{theorem:ThetaValuesOnUt} by approximating the indicator $1_{B_1(0)}$ with continuous functions.
	\end{proof}

\section{Characterization of $\U(n)$-invariant valuations}
	\label{section:CharacterizationResults}
\subsection{Preliminary results for unitarily invariant valuations}
The proofs of the main results given below rely on the following three results obtained in \cite{KnoerrUnitarilyinvariantvaluations2021}. 
\begin{theorem}[\cite{KnoerrUnitarilyinvariantvaluations2021} Theorem 1]
	\label{theorem:VanishingPropertyUnitaryValuations}
	A valuation $\mu\in\VConv_k(\C^n)^{\U(n)}$ satisfies $\mu\equiv 0$ if and only if $\pi_{E_{k,q}*}\mu=0$ for all $\max(0,k-n)\le q\le\lfloor\frac{k}{2}\rfloor$.
\end{theorem}
Note that this implies that if $\mu_1,\mu_2\in\VConv_k(\C^n)^{\U(n)}$ satisfy $\pi_{E_{k,q}*}\mu_1=\pi_{E_{k,q}*}\mu_2$ for all $\max(0,k-n)\le q\le\lfloor\frac{k}{2}\rfloor$, then $\mu_1=\mu_2$. In particular, the decomposition in Theorem \ref{maintheorem:decomposition} is direct.\\

The second result shows that the decomposition holds for smooth valuations and provides a characterization of these functionals. It also implies that the components of a smooth valuation with respect to the decomposition are again smooth valuations.
\begin{theorem}[\cite{KnoerrUnitarilyinvariantvaluations2021} Theorem 2]
	\label{theorem:characterization_smooth_unitarily_invariant_valuations}
	For $1\le k\le 2n$,
	\begin{align*}
		\VConv_k(\C^n)^{\U(n),sm}=\bigoplus_{q=\max(0,k-n)}^{\lfloor\frac{k}{2}\rfloor}\VConv_{k,q}(\C^n)^{\U(n),sm}.
	\end{align*}
	Moreover, for every $\max(0,k-n)\le q\le \lfloor\frac{k}{2}\rfloor$ and $\mu\in \VConv_{k,q}(\C^n)^{\U(n),sm}$ there exist unique $\phi_q,\psi_q\in C_c^\infty([0,\infty))$ such that
	\begin{align*}
		\mu(f)=\int_{\C^n} \phi_q(|z|)d\Theta^n_{k,q}(f)+\int_{\C^n} \psi_q(|z|)d\Upsilon^n_{k,q}(f),
	\end{align*}
	where we set $\psi_q=0$ if $q=0$ or $q=\frac{k}{2}$.
\end{theorem}
Since some of the proofs given below rely on an approximation argument, we need to control the supports of the functions $\phi_q$ and $\psi_q$ in terms of the supports of the corresponding valuations.
\begin{proposition}[\cite{KnoerrUnitarilyinvariantvaluations2021} Proposition 4.17]
	Let $\mu\in\VConv_k(\C^n)^{\U(n),sm}$ be given by
\begin{align*}
	\mu(f)=\sum_{q=\max(0,k-n)}^{\lfloor\frac{k}{2}\rfloor}\int_{\C^n} \phi_q(|z|)d\Theta^n_{k,q}(f)+\sum_{q=\max(1,k-n)}^{\lfloor\frac{k-1}{2}\rfloor}\int_{\C^n} \psi_q(|z|)d\Upsilon^n_{k,q}(f)
\end{align*}
for $\phi_q,\psi_q\in C^\infty_c([0,\infty))$. If $R>0$ is such that $\supp\mu\subset B_R(0)$, then $\supp \phi_q\subset [0,R]$ and $\supp\psi_q\subset[0,R]$ for all all $q$.
\end{proposition}
\subsection{Preliminary considerations}	
	In this section we relate the integral transforms from Section \ref{section:IntegralTransforms} to the values of the valuations obtained from $T^n_{k,q}$ and $Y^n_{k,q}$ on sums of functions corresponding to a complex orthogonal decomposition $\C^n=X\oplus Y$. Let $u_s^X\in\Conv(X,\R)$ be given by $u_s^X(z_X)=\max(0,|z_X|-s)$. We similarly define $u_t^Y\in \Conv(Y,\R)$. As in Section \ref{section:MAOperators}, we consider these functions as elements of $\Conv(\C^n,\R)$ using the constant extension.
	\begin{lemma}
		\label{lemma:ValuesYOnCyclinderFunctions}
		Let $\C^n=X\oplus Y$ be a complex orthogonal decomposition, $\zeta\in D^{2n-k}$, and $\tilde{\zeta}\in \tilde{D}^{2n-k+2}$.\\
		For $\dim_\C X=1$, $n\ge 2$, $k\ge 2$,
		\begin{align*}
			\overline{T^n_{k,q}(\zeta)}(u_s^X[1],u_t^Y[k-1])=&\frac{2\omega_2(2n-2)\omega_{2n-2}2^{k-2q}}{n-1}\binom{n-1}{k-1-2q,q}\mathcal{R}^{1,2n-k-1}\left(\zeta\right)[s,t],
		\end{align*}
		and for $k\ge 3$,
		\begin{align*}
			\overline{T^n_{k,q}(\zeta)}(u_0^X[2],u_t^Y[k-2])=&\frac{2\omega_2(2n-2)\omega_{2n-2}2^{k-2q}}{n-1}\binom{n-1}{k-2q,q-1}\mathcal{R}^{2n-k}\left(\zeta\right)[t],\\
			\overline{Y_{k,q}(\tilde{\zeta})}(u_0^X[2],u_t^Y[k-2])=&-\frac{2\omega_2(2n-2)\omega_{2n-2}2^{k-2q}}{2q(k-2)(n-1)}\binom{n-1}{k-2q,q-1}\mathcal{P}^{2n-k+2}(\tilde{\zeta})[t].
		\end{align*}
		For $\dim_CX=2$, $n\ge 3$, $k\ge 5$, $q\ge 2$,
		\begin{align*}
			\overline{T^n_{k,q}(\zeta)}(u_0^X[4],u_t^Y[k-4])=&\frac{4\omega_4(2n-4)\omega_{2n-4}2^{k-2q}}{n-2}\binom{n-2}{k-2q,q-2}\mathcal{R}^{2n-k}\left(\zeta\right)[t],\\
			\overline{Y_{k,q}(\tilde{\zeta})}(u_0^X[4],u_t^Y[k-4])=&-\frac{4\omega_4(2n-4)\omega_{2n-4}2^{k-2q}}{q(k-4)(n-2)}\binom{n-2}{k-2q,q-2}\mathcal{P}^{2n-k+2}(\tilde{\zeta})[t].
		\end{align*}
	\end{lemma}
	\begin{proof}
		Note that in all of these cases both sides of the equation depend continuously on $\zeta\in D^{2n-k}_R$ and $\tilde{\zeta}\in \tilde{D}^{2n-k+2}_R$ respectively. For the left hand sides, this follows from Theorem \ref{theorem:ExistenceT} in combination with Proposition \ref{lemma:continuityPolarization}, while it follows from Lemma \ref{lemma:PaIsomorphism} and Lemma \ref{lemma:ContinuityRab} for the right hand sides. As continuous functions are dense in $D^{2n-k}_R$ and $\tilde{D}^{2n-k+2}_R$ by Lemma \ref{lemma:DensityDNJ_continuousFunctions} and Lemma \ref{lemma:DensityTildeD_continuousFunctions} respectively, it is sufficient to consider $\zeta,\tilde{\zeta}\in C_c([0,\infty))$. In this case, Lemma \ref{lemma:DecompositionMeasuresDirectSum} implies
		\begin{align*}
			\overline{T^n_{k,q}(\zeta)}(u^X_s[j],u^Y_t[k-j])=&\frac{1}{j!(k-j)!}\frac{\partial^k}{\partial\lambda^j\partial \delta^{k-j}}\Big|_0T^n_{k,q}(\zeta)[\lambda u^X_s+\delta u^Y_t]\\
			=&\int_{\C^n}\zeta(|(z_X,z_Y)|)d\left[\Theta^{X}_{j,\lfloor\frac{j}{2}\rfloor}(u_t^X,z_X)\otimes \Theta^{Y}_{k-j,q-\lfloor\frac{j}{2}\rfloor}(u_t^Y,z_Y)\right].
		\end{align*}
		This expression can easily be evaluated using Lemma \ref{lemma:BehaviorT_u_t}, which yields the desired result. As we only consider $Y^n_{k,q}$ if $q\ge \max(1,k-n)$, Lemma \ref{lemma:DecompositionMeasuresDirectSum} similarly implies
		\begin{align*}
			&\overline{Y_{k,q}(\tilde{\zeta})}(u^X_s[2],u^Y_t[k-2])\\
			=&\frac{1}{q}\int_{\C^n}\zeta(|(z_X,z_Y)|)d\mathcal{B}^{X}_{2,1}(u_s^X,z_X) d\Theta^{Y}_{k-2,q-1}(u_t^Y,z_Y)\\
			&+\int_{\C^n}\zeta(|(z_X,z_Y)|)d\left[\Theta^{X}_{2,1}(u_s^X,z_X)\otimes\left(\frac{q-1}{q}\mathcal{B}^{Y}_{k-2,q-1}-\mathcal{C}^Y_{k-2,q-1}\right)(u_t^Y,z_Y)\right],
		\end{align*}
		which we can evaluate using Theorem \ref{theorem:ThetaValuesOnUt}. The claim follows.
	\end{proof}
	
	\begin{corollary}
		\label{corollary:InjectivityTY}
		Let $\max(1,k-n)\le q\le \left\lfloor\frac{k-1}{2}\right\rfloor$. The map
		\begin{align*}
			T^n_{k,q}\oplus Y^n_{k,q}:D^{2n-k}\oplus \tilde{D}^{2n-k+2}&\rightarrow\VConv_{k,q}(\C^n)^{\U(n)}\\
			(\zeta,\tilde{\zeta})&\mapsto T^n_{k,q}(\zeta)+Y^n_{k,q}(\tilde{\zeta})
		\end{align*}
		is injective.
	\end{corollary}
	\begin{proof}
		First note that this case only occurs for $n\ge 2$, $k\ge 3$.	Lemma \ref{lemma:ValuesYOnCyclinderFunctions} thus shows that $Y^n_{k,q}$ is injective. Since the values of $T^n_{k,q}(\zeta)$ on rotation invariant functions completely determine $\zeta$ by Lemma \ref{lemma:BehaviorT_u_t} and Lemma \ref{lemma:InjectivityTransformR}, and $Y^n_{k,q}(\tilde{\zeta})$ vanishes on rotation invariant functions, the claim follows.
	\end{proof}

\subsection{Proof of the decomposition}
	In order to simplify the notation, set
\begin{align*}
	\VConv^R_{k}(\C^n)^{\U(n)}:=&\VConv_{k}(\C^n)^{\U(n)}\cap \VConv_{B_R(0)}(\C^n),\\
	\VConv^R_{k,q}(\C^n)^{\U(n)}:=&\VConv_{k,q}(\C^n)^{\U(n)}\cap \VConv_{B_R(0)}(\C^n).
\end{align*}

In this section we will establish the following version of Theorem \ref{maintheorem:decomposition}.
\begin{theorem}
	\label{theorem:DecompositionRefinedVersion}
	Let $R>0$. For $n\in\mathbb{N}$, $0\le k\le 2n$:
	\begin{align}
		\label{equation:induction_hypothesis}
		\VConv^R_{k}(\C^n)^{\U(n)} =\bigoplus_{q=\max(0,k-n)}^{\lfloor\frac{k}{2}\rfloor}\VConv^R_{k,q}(\C^n)^{\U(n)}.
	\end{align}
\end{theorem}
This directly implies Theorem \ref{maintheorem:decomposition} since every valuation in $\VConv_k(\C^n)^{\U(n)}$ has compact support. Note that the cases $n=1$ as well as $k=0,1,2n-1$ and $2n$ are already covered by Theorem \ref{theorem:VanishingPropertyUnitaryValuations}, as there is only one space in the decomposition in this case. It is therefore sufficient to establish Theorem \ref{theorem:DecompositionRefinedVersion} for homogeneous valuations of degree $2\le k\le 2n-2$.\\

As discussed in the introduction, the proof will proceed by induction. Since the result holds for $n=1$, let us thus assume that $n\ge 2$ and that Theorem \ref{theorem:DecompositionRefinedVersion} holds for $R>0$, all dimensions $l\le n-1$ and all $0\le k\le 2l$. This has the following direct implication.
\begin{corollary}
	\label{corollary:lowerdimensional_decomposition_topological_direct_sum} For $1\le l\le n-1$,
	\eqref{equation:induction_hypothesis} is a topological direct sum decomposition, that is, the induced projection 
	\begin{align*}
		P^{l}_{i,q}:\VConv^R_i(\C^l)^{\U(l)}\rightarrow\VConv^R_{i,q}(\C^l)^{\U(l)}
	\end{align*} is continuous for each $\max(0,i-l)\le q\le \lfloor\frac{i}{2}\rfloor$.
\end{corollary}
\begin{proof}
	The spaces $\VConv^R_{i,q}(\C^l)^{\U(l)}$ are closed subspaces of the Banach space $\VConv^R_{i}(\C^l)^{\U(l)}$ and thus in particular Banach spaces themselves. In particular, \eqref{equation:induction_hypothesis} expresses the Banach space $\VConv^R_i(\C^l)^{\U(l)}$ as a finite direct sum of Banach subspaces, which is thus a topological decomposition by the open mapping theorem. 
\end{proof}

Let us now briefly discuss the general strategy of our proof. By Theorem \ref{theorem:VanishingPropertyUnitaryValuations} the sum on the right hand side of \eqref{equation:induction_hypothesis} is direct. It is thus sufficient to decompose a given valuation into the relevant components. We will inductively construct \emph{continuous} linear maps
\begin{align*}
	\VConv^R_k(\C^n)^{\U(n)}\rightarrow\VConv^R_{k,q}(\C^n)^{\U(n)}
\end{align*}
and verify that these are the desired projections by examining their behavior on the dense subspace of smooth valuations, where we know that the decomposition holds (compare Theorem \ref{theorem:characterization_smooth_unitarily_invariant_valuations}).\\

As a first step, we consider the decomposition of the polarization of $\mu$ under proper complex orthogonal  decompositions of $\C^n$. For $1\le l\le n-1$ let us set $X=\C^l\times\{0\}^{n-l}$, $Y=\{0\}^l\times\C^{n-l}$ such that $\C^n=X\oplus Y$ is a decomposition of $\C^n$ into two lower dimensional orthogonal complex subspaces. 
We will denote by $X_{i,q}$ and $Y_{i,q}$ $i$-dimensional subspaces of $X$ and $Y$ that belong to the same orbit as $\C^q\times \R^{i-2q}$ under the natural operation of $\U(l)$ and $\U(n-l)$ on $X$ and $Y$. 
	\begin{lemma}
		\label{lemma:continous_decomposition_polarization}
		For $2\le k\le 2n-2$ let $\mu\in\VConv^R_k(\C^n)^{\U(n)}$ be given.
		For each $1\le i\le k-1$ and every $\max(0,l-i)\le q\le \lfloor \frac{l}{2}\rfloor$, $\max(0,n-l-(k-i))\le p\le \lfloor \frac{n-l}{2}\rfloor$ there exist unique maps $\mu^l_{(i,q),(k-i,p)}:\Conv(X,\R)\times\Conv(Y,\R)\rightarrow\R$
		with the following properties:
		\begin{enumerate}
			\item For all $(f_X,f_Y)\in\Conv(X,\R)\times\Conv(Y,\R)$:
			\begin{align*}
				\bar{\mu}(f_X[i],f_Y[k-i])=\sum_{q=\max(0,i-l)}^{\lfloor \frac{l}{2}\rfloor}\sum_{p=\max(0,k-i-(n-l))}^{\lfloor \frac{n-l}{2}\rfloor}\mu^l_{(i,q),(k-i,p)}(f_X,f_Y)
			\end{align*}
			\item $\mu^l_{(i,q),(k-i,p)}:\Conv(X,\R)\times\Conv(Y,\R)\rightarrow\R$ is jointly continuous.
			\item For every $f_X\in\Conv(X,\R)$: $\mu^l_{(i,q),(k-i,p)}(f_X,\cdot)\in\VConv^R_{k-i,p}(Y)^{\U(n-l)}$.
			\item For every $f_Y\in\Conv(Y,\R)$: $\mu^l_{(i,q),(k-i,p)}(\cdot,f_Y)\in\VConv^R_{i,q}(X)^{\U(l)}$.
		\end{enumerate}
	\end{lemma}
	\begin{proof}
		Let us first establish that maps with these properties exist. The polarization $\bar{\mu}$ of $\mu$ is jointly continuous, so the map
		\begin{align*}
			(f_X,f_Y)\mapsto \bar{\mu}(f_X[i],f_Y[k-i])
		\end{align*}
		is jointly continuous. Moreover, it is a dually epi-translation invariant valuation in each argument, homogeneous of degree $i$ and $k-i$ respectively. Obviously, this map is $U(l)$-invariant in the first and $\U(n-l)$-invariant in the second argument. We thus obtain a well defined and continuous map
		\begin{align*}
			\mu_1:\Conv(X,\R)&\rightarrow\VConv_{k-i}(Y)^{\U(n-l)}\\
			f_X&\mapsto \bar{\mu}(f_X[i],\cdot[k-i]),
		\end{align*}
		which is in addition a dually epi-translation and $\U(l)$-invariant valuation. In addition, its support is contained in $B_R(0)$ by Lemma \ref{lemma:SupportPolarization}. According to Corollary \ref{corollary:lowerdimensional_decomposition_topological_direct_sum},
		\begin{align*}
			\VConv^R_{k-i}(Y)^{\U(n-l)}=\bigoplus_{p=\max(0,k-i-(n-l))}^{\lfloor \frac{n-l}{2}\rfloor}\VConv^R_{k-i,p}(Y)^{\U(n-l)}
		\end{align*}
		is a topological decomposition since $\dim_\C Y=n-l<n$. Using the continuous projections $P^Y_{k-i,p}:\VConv^R_{k-i}(Y)^{\U(n-l)}\rightarrow \VConv^R_{k-i,p}(Y)^{\U(n-l)}$, we obtain for $f_X\in\Conv(X,\R)$
		\begin{align*}
			\mu_1(f_X)=\sum_{p=\max(0,k-i-(n-l))}^{\lfloor \frac{n-l}{2}\rfloor}P^Y_{k-i,p}[\mu_1(f_X)]\in\VConv^R_k(Y)^{\U(n)}.
		\end{align*}
		Note that the map
		\begin{align*}
			\mu^p:\Conv(X,\R)\times\Conv(Y,\R)&\rightarrow\R\\
			(f_X,f_Y)&\mapsto P^Y_{k-i,p}[\mu_1(f_X)](f_Y)
		\end{align*}
		is continuous: If $(f_{X,j})_j$ and $(f_{Y,j})_j$ are sequences in $\Conv(X,\R)$ and $\Conv(Y,\R)$ converging to $f_X$ and $f_Y$ respectively, then  $P^Y_{k-i,p}[\mu_1(f_{X,j})]$ converges to $P^Y_{k-i,p}[\mu_1(f_X)]$ in $\VConv^R_k(Y)^{\U(n-l)}$ as $P^Y_{k-i,p}$ is continuous. In other words, it converges uniformly on compact subsets in $\Conv(Y,\R)$. As the set $\{f_{Y,j}:j\in\mathbb{N}\}\cup \{f_Y\}\subset\Conv(Y,\R)$ is compact, we thus obtain
		\begin{align*}
			\lim\limits_{j\rightarrow\infty}P^Y_{k-i,p}[\mu_1(f_{X,j})](f_{Y,j})=P^Y_{k-i,p}[\mu_1(f_X)](f_Y),
		\end{align*}
		which shows the claim. It is easy to see that $\mu^p$ is a dually epi-translation and $\U(l)$-invariant valuation homogeneous of degree $i$ in the first argument. We thus obtain a continuous map
		\begin{align*}
			\tilde{\mu}^p_2:\Conv(Y,\R)&\rightarrow\VConv^R_{i}(X)^{\U(l)}\\
			f_Y&\mapsto \mu^p(\cdot,f_Y).
		\end{align*}
		As before, 
		\begin{align*}
			\VConv^R_{i}(X)^{\U(l)}=\bigoplus_{q=\max(0,i-l)}^{\lfloor \frac{l}{2}\rfloor}\VConv^R_{i,q}(X)^{\U(l)}
		\end{align*}
		is a topological decomposition by Corollary \ref{corollary:lowerdimensional_decomposition_topological_direct_sum} as $\dim_\C X=l<n$. Denoting the continuous projections onto the different components by $P^X_{i,q}$ and setting $\mu_2^{q,p}(f_Y):=P^X_{i,q}(\tilde{\mu}_2^p(f_Y))\in\VConv^R_{i,q}(X)^{\U(l)}$, we obtain
		continuous maps $\mu_2^{q,p}:\Conv(Y,\R)\rightarrow\VConv^R_{i,q}(X)^{\U(l)}$ that satisfy
		\begin{align*}
			\tilde{\mu}^p_2(f_Y)=\sum_{q=\max(0,i-l)}^{\lfloor \frac{l}{2}\rfloor}\mu_2^{q,p}(f_Y)\in\VConv^R_{i}(X)^{\U(l)}.
		\end{align*}
		We may now define
		\begin{align*}
			\mu^l_{(i,q),(k-i,p)}(f_X,f_Y):=\mu_2^{q,p}(f_Y)[f_X].
		\end{align*}
		As before, one easily verifies that $\mu^l_{(i,q),(k-i,p)}$ is in fact jointly continuous. Obviously, $\mu^l_{(i,q),(k-i,p)}$ also satisfies the remaining properties.\\
		
		To see that the functionals $\mu^l_{(i,q),(k-i,p)}$ are unique, assume that $\tilde{\mu}^l_{(i,q),(k-i,p)}$ is a second family of functionals satisfying these properties. By restricting $\bar{\mu}(f_X[i],f_Y[k-i])$ in the first argument to functions defined on $X_{i,q}$, we obtain
		\begin{align*}
			\sum_{p=\max(0,k-i-(n-l))}^{\lfloor \frac{n-l}{2}\rfloor}\mu^l_{(i,q),(k-i,p)}(\cdot,f_Y)=\sum_{p=\max(0,k-i-(n-l))}^{\lfloor \frac{n-l}{2}\rfloor}\tilde{\mu}^l_{(i,q),(k-i,p)}(\cdot,f_Y) \quad\text{on }\Conv(X_{i,q},\R)
		\end{align*}
		for all $f_Y\in\Conv(Y,\R)$. As this holds for all $\max(0,l-i)\le q\le \lfloor\frac{l}{2}\rfloor$, this equation holds in fact on $\Conv(X,\R)$ due to Theorem \ref{theorem:VanishingPropertyUnitaryValuations}. Thus
		\begin{align*}
			\sum_{p=\max(0,k-i-(n-l))}^{\lfloor \frac{n-l}{2}\rfloor}\mu^l_{(i,q),(k-i,p)}(f_X,\cdot)=\sum_{p=\max(0,k-i-(n-l))}^{\lfloor \frac{n-l}{2}\rfloor}\tilde{\mu}^l_{(i,q),(k-i,p)}(f_X,\cdot)
		\end{align*}
		for all $f_X\in\Conv(X,\R)$. But both sides are a decomposition of the same valuation in $\VConv^R_{k-i}(Y)^{\U(n-l)}=\bigoplus_{p=\max(0,k-i-(n-l))}^{\lfloor \frac{n-l}{2}\rfloor}\VConv^R_{k-i,p}(Y)^{\U(n-l)}$. As this is a direct sum decomposition, we deduce $\mu^l_{(i,q),(k-i,p)}(f_X,\cdot)=\tilde{\mu}^l_{(i,q),(k-i,p)}(f_X,\cdot)$ for all $f_X\in\Conv(X,\R)$ and all $\max(0,k-i-(n-l))\le p\le \lfloor \frac{n-l}{2}\rfloor$, i.e. the functionals are unique.
	\end{proof}
	\begin{definition}
		For any $\mu\in\VConv^R_k(\C^n)^{\U(n)}$, we set
		\begin{align*}
			P^{l}_{(i,q),(k-i,p)}(\mu):=\mu^l_{(i,q),(k-i,p)}.
		\end{align*}
	\end{definition}
	\begin{corollary}
		\label{corollaryPLinear}
		The map $\mu\mapsto P^l_{(i,q),(k-i,p)}(\mu)$ is linear.
	\end{corollary}
	\begin{proof}
		This follows directly from the uniqueness of the decomposition in Lemma \ref{lemma:continous_decomposition_polarization}.
	\end{proof}

	\begin{lemma}
		\label{lemma:continuityDecompositionFunctionNu}
		If $(\mu_j)_j$ is a sequence in $\VConv_k(\C^n)^{\U(n)}$ converging to $\mu_0$, then $P^l_{(i,q),(k-i,p)}(\mu_j)$ converges uniformly to $P^l_{(i,q),(k-i,p)}(\mu_0)$ on compact subsets of $\Conv(X,\R)\times \Conv(Y,\R)$.
	\end{lemma}
	\begin{proof}
		It is sufficient to show that the convergence is uniform on sets of the form $K_X\times K_Y$, where $K_X\subset \Conv(X,\R)$ and $K_Y\subset\Conv(Y,\R)$ are compact. As $P^l_{(i,q),c(k-i,p)}$ is linear, it is furthermore sufficient to consider the case where $\mu_0=0$, that is, $(\mu_j)_j$ converges uniformly to $0$ on compact subsets of $\Conv(\C^n,\R)$. Recall from the proof of Lemma \ref{lemma:continous_decomposition_polarization} that $P^l_{(i,q),(k-i,p)}(\mu)$ is given for $\mu\in\VConv_k(\C^n)^{\U(n)}$ by 
		\begin{align*}
			P^l_{(i,q),(k-i,p)}(\mu)[f_X,f_Y]=P^X_{i,q}(\tilde{\mu}_2^p(f_Y))[f_X],
		\end{align*}
		for the valuation $\tilde{\mu}_2^p(f_Y)\in \VConv_{i}(X)^{\U(l)}$ given by
		\begin{align*}
			\tilde{\mu}_2^p(f_Y)[f_X]=P^Y_{k-i,p}(\bar{\mu}(f_X[i],\cdot[k-i]))[f_Y]
		\end{align*}
		Denote the norms on $\VConv_{B_R(0)}(X)$ and $\VConv_{B_R(0)}(Y)$ defined in Lemma \ref{lemma:NormsValuationsBoundedSupport} by $\|\cdot\|_{X;R}$ and $\|\cdot\|_{Y;R}$, respectively. Note that $\mu_X\mapsto\sup_{f_X\in K_X}|\mu_X(f_X)|$ and $\mu_Y\mapsto \sup_{f_Y\in K_Y}|\mu_Y(f_Y)|$  define continuous semi-norms. Because the projections $P^X_{i,q}$ and $P^Y_{k-i,p}$ are continuous on $\VConv^R_i(X)^{U(l)}$ and $\VConv^R_{k-i}(Y)^{n-l}$, we obtain constants $C_{K_X}$, $C_{K_Y}>0$ such that
		\begin{align*}
			\sup_{f_X\in K_X}|P^X_{i,q}(\mu_X)|\le C_{K_X}\|\mu_X\|_{X;R},\\
			\sup_{f_Y\in K_Y}|P^Y_{k-i,p}(\mu_Y)|\le C_{K_Y}\|\mu_X\|_{Y;R},
		\end{align*}
		for $\mu_X\in\VConv^R_i(X)^{U(l)}$ and $\mu_Y\in \VConv^R_{k-i}(Y)^{n-l}$. If we denote by $A_R(X)\subset \Conv(X,\R)$ and $A_R(Y)\subset \Conv(Y,\R)$ the subsets of functions that are bounded by $1$ on $B_{R+1}(0)$, this implies
		\begin{align*}
			&\sup_{(f_X,f_Y)\in K_X\times K_Y}|P^l_{(i,q),(k-i,p)}(\mu)[f_X,f_Y]|=\sup_{(f_X,f_Y)\in K_X\times K_Y}|P^X_{i,q}(\tilde{\mu}_2^p(f_Y))[f_X]|\\
			\le & C_{K_X}\sup_{f_Y\in K_Y}\|\tilde{\mu}^p_2(f_Y)\|_{X;R}=C_{K_X}\sup_{f_Y\in K_Y, f_X\in A_R(X)}|\tilde{\mu}^p_2(f_Y)[f_X]|\\
			=&C_{K_X}\sup_{f_Y\in K_Y, f_X\in A_R(X)}|P^Y_{k-i,p}(\bar{\mu}(f_X[i],\cdot[k-i]))[f_Y]|\\
			\le& C_{K_X}C_{K_Y}\sup_{f_X\in A_R(X)}\|\bar{\mu}(f_X[i],\cdot[k-i])\|_{Y;R}\\
			=&C_{K_X}C_{K_Y}\sup_{f_X\in A_R(X), f_Y\in A_R(Y)}|\bar{\mu}(f_X[i],f_Y[k-i])|\\
			\le& C_{K_X}C_{K_Y}\sup_{f,h\in A_R}|\bar{\mu}(f_X[i],f_Y[k-i])|\le c_k C_{K,X}C_{K,Y}\|\mu\|_{R,1},
		\end{align*}
		where $c_k>0$ is the constant from Lemma \ref{lemma:continuityPolarization}.\\
		As the norm $\|\cdot\|_R$ metrizes the topology of uniform convergence on compact subsets on $\VConv_{B_R(0)}(\C^n)$ and $P^l_{(i,q),(n-i,p)}$ is linear by Corollary \ref{corollaryPLinear}, the claim follows.
	\end{proof}

	\begin{lemma}
		\label{lemma:DecompositionCylinderFunctionsVConvKQ}
		If $\mu\in\VConv_{k,q}(\C^n)^{\U(n)}$, then $\mu^l_{(i,p_1),(k-i,p_2)}=0$ for $p_1+p_2\ne q$.
	\end{lemma}
	\begin{proof}
		Let $f_X\in \Conv(X_{i,p_1},\R)$, $f_Y\in\Conv(Y_{k-i,p_2},\R)$. Then $s\pi_{X_{i,p_1}}^*f_X+t\pi_{Y_{k-i,p_2}}^*f_Y\in\Conv(X_{i,p_1}\oplus Y_{k-i,p_2},\R)$ for $s,t\ge 0$, where $X_{i,p_1}\oplus Y_{k-i,p_2}$ belongs to the same $\U(n)$ orbit as $E_{k,p_1+p_2}$. If $p_1+p_2\ne q$, this implies 
		\begin{align*}
			0=\frac{1}{i!(k-i)!}\frac{\partial^k}{\partial s^i\partial t^{k-i}}\Big|_0\mu\left(s\pi_{X_{i,p_1}}^*f_X+t\pi_{Y_{k-i,p_2}}^*f_Y\right)=\bar{\mu}(\pi_{X_{i,p_1}}^*[i],\pi_{Y_{k-i,p_2}}^*f_Y[k-i]).
		\end{align*} On the other hand, properties 3. and 4. in Lemma \ref{lemma:continous_decomposition_polarization} imply
		\begin{align*}
			0=\bar{\mu}(\pi_{X_{i,p_1}}^*f_X[i],\pi_{Y_{k-i,p_2}}^*f_Y[k-i])=\mu^l_{(i,p_1),(k-i,p_2)}(\pi_{X_{i,p_1}}^*f_X,\pi_{Y_{k-i,p_2}}^*f_Y),
		\end{align*}
		since all other components in the decomposition vanish on these functions. If we fix $f_X\in \Conv(X_{i,p_1},\R)$, Theorem \ref{theorem:VanishingPropertyUnitaryValuations} thus implies that
		$\mu^l_{(i,p_1),(k-i,p_2)}(\pi_{X_{i,p_1}}^*f_X,\cdot)$, which belongs to $\VConv_{k-i,p_2}(Y)^{\U(n-l)}$, has to vanish identically. Thus 
		\begin{align*}
			\mu^l_{(i,p_1),(k-i,p_2)}(\pi_{X_{i,p_1}}^*f_X,f_Y)=0\quad \text{for any fixed}~f_Y\in\Conv(Y,\R),
		\end{align*}
		which due to Theorem \ref{theorem:VanishingPropertyUnitaryValuations} implies that $\mu^l_{(i,p_1),(k-i,p_2)}(\cdot,f_Y)\in \VConv_{i,p_1}(X)^{\U(l)}$ has to vanish identically. Thus $\mu^l_{(i,p_1),(k-i,p_2)}=0$.
	\end{proof}

		\begin{lemma}
		\label{lemma:DefinitionUQ}
		Let $1\le k\le 2n$. For $\max(0,k-n)\le q\le \lfloor\frac{k}{2}\rfloor$ define $U_q:\VConv^R_{k}(\C^n)^{\U(n)}\rightarrow C_c([0,\infty)^2)$ by
		\begin{align*}
			U_q(\mu)[s,t]:=&\frac{1}{\frac{2\omega_2(2n-2)\omega_{2n-2}2^{k-2q}}{n-1}\binom{n-1}{k-1-2q,q}}\mu^1_{(1,0),(k-1,0)}(u^X_s,u^Y_t).
		\end{align*}
		Similarly, for $\max(1,k-n)\le q\le \lfloor\frac{k}{2}\rfloor$ define $\tilde{V}_q,\tilde{W}_q:\VConv^R_{k}(\C^n)^{\U(n)}\rightarrow C_c([0,\infty))$ by
		\begin{align*}
			\tilde{V}_q(\mu)[t]:=&\frac{1}{\frac{2\omega_2(2n-2)\omega_{2n-2}2^{k-2q}}{n-1}\binom{n-1}{k-2q,q-1}}\mu^1_{(2,1),(k-2,q-1)}(u^X_0,u^Y_t) && n\ge 2, k\ge 3, q\ge 1,\\
			\tilde{W}_q(\mu)[t]:=&\frac{1}{\frac{4\omega_4(2n-4)\omega_{2n-4}2^{k-2q}}{n-2}\binom{n-2}{k-2q,q-2}}\mu^2_{(4,2),(k-4,q-2)}(u^X_0,u^Y_t) && n\ge 3, k\ge 5, q\ge 2.
		\end{align*}
		Then $U_q$, $V_q$ and $W_q$ are well defined and continuous.
	\end{lemma}
	\begin{proof}
		Note that $(f_X,f_Y)\rightarrow\mu^l_{(i,p_1),(k-i,p_2)}(f_X,f_Y)$ is jointly continuous by Lemma \ref{lemma:continous_decomposition_polarization}. Since $s\mapsto u_s^X$ and $t\mapsto u_t^Y$ define continuous curves in $\Conv(X,\R)$ and $\Conv(Y,\R)$, respectively, the functions $U_q(\mu),\tilde{V}_q(\mu)$ and $\tilde{W}_q(\mu)$ are continuous as compositions of continuous functions. Moreover, it follows from Lemma \ref{lemma:continous_decomposition_polarization} that these functions are supported on $B_R^X\times B_R^Y$ since $u_s^X$ and $u_t^Y$ coincide with the zero function on a neighborhood of the support of $\mu^l_{(i,p_1),(k-i,p_2)}$ for $s,t>R$. If $(\mu_j)_j$ is a sequence in $\VConv_k^R(\C^n)^{\U(n)}$ converging to $\mu$, then $(\mu_j)^l_{(i,p_1),(k-i,p_2)}=P^l_{(i,p_1),(k-i,p_2)}(\mu_j)$ converges uniformly to $P^l_{(i,p_1),(k-i,p_2)}(\mu)$ on compact sets in $\Conv(X,\R)\times\Conv(Y,\R)$ by Lemma \ref{lemma:continuityDecompositionFunctionNu}, so in particular it converges uniformly on $\{(u_s^X,u_t^Y):0\le s,t\le R\}$. Since the supports of the functions $U_q(\mu_j), \tilde{V}(\mu_j)$ and $\tilde{W}(\mu_j)$ are uniformly bounded, we see that these functions converge in $C_c([0,\infty)^2)$ and $C_c([0,\infty))$ respectively to $U_q(\mu), \tilde{V}(\mu)$ and $\tilde{W}(\mu)$. The claim follows.
	\end{proof}
	\begin{lemma}
		\label{lemma:HelpLemmaClassification}
		Let $\mu\in\VConv^R_k(\C^n)^{U(n)}$ be a given for $\zeta_q\in D^{2n-k}_R$, $\tilde{\zeta}_q\in \tilde{D}_R^{2n-k+2}$ by
		\begin{align*}
			\mu=\sum_{\max(0,k-n)}^{\lfloor\frac{k}{2}\rfloor}T^n_{k,q}(\zeta_q)+\sum_{\max(1,k-n)}^{\lfloor\frac{k-1}{2}\rfloor}Y^n_{k,q}(\tilde{\zeta}_q).
		\end{align*}
		Then
		\begin{align*}
			U_0(\mu)=&\mathcal{R}^{1,2n-k-1}(\zeta_0)\\
			\tilde{V}_q(\mu)=&\begin{cases}
				\mathcal{R}^{2n-k}(\zeta_q)-2q(k-2)\mathcal{P}^{2n-k+2}(\tilde{\zeta}_q)& \max(1,k-n)\le q\le \lfloor\frac{k-1}{2}\rfloor, n\ge 3, k\ge 3,\\
				\mathcal{R}^{2n-k}(\zeta_{\frac{k}{2}}) & q=\frac{k}{2},n\ge 3, k\ge 3,
			\end{cases}\\
			\tilde{W}_q(\mu)=&\begin{cases}
				\mathcal{R}^{2n-k}(\zeta_q)-q(k-4)\mathcal{P}^{2n-k+2}(\tilde{\zeta}_q)& \max(1,k-n)\le q\le \lfloor\frac{k-1}{2}\rfloor, n\ge 3, k\ge 5,\\
				\mathcal{R}^{2n-k}(\zeta_{\frac{k}{2}}) & q=\frac{k}{2}, n\ge 3, k\ge 5.
			\end{cases}
		\end{align*}
		In particular, for $n\ge 3$, $k\ge 5$, $q\ge 2$:
		\begin{align*}
			\mathcal{R}^{2n-k}(\zeta_q)=&\frac{1}{k}\left(2(k-2)\tilde{W}_q(\mu)-(k-4)\tilde{V}_q(\mu)\right) && \max(2,k-n)\le q\le \left\lfloor \frac{k-1}{2}\right\rfloor,\\
			\mathcal{P}^{2n-k+2}(\tilde{\zeta}_q)=&\frac{1}{kq}\left(\tilde{W}_q(\mu)-\tilde{V}_q(\mu)\right) && \max(2,k-n)\le q\le \left\lfloor \frac{k-1}{2}\right\rfloor.
		\end{align*}
	\end{lemma}
	\begin{proof}
			This follows directly from Lemma \ref{lemma:ValuesYOnCyclinderFunctions} in combination with Lemma \ref{lemma:DecompositionCylinderFunctionsVConvKQ}.
	\end{proof}
		We define $V_q,W_q:\VConv^R_k(\C^n)^{\U(n)}\rightarrow C_c([0,\infty])$ for $k\ge 5$ by
		\begin{align*}
			V_q(\mu)=&\frac{1}{k}\left(2(k-2)\tilde{W}_q(\mu)-(k-4)\tilde{V}_q(\mu)\right) && \max(2,k-n)\le q\le \left\lfloor \frac{k-1}{2}\right\rfloor,\\
			W_q(\mu)=&\frac{1}{kq}\left(\tilde{W}_q(\mu)-\tilde{V}_q(\mu)\right) && \max(2,k-n)\le q\le \left\lfloor \frac{k-1}{2}\right\rfloor.
		\end{align*}
		Since $\tilde{V}_q$ and $\tilde{W}_q$ are continuous, so are $V_q$ and $W_q$.

	\begin{corollary}
		\label{corollary:RegularityUVW}
		For $\mu\in \VConv^R_k(\C^n)^{\U(n)}$, 
		\begin{align*}
			U_0(\mu)&\in \mathcal{C}_{(1,2n-k-1),R},\\
			\tilde{V}_{\frac{k}{2}}(\mu),V_q(\mu),W_q(\mu)&\in C_{c,R}([0,\infty)),
		\end{align*} whenever $n,k,q$ are such that the functions are defined.
	\end{corollary}
	\begin{proof}
		It follows from Lemma \ref{lemma:HelpLemmaClassification} and Theorem \ref{theorem:characterization_smooth_unitarily_invariant_valuations} that the claim holds for all smooth valuations in $\VConv_k^R(\C^n)^{\U(n)}$. Since the maps $U_0$, $V_q$ and $W_q$ are continuous and $\mathcal{C}_{(1,2n-k-1),R}\subset C_c([0,\infty)^2)$ and $C_{c,R}([0,\infty))$ are closed, the claim follows from the fact that smooth valuations are dense in $\VConv^R_k(\C^n)^{\U(n)}$ by Corollary \ref{corollary:approximationUnderSupportRestrictionInvariantCase}.
	\end{proof}

	\begin{proof}[Proof of the induction step in Theorem \ref{theorem:DecompositionRefinedVersion}]
		Recall that $n\ge 2$, $2\le k\le 2n-2$. Consider the map 
		\begin{align*}
			A:\VConv^R_k(\C^n)^{\U(n)}\rightarrow\bigoplus_{q=\max(0,k-n),q\ne 1}^{\lfloor\frac{k}{2}\rfloor}\VConv^R_{k,q}(\C^n)^{\U(n)}
		\end{align*} given by
		\begin{align*}
			A(\mu)=&\begin{cases}
				\sum\limits_{q=\max(2,k-n)}^{\lfloor\frac{k-1}{2}\rfloor}T^n_{k,q}(\left(\mathcal{R}^{2n-k}\right)^{-1}\circ V_q(\mu))+Y^n_{k,q}(\left(\mathcal{P}^{2n-k+2}\right)^{-1}\circ W_q(\mu)) & n\ge 5,\\
				0&\text{else},
			\end{cases}\\
					&+\begin{cases}
						T^n_{k,0}(\mathcal{S}^{1,2n-k-1}\circ U_0(\mu)) & k\le n,\\
						0 & \text{else},
					\end{cases}\\
					&+\begin{cases}
						T^n_{k,\frac{k}{2}}\left(\left(\mathcal{R}^{2n-k}\right)^{-1}\circ \tilde{V}_{\frac{k}{2}}(\mu)\right) & k\ge 4~\text{even}, n\ge 3,\\
						0 &\text{else}.
					\end{cases}
		\end{align*}
		Note that this is well defined due to Corollary \ref{corollary:RegularityUVW}. We claim that $\mu-A(\mu)\in \VConv^R_{k,1}(\C^n)^{\U(n)}$. If $\mu$ is a smooth valuation, then this follows from Lemma \ref{lemma:HelpLemmaClassification} in combination with the classification of smooth valuations in Theorem \ref{theorem:characterization_smooth_unitarily_invariant_valuations}. Next, note that $A$ is continuous, as all of the occurring maps are continuous on their respective domains (compare Section \ref{section:IntegralTransforms}, Theorem \ref{theorem:ExistenceT}, Lemma \ref{lemma:DefinitionUQ}, and Corollary \ref{corollary:RegularityUVW}). Since $\VConv^R_{k,1}(\C^n)^{\U(n)}$ is closed and smooth valuations are dense in $\VConv_k^R(\C^n)^{\U(n)}$ by Corollary \ref{corollary:approximationUnderSupportRestrictionInvariantCase}, we obtain $\mu-A(\mu)\in \VConv^R_{k,1}(\C^n)^{\U(n)}$ for every $\mu\in\VConv_k^R(\C^n)^{\U(n)}$ by continuity. In particular,
		\begin{align*}
			\mu=(\mu-A(\mu))+A(\mu)\in\bigoplus_{q=\max(0,k-n)}^{\lfloor\frac{k}{2}\rfloor}\VConv^R_{k,q}(\C^n)^{\U(n)}.
		\end{align*}
		This completes the proof of Theorem \ref{theorem:DecompositionRefinedVersion}.
	\end{proof}

	\begin{remark}
		\label{remark:ProblematicComponent}
		Note that the proof of Theorem \ref{theorem:DecompositionRefinedVersion} does not give a direct formula for the projection onto $\VConv_{k,1}(\C^n)^{\U(n)}$ in contrast to the projections onto all of the other components. The reason for this different treatment is that the reconstruction breaks down for $k=2$ if $n=2$ or $n=3$. In both of these cases we essentially have only one way to split $\C^2$ or $\C^3$ into a proper complex orthogonal decomposition $X\oplus Y$ - one of these spaces has to be $1$-dimensional. For $n=2$, Lemma \ref{lemma:DecompositionCylinderFunctionsVConvKQ} simplifies for $\mu\in\VConv_{2,1}(\C^2)^{\U(2)}$ to
		\begin{align*}
			\bar{\mu}(f_X[1],f_Y[1])=&0,
		\end{align*}
		and we cannot obtain any information from this decomposition. One may also consider the terms $\bar{\mu}(f_X[2],f_Y[0])$ and $\bar{\mu}(f_X[0],f_Y[2])$, which corresponds to the restriction of $\mu$ to $X$, $Y$ respectively. In these cases, the densities of these valuations are obtained from the densities of $\mu$ by applying the Abel transform twice. While this transform is injective for our classes of densities, it only has dense image. In particular, the inverse is only defined on a dense subspace and is in particular not continuous, which prevents a direct argument by approximation. A similar problem occurs for $n=3$. For $n\ge 4$, we have at least two possible proper orthogonal decompositions $\C^n=X\oplus Y$, which can be used to construct the projection explicitly.
	\end{remark}

	\subsection{Characterization result for $\VConv_{k,q}(\C^n)^{\U(n)}$}
	Fix $1\le k\le 2n-1$, $\max(0,k-n)\le q\le \left\lfloor\frac{k}{2}\right\rfloor$. We will obtain the characterization of valuations in $\VConv_{k,q}(\C^n)^{\U(n)}$ in Theorem \ref{maintheorem:RepresentationVConvKQ} by approximation from the smooth case. The key observation is that Theorem \ref{theorem:DecompositionRefinedVersion} implies the following density result.
	\begin{corollary}
		\label{corollary:SequentialDensityVConv_kq}
		$\VConv^R_{k,q}(\C^n)^{\U(n)}\cap \VConv(\C^n)^{sm}$ is dense in $\VConv^R_{k,q}(\C^n)^{\U(n)}$.
	\end{corollary}
	\begin{proof}
	By Corollary \ref{corollary:approximationUnderSupportRestrictionInvariantCase}, smooth valuations are dense in $\VConv^R_k(\C^n)^{\U(n)}$. Since the projection $P^n_{k,q}:\VConv^R_{k}(\C^n)^{\U(n)}\rightarrow\VConv_{k,q}^R(\C^n)^{\U(n)}$ is continuous by Corollary \ref{corollary:lowerdimensional_decomposition_topological_direct_sum}, the image of $\VConv^R_k(\C^n)^{\U(n)}\cap \VConv(\C^n)^{sm}$ under this map is dense. However, the components of any smooth valuations with respect to the decomposition are again smooth by Theorem \ref{theorem:characterization_smooth_unitarily_invariant_valuations}. Thus smooth valuations are dense in $\VConv^R_{k,q}(\C^n)^{\U(n)}$.
	\end{proof}
	We now split $\C^n=\C\oplus \C^{n-1}=:X\oplus Y$. Consider the functions $U,V:\VConv^R_{k,q}(\C^n)^{\U(n)}\rightarrow C_{c}([0,\infty))$ defined by
	\begin{align*}
		U(\mu)[t]=&\frac{1}{\omega_{2n}\binom{n}{k-2q,q}2^{k-2q}}\mu(u_t) && 1\le k\le 2n-1,\\	V(\mu)[t]=&\frac{qn}{2\omega_2(2n-2)\omega_{2n-2}\binom{n-1}{k-2q,q-1}2^{k-2q}}\overline{\mu}(u^X_0[2],u^Y_t[k-2]) && 3\le k\le 2n-1.
	\end{align*}
	As in the proof of Lemma \ref{lemma:DefinitionUQ} one easily checks that these maps are well defined and continuous. We now prove the following refined version of Theorem \ref{maintheorem:RepresentationVConvKQ}.	
	\begin{theorem}
		\label{theorem:RepresentationVConvKQWithSupport}
		Let $1\le k\le 2n-1$, $\max(0, k-n)\le q\le \lfloor\frac{k}{2}\rfloor$, $R>0$.
		\begin{itemize}
			\item If $\max(1,k-n)\le q\le \lfloor\frac{k-1}{2}\rfloor$ and  $\mu\in\VConv^R_{k,q}(\C^n)^{\U(n)}$, 
			then the functions 
			\begin{align*}
				\zeta=&(\mathcal{R}^{2n-k})^{-1}(U(\mu))\\
				\tilde{\zeta}=&\left(\mathcal{P}^{2n-k+2}\right)^{-1} \left(V(\mu-T^n_{k,q}(\zeta))\right)
			\end{align*}
			satisfy $\zeta\in D_R^{2n-k}$, $\tilde{\zeta}\in \tilde{D}_R^{2n-k+2}$, and
			\begin{align*}
				\mu=T^{n}_{k,q}(\zeta)+Y^{n}_{k,q}(\tilde{\zeta}).
			\end{align*}
			\item If $k\le n$ and $q=0$, or $k$ is even and $q=\frac{k}{2}$, and  $\mu\in\VConv^R_{k,q}(\C^n)^{\U(n)}$, then the function
			\begin{align*}
				\zeta=&(\mathcal{R}^{2n-k})^{-1}(U(\mu))
			\end{align*}
			satisfies $\zeta\in D_R^{2n-k}$, and
			\begin{align*}
				\mu=T^{n}_{k,q}(\zeta).
			\end{align*}
			
		\end{itemize} 
		The representation is unique in both cases.
	\end{theorem}
	\begin{proof}
		In the second case, the uniqueness follows directly from Lemma \ref{lemma:BehaviorT_u_t} and Lemma \ref{lemma:InjectivityTransformR}, while it follows from  Corollary \ref{corollary:InjectivityTY} in the first. It is thus sufficient to show that such a representation exists. We will show the first case $\max(1,k-n)\le q\le \lfloor\frac{k-1}{2}\rfloor$, the second case is similar but simpler. Note that the first case only occurs for $n\ge 2$, $k\ge3$.\\
		
		The characterization of the support in Proposition \ref{proposition:characterizationSupport} implies that $U(\mu)\in C_{c,R}([0,\infty))$ for all $\mu\in\VConv_{k,q}^R(\C^n)^{\U(n)}$, so 
		\begin{align*}
			\zeta= (\mathcal{R}^{2n-k})^{-1}\circ U(\mu)\in D^{2n-k}_R.
		\end{align*}
		Consequently, the map $\VConv^R_{k,q}(\C^n)^{\U(n)}\rightarrow C_c([0,\infty))$ given by
		\begin{align*}
			\mu\mapsto V(\mu-T^n_{k,q}\circ (\mathcal{R}^{2n-k})^{-1}\circ U(\mu))
		\end{align*}
		is well defined and continuous. If $\mu$ is a smooth valuation, it is easy to see that this implies $V(\mu-T^n_{k,q}\circ U(\mu))\in C_{0,R}$. Since this map is continuous, $C_{0,R}\subset C_c([0,\infty))$ is closed, and smooth valuations are dense in $\VConv_{k,q}(\C^n)^{\U(n)}$ by Corollary \ref{corollary:SequentialDensityVConv_kq}, this holds in fact for all $\mu\in\VConv^R_{k,q}(\C^n)^{\U(n)}$. In particular, Lemma \ref{lemma:PaIsomorphism} shows that
		\begin{align*}
			\tilde{\zeta}=&\left(\mathcal{P}^{2n-k+2}\right)^{-1} \left(V(\mu-T^n_{k,q}(\zeta))\right)\in \tilde{D}^{2n-k+2}_R \quad\text{for all}~\mu\in\VConv^R_{k,q}(\C^n)^{\U(n)},
		\end{align*}
		and the map $\VConv^R_{k,q}(\C^n)^{\U(n)}\rightarrow\VConv^R_{k,q}(\C^n)^{\U(n)}$ given by
		\begin{align*}
			\mu\mapsto \mu-T^n_{k,q}\circ(\mathcal{R}^{2n-k})^{-1}\circ U(\mu)-Y^n_{k,q}\circ \left(\mathcal{P}^{2n-k+2}\right) ^{-1}\left(V(\mu-T^n_{k,q}\circ(\mathcal{R}^{2n-k})^{-1}\circ U(\mu))\right)
		\end{align*}
		is well defined and continuous. By construction, it vanishes on the dense subspace of smooth valuations. Thus it vanishes identically, so 
		\begin{align*}
			\mu=T^n_{k,q}\circ(\mathcal{R}^{2n-k})^{-1}\circ U(\mu)+Y^n_{k,q}\circ \left(\mathcal{P}^{2n-k+2}\right) ^{-1}\left(V(\mu-T^n_{k,q}\circ U(\mu))\right)
		\end{align*}
		for all $\mu\in\VConv^R_{k,q}(\C^n)^{\U(n)}$. The claim follows.		
	\end{proof}
	Let us give a slight reformulation of the classification result. For $\max(0,k-n)\le q\le \lfloor \frac{k}{2}\rfloor$ set 
	\begin{align*}
		D^{n,k,q}_R:=\begin{cases}
			D^{2n-k}_R & q=0~\text{or}~ q=\frac{k}{2},\\
			D^{2n-k}_R\oplus \tilde{D}^{2n-k+2} & \text{else}
		\end{cases}
	\end{align*}
	and consider the maps $H^n_{k,q}:D^{n,k,q}_R\rightarrow\VConv_{k,q}^R(\C^n)^{\U(n)}$ defined by
	\begin{align*}
		H^n_{k,q}(\zeta)=&T^n_{k,q}(\zeta)&&\text{for}~q=0~\text{or}~ q=\frac{k}{2},\\
		H^n_{k,q}(\zeta,\tilde{\zeta})=&T^n_{k,q}(\zeta)+Y^n_{k,q}(\tilde{\zeta})&&\text{else}.
	\end{align*}
	Note that $D^{n,k,q}_R$ is a Banach space by Lemma \ref{lemma:tildeDcomplete} and \cite[Lemma 2.7]{KnoerrSingularvaluationsHadwiger2022}. Let
	\begin{align*}
		H^n_k:\bigoplus_{q=\max(0,k-n)}^{\left\lfloor\frac{k}{2}\right\rfloor}D^{n,k,q}_R\rightarrow\VConv_{k}^R(\C^n)^{\U(n)}
	\end{align*}
	denote the sum of these maps. The following result shows that the densities are jointly continuously as a function of the given valuation.
	\begin{corollary}
		\begin{align*}
			H^n_{k}:\bigoplus_{q=\max(0,k-n)}^{\left\lfloor\frac{k}{2}\right\rfloor}D^{n,k,q}_R\rightarrow\VConv_{k}^R(\C^n)^{\U(n)}
		\end{align*} is a topological isomorphism.
	\end{corollary}
	\begin{proof}
		Since it is the sum of continuous maps, this map is a continuous linear map between Banach spaces. Since it is bijective by Theorem \ref{theorem:DecompositionRefinedVersion} and Theorem \ref{theorem:RepresentationVConvKQWithSupport}, its inverse is continuous by the open mapping theorem. The claim follows.
	\end{proof}

	\bibliography{../../library/library.bib}

\Addresses
\end{document}